\documentclass[10pt, reqno]{amsart}
\usepackage{amsaddr}
\usepackage{latexsym,amsfonts}
\usepackage{amssymb,amsthm,upref,amscd}
\usepackage[T1]{fontenc}
\usepackage{times}
\usepackage{amscd}
\usepackage{enumerate}
\usepackage{mathrsfs}
\usepackage{amsmath}
\usepackage{color}
\usepackage{soul}
\usepackage{indentfirst}
\usepackage{comment}
\PassOptionsToPackage{reqno}{amsmath}

\newtheorem{thm}{Theorem}[section]
\newtheorem{prop}{Proposition}[section]
\newtheorem{defi}{Definition}[section]
\newtheorem{lem}{Lemma}[section]
\newtheorem{cor}{Corollary}[section]

\newtheorem{rem}{Remark}[section]
\theoremstyle{notation}

\newcommand{\R}{\mathbb{R}}

\numberwithin{equation}{section}
\newcommand{\N}{\mathbb{N}}

\newcommand{\eps}{\epsilon}

\newcommand{\wto}{\rightharpoonup}

\makeatletter \@addtoreset{equation}{section} \makeatother

\usepackage[textwidth=138mm,
textheight=215mm,
left=30mm,
right=30mm,
top=25.4mm,
bottom=25.4mm,
headheight=2.17cm,
headsep=4mm,
footskip=12mm,
heightrounded,]{geometry}

\usepackage[colorlinks=true,urlcolor=blue,
citecolor=black,linkcolor=black,linktocpage,pdfpagelabels,
bookmarksnumbered,bookmarksopen]{hyperref}
\newcounter{const}
\author[T. Gou]{Tianxiang Gou}

\address[T. Gou]{%
\centerline{School of Mathematics and Statistics, Xi’an Jiaotong University,}
\centerline{710049, Xi’an, Shaanxi, China}}



\email{tianxiang.gou@xjtu.edu.cn}
\title[Inhomogeneous NLS]{Standing waves with prescribed $L^2$-norm to nonlinear Schr\"odinger equations with combined inhomogeneous nonlinearities}
\thanks{{\it Statements and Declarations}. The author declares that there are no conflict of interests.}
\thanks{{\it Data Availability Statement}. Data sharing not applicable to this article as no datasets were generated or analysed during the current study.}
\thanks{{\it Acknowledgments}. The work was supported by the National Natural Science Foundation of China (No. 12101483) and the Postdoctoral Science Foundation of China (No.2021M702620). The author thanks Prof. Vladimir Georgiev for helpful discussions. The author also thanks Dr. Alex H. Ardila for providing him with the references \cite{Luo, MZZ}.}

\begin{document}

\begin{abstract} 
In this paper, we are concerned with solutions to the following nonlinear Schr\"odinger equation with combined inhomogeneous nonlinearities,
$$ 
-\Delta u + \lambda u= \mu |x|^{-b}|u|^{q-2} u + |x|^{-b}|u|^{p-2} u \quad \mbox{in} \,\, \R^N,
$$
under the $L^2$-norm constraint
$$
\int_{\R^N} |u|^2 \, dx=c>0,
$$
where $N \geq 1$, $\mu =\pm 1$, $2<q<p<{2(N-b)}/{(N-2)^+}$, $0<b<\min\{2, N\}$ and the parameter $\lambda \in \R$ appearing as Lagrange multiplier is unknown. In the mass subcritical case, we establish the compactness of any minimizing sequence to the minimization problem given by the underlying energy functional restricted on the constraint. As a consequence of the compactness of any minimizing sequence, orbital stability of minimizers is derived. In the mass critical and supercritical cases, we investigate the existence, radial symmetry and orbital instability of solutions. Meanwhile, we consider the existence, radial symmetry and algebraical decay of ground states to the corresponding zero mass equation with defocusing perturbation. In addition, dynamical behaviors of solutions to the Cauchy problem for the associated dispersive equation are discussed. 
\end{abstract}

\maketitle


\section{Introduction}

In this paper, we consider the existence and dynamical behaviors of solutions to the following nonlinear Schr\"odinger equation with combined inhomogeneous nonlinearities,
\begin{align} \label{equ}
-\Delta u + \lambda u= \mu |x|^{-b}|u|^{q-2} u + |x|^{-b}|u|^{p-2} u \quad \mbox{in} \,\, \R^N,
\end{align}
under the $L^2$-norm constraint
\begin{align} \label{mass}
\int_{\R^N} |u|^2 \, dx=c>0,
\end{align}
where $N \geq 1$, $\mu =\pm 1$, $2<q<p<{2(N-b)}/{(N-2)^+}$, $0<b<\min\{2, N\}$ and the parameter $\lambda \in \R$ appearing as Lagrange multiplier is unknown. The problem under consideration arises in the study of standing waves to the following time-dependent nonlinear Schr\"odinger equation,
\begin{align} \label{equt}
\textnormal{i} \partial_t \psi +\Delta \psi +\mu |x|^{-b}|\psi|^{q-2} \psi + |x|^{-b}|\psi|^{p-2} \psi=0 \quad \mbox{in} \,\, \R\times \R^N.
\end{align}
{Equation \eqref{equt} arises in various physical contexts, for example in nonlinear optics as well as in the description of nonlinear waves such as the propagation of laser beams, water waves at the free surface of an ideal fluid and plasma waves, we refer the readers to \cite{BPVT, KA} and references therein for further interpretations.} Here a standing wave to \eqref{equt} is a solution of the form
$$
\psi(t, x)=e^{\textnormal{i} \lambda t} u(x), \quad \lambda \in \R.
$$
It is clear that a standing wave $\psi$ solves \eqref{equt} if and only if $u$ solves \eqref{equ}.

In view of Lemma \ref{localwp}, we know that the mass of any solution to the Cauchy problem for \eqref{equt} is conserved along time, i.e. for any $t \in [0, T)$,
$$
\int_{\R^N} |\psi(t, \cdot)|^2 \,dx=\int_{\R^N} |\psi(0, \cdot)|^2 \,dx.
$$
In physics, the mass has important significance, which is often used to represent the power supply in nonlinear optics or the total number of atoms in Bose-Einstein condensation, see for example \cite{EGBB, F, M}. As a consequence, from physical point of views, it is interesting to consider solutions to \eqref{equt} with prescribed mass. This naturally leads to the study of solutions to \eqref{equ}-\eqref{mass} for $c>0$ given. In this scenario, solutions to \eqref{equ}-\eqref{mass} are referred to as normalized solutions and the parameter $\lambda \in \R$ appearing as Lagrange multiplier related to the constraint is unknown. Indeed, solutions to \eqref{equ}-\eqref{mass} correspond to critical points of the underlying energy functional $E$ restricted on $S(c)$, where
$$
E(u):=\frac 12 \int_{\R^N} |\nabla u|^2 \,dx -\frac{\mu}{q} \int_{\R^N} |x|^{-b}|u|^q \,dx -\frac 1 p \int_{\R^N} |x|^{-b} |u|^p \, dx
$$
and
$$
S(c):=\left\{ u\in H^1(\R^N) : \int_{\R^N} |u|^2 \, dx =c\right\}.
$$
Moreover, from mathematical perspectives, the study of normalized solutions turns out to be also meaningful, because it paves the way to better understand dynamical properties of stationary solutions to \eqref{equt}. For these reasons, we shall focus on the survey of solutions to \eqref{equ}-\eqref{mass} in this paper.

When $\mu=0$, following the seminar works of Kening and Merle \cite{KM1, KM2, Me}, dynamical behaviors of solutions to the Cauchy problem for \eqref{equt} have been extensively considered during recent years, see for instance \cite{AC, C1, CC1, CC2, DBF, DK1, DK2, D1, D2, D3, F, FG1, FG2, GS, Luo, LWW, RS} and references therein. When $\mu \neq 0$, the problem is new and the study of that is open up to now. In this situation, \eqref{equt} is no longer invariant under the following scaling,
$$
\psi_{\tau}(t, x):=\tau^{\frac{2-b}{p-2}} \psi(\tau^2 t, \tau x), \quad \tau>0.
$$ 
In comparison with the unperturbed case $\mu=0$, the presence of the lower order term $\mu |\psi|^{p-2} \psi$ will play an important role in the forthcoming discussion, which strongly impacts the geometry structure of the energy functional $E$ restricted on $S(c)$ and brings about diverse phenomena. Indeed, the study carried out in the current paper is somehow reminiscent of the one for the problem with combined power-type nonlinearities, see \cite{AIKN, CMZ, FH, FO, GZ, LMR, MXZ, MZZ, TVZ, X} concerning dynamical behaviors of solutions to the Cauchy problem for \eqref{equt} with $b=0$ and \cite{BFG, JL, S1, S2} concerning the existence and properties of ground states to \eqref{equ}-\eqref{mass} with $b=0$. 
Although, when $b>0$, then there exists the compact embedding (see Lemma \ref{cembedding}). This actually helps to derive the compactness of Palais-Smale sequences and the existence of solutions. However, the study of symmetry and decay of solutions becomes complex, which requires different ingredients and careful treatments, because of the presence of the term $|x|^{-b}$ for $b>0$. It is worth mentioning that the parameter $b>0$ has significant and nontrivial effects in the whole discussions. 

For the case when the energy functional restricted on the $L^2$-norm constraint is bounded from below, the study of normalized solutions is often carried out by introducing a global minimization problem. In this case, the existence and orbital stability of the solutions can be derived by using well-known Lions concentration compactness principle in \cite{Li1, Li2}, see for example \cite{AB, CCW, CDSS, CP, G, Gou, GJ2, NW1, NW2, NW3} and references therein. For the case when the energy functional restricted on the $L^2$-norm constraint is unbounded from below, then it is impossible to introduce a global minimization problem to investigate normalized solutions, which enables the study of normalized solutions become hard. In this case, the solutions often correspond to local minimizers or saddle type critical points of the energy functional restricted on the constraint. Starting from the early work due to Jeanjean \cite{Je}, the study of the solutions in this direction has received much attention during recent years, see for example \cite{BMRV, BJS, BS1, BS2, BV, BZZ, BJ, BJT, BCGJ, CJ, GJ1, GZ, JS, NTV2, PPVV, PG}.

\subsection*{Statement of main results} To address our main results, we first introduce a scaling of $u \in S(c)$ as $u_t(x)=t^{N/2}u(tx)$ for $x \in \R^N$ and $t>0$. Such a scaling will be adopted frequently throughout the discussion of the paper. By direct calculations, it is not difficult to get that $\|u_t\|_2=\|u\|_2$ and 
\begin{align} \label{scaling1}
E(u_t)=\frac{t^2}{2} \int_{\R^N} |\nabla u|^2 \,dx-\frac{\mu t^{\frac{N}{2}(q-2)+b}}{q} \int_{\R^N}|x|^{-b}|u|^q \, dx -\frac{t^{\frac{N}{2}(p-2)+b}}{p} \int_{\R^N}|x|^{-b}|u|^p \, dx.
\end{align}
{In the sense of Gagliardo-Nirenberg's inequality \eqref{GN}, $p=2+2(2-b)/N$ is referred to as the mass critical case, $p<2+2(2-b)/N$ and $p>2+2(2-b)/N$ are referred to as the mass subcritical and supercritical cases, respectively.} 

First, we consider the mass subcritical case $2<q<p<2+2(2-b)/N$. In this case, applying the well-known Gagliardo-Nirenberg's inequality \eqref{GN}, we find that $E$ restricted on $S(c)$ is bounded from below for any $c>0$. Then we are able to introduce the following minimization problem, for any $c>0$,
\begin{align} \label{gmin}
m(c):={\min_{u \in S(c)}} E(u).
\end{align}

As a direct consequence of \eqref{scaling1} and Gagliardo-Nirenberg’s inequality \eqref{GN}, 
we first obtain some properties of the function $c \mapsto m(c)$ for $c>0$.

\begin{prop} \label{prop11}
Let ${N \geq 1}$, $2<q<p<2+2(2-b)/N$, $0<b<\min\{2, N\}$ and $\mu=\pm 1$, then the following assertions hold true.
\begin{itemize}
\item[$(\textnormal{i})$] $-\infty<m(c) \leq 0$ for any $c>0$ and $m(c_1+c_2) \leq m(c_1)+m(c_2)$ for any $c_1, c_2 \geq 0$. In particular, the function $c \mapsto m(c)$ is nonincreasing on $(0, \infty)$.
\item[$(\textnormal{ii})$] The function $c \mapsto m(c)$ is continuous for any $c>0$.
\item[$(\textnormal{iii})$] $m(c) \to 0$ as $c \to 0^+$ and $m(c) \to -\infty$ as $c \to \infty$.
\end{itemize} 
\end{prop}

In the mass subcritical case, one of interesting topics is to discuss the compactness of any minimizing sequence to \eqref{gmin}, which lays a foundation for the forthcoming study of orbital stability of minimizers to \eqref{gmin}.

\begin{thm} \label{thm1}
Let ${N \geq 1}$, $2<q<p<2+2(2-b)/N$ and $0<b<\min\{2, N\}$.
\begin{itemize}
\item[$(\textnormal{i})$] If $\mu =1$, then $m(c)<0$ for any $c>0$. Moreover, any minimizing sequence to \eqref{gmin} is compact in $H^1(\R^N)$ up to translations for any $c>0$.
\item[$(\textnormal{ii})$]  If $\mu=-1$, then there exist two constants $\hat{c}_0>\tilde{c}_0>0$ depending only on $N,p,q$ and $b$ such that $m(c)=0$ for any $0<c<\tilde{c}_0$ and $m(c) <0$ for any $c>\hat{c}_0$. Moreover, any minimizing sequence to \eqref{gmin} is compact in $H^1(\R^N)$ up to translations for any $c>\hat{c}_0$.
\end{itemize}
\end{thm}

Second, we study the mass critical case $2<q<p=2+2(2-b)/N$. In this case, our first result reads as follows.

\begin{thm} \label{thm2}
Let ${N \geq 1}$, $2<q<p=2+2(2-b)/N$ and $0<b<\min\{2, N\}$.
\begin{itemize}
\item[$(\textnormal{i})$] If $\mu=1$, then $m(c)<0$ for any $0<c<c_1$ and $m(c)=-\infty$ for any $c \geq c_1$. Moreover, any minimizing sequence to \eqref{gmin} is compact in $H^1(\R^N)$ up to translations for any $0<c<c_1$, where $c_1>0$ is defined by
$$
c_1:=\left(\frac{N+2-b}{NC_{N,b}}\right)^{\frac{N}{2-b}}
$$
and $C_{N, b}>0$ is the optimal constant in \eqref{GN} with $p=2+2(2-b)/N$.
\item[$(\textnormal{ii})$] If $\mu=-1$, then $m(c)=0$ for any $0<c \leq c_1$ and $m(c)=-\infty$ for any $c>c_1$. Moreover, there exists no solutions to \eqref{equ}-\eqref{mass} for any $0<c \leq c_1$.
\end{itemize}
\end{thm}

In Theorems \ref{thm1} and \ref{thm2}, the compactness of any minimizing sequence to \eqref{gmin} is achieved by adapting the Lions concentration compactness principle in \cite{Li1, Li2}.


Third, we investigate the mass supercritical case $p>2+2(2-b)/N$. In this case, applying \eqref{scaling1}, we can easily conclude that $E$ restricted on $S(c)$ becomes unbounded from below. Indeed, by \eqref{scaling1}, there holds that $E(u_t) \to -\infty$ as $t \to \infty$ for any $u \in S(c)$ and $c>0$. As a result, it is unlikely to make use of \eqref{gmin} to establish the existence of solutions to \eqref{equ}-\eqref{mass}. For this reason, we introduce another minimization problem with the help of so-called Pohozaev manifold $P(c)$ defined by
$$
P(c):=\{u \in S(c) : Q(u)=0\},
$$
where
\begin{align*}
Q(u):&=\frac{d}{dt}E(u_t)\mid_{t=1}\\
&=\int_{\R^N} |\nabla u|^2 \,dx-\frac{\mu\left(N(q-2)+2b\right)}{2q}\int_{\R^N}|x|^{-b}|u|^q \, dx-\frac{N(p-2)+2b}{2p}\int_{\R^N}|x|^{-b}|u|^p \, dx.
\end{align*}
Here $Q(u)=0$ is the Pohozaev identity related to \eqref{equ}-\eqref{mass}, see Lemma \ref{ph}.
\medskip

In the mass supercritical case, our first aim is to investigate the existence of solutions to \eqref{equ}-\eqref{mass} with focusing perturbation $\mu=1$.

\begin{thm} \label{thm3}
Let ${N \geq 1}$, $2<q<2+2(2-b)/N<p$, $0<b<\min\{2, N\}$ and $\mu=1$, then there exists a constant $c_2>0$ such that \eqref{equ}-\eqref{mass} admits two positive, radially symmetric and decreasing solutions for any $0<c<c_2$, one of which is an interior local minimizer with negative energy and the other is a mountain pass type solution with positive energy, where $c_2>0$ is defined by
\begin{align*}
c_2:= \left(\frac{2p\left(2(2-b)-N(q-2)\right)}{C_{N,p, b}N\left(N(p-2)+2b\right)(p-q)}\right)^{\frac{2(2-b)-N(q-2)}{(p-q)(2-b)}}\left(\frac{q \left(N(p-2)-2(2-b)\right)}{2C_{N, q, b} N(p-q)} \right)^{\frac{N(p-2)-2(b-2)}{(p-q)(2-b)}}
\end{align*}
and $C_{N,p,b}>0$ is the optimal constant in \eqref{GN}.
\end{thm}

In order to establish Theorem \ref{thm3}, we shall decompose the manifold $P(c)$ into disjoint union, i.e. $P(c)=P_+(c) \cup P_0(c) \cup P_-(c)$, where
$$
P_+(c):=\{u \in P(c) :  \Psi(u)>0\},
$$ 
$$
P_0(c):=\{u \in P(c) :  \Psi(u)=0\},
$$ 
$$
P_-(c):=\{u \in P(c) :  \Psi(u)<0\}
$$
and 
\begin{align*}
\Psi(u):=\frac{d^2}{dt^2} E(u_t) \mid_{t=1}&=\int_{\R^N}|\nabla u|^2 \,dx-\frac{\left(N(q-2)+2b\right)\left(N(q-2)+2(b-1)\right)}{4q} \int_{\R^N}|x|^{-b}|u|^q \, dx\\
&\quad -\frac{\left(N(p-2)+2b\right)\left(N(p-2)+2(b-1)\right)} {4p}\int_{\R^N}|x|^{-b}|u|^p \, dx.
\end{align*}
To seek for the first solution, we introduce a local minimization problem for $0<c<c_2$, which is indeed defined by
\begin{align} \label{localmin1} 
M(c):=\inf_{u \in V_{\rho}(c)} E(u),
\end{align}
where $V_{\rho}(c):=\{u \in S(c) : \|\nabla u\| \leq \rho\}$ for some proper $\rho>0$. Then, by using the Lions concentration compactness principle, we can derive the compactness of any minimizing sequence to \eqref{localmin1}. This gives rise to existence of the solution. To detect the existence of the second solution, the key argument is to prove that there exists a Palais-Smale sequence belonging to $P_-(c)$ for $E$ restricted on $S(c)$ at the level $\sigma_-(c)$ for $0<c<c_2$, where
\begin{align*}
\sigma_-(c):=\inf_{u\in P_-(c)}E(u).
\end{align*}
Then, by inferring that $E$ restricted on $P(c)$ is coercive and the associated Lagrange multiplier is positive, we are able to obtain the existence of the solution. To further reveal the radially symmetric and decreasing property of two solutions, we shall take advantage of Proposition \ref{radial}, whose proof relies essentially on the moving plane method. Proposition \ref{radial} states that any positive solution to \eqref{equ} for $\lambda>0$ and $\mu=\pm 1$ is radially symmetric and decreasing in the radial direction.
\begin{rem}
Under the assumptions of Theorem \ref{thm3}, in view of Proposition \ref{prop1}, then there holds that, for any $0<c<c_2$,
$$
M(c)=\inf_{u \in P(c)} E(u).
$$
This shows that ${M(c)}$ is the ground state energy.
\end{rem}

Furthermore, we establish the following results.

\begin{thm} \label{thm4}
Let ${N \geq 1}$, $2<q=2+2(2-b)/N<p$, $0<b<\min\{2, N\}$ and $\mu=1$, then \eqref{equ}-\eqref{mass} has a positive, radially symmetric and decreasing ground state for any $0<c<c_1$, where $c_1>0$ is the constant given in Theorem \ref{thm2}.
\end{thm}

To prove Theorem \ref{thm4}, we shall introduce the following minimization problem, for any $0<c<c_1$,
\begin{align}\label{min31}
\sigma(c):=\inf_{u \in P(c)} E(u).
\end{align}
Then, arguing as the proof of the existence of the second solution to \eqref{equ}-\eqref{mass} under the assumptions of Theorem \ref{thm3}, we are able to establish the existence result. The radial symmetry of the solution is a direct consequence of Proposition \ref{radial}.


\begin{thm} \label{thm41}
Let ${N \geq 1}$, $2<2+2(2-b)/N<q<p$, $0<b<\min\{2, N\}$ and $\mu=1$, then \eqref{equ}-\eqref{mass} has a positive, radially symmetric and decreasing ground state for any $c>0$.
\end{thm}

\begin{rem}
In fact, verifying the radially symmetric and decreasing property of ground states to \eqref{equ}-\eqref{mass} with focusing perturbation, one can alternatively take advantage of the symmetric-decreasing rearrangement arguments in \cite{LL} and \cite[Theorem 1.1]{BZ}. However, such arguments are not applicable for the case with defocusing perturbation. In this case, Proposition \ref{radial} comes into play.
\end{rem}

Regarding the existence of solutions to \eqref{equ}-\eqref{mass} with defocusing perturbation $\mu=-1$, our main results read as follows.

\begin{thm} \label{thm5} 
Let ${N \geq 1}$, $2<q<p$, $p>2+2(2-b)/N$, $0<b<\min\{2, N\}$ and $\mu=-1$, then there exists a constant $c_3>0$ such that \eqref{equ}-\eqref{mass} has a positive, radially symmetric and decreasing ground state for any $0<c<c_3$, , where $c_3>0$ is defined by
$$
c_3:=\left(\frac{q\left(N(p-2)-2(2-b)\right)}{2C_{N,q,b} \left(N(p-q)(N-b)\right)}\right)^{\frac{N(p-2)-2(2-b)}{(p-q)(2-b)}}\left(\frac{2p}{C_{N,p,b}\left(N(p-2)+2b\right)}\right)^{\frac{2(2-b)-N(q-2)}{(p-q)(2-b)}}.
$$
\end{thm}

\begin{rem}
The smallness restriction on the mass $c$ appearing in Theorem \ref{thm5} is only use to guarantee that the associated Lagrange multiplier is positive, see Lemma \ref{la1}. This is distinctive with the case for $\mu=1$, where the sign of Lagrange multiplier is positive for any $c>0$, see Lemma \ref{la}.
\end{rem}

In the mass critical case with $\mu=-1$, from the assertion $(\textnormal{ii})$ of Theorem \ref{thm2}, we realize that there exists no solutions to \eqref{equ}-\eqref{mass} for any $0<c \leq c_1$. One may wander if there exist solutions to \eqref{equ}-\eqref{mass} for $c>c_1$. In fact, there do exist solutions to \eqref{equ}-\eqref{mass} for some $c>c_1$.

\begin{thm} \label{thm6}
Let ${N \geq 1}$, $2<q<p=2+2(2-b)/N$, $0<b<\min\{2, N\}$ and $\mu=-1$, then there exists a constant $c_1^*>c_1>0$ such that \eqref{equ}-\eqref{mass} has a positive, radially symmetric and decreasing ground state for any $c_1<c<c_1^*$, where $c_1^*>0$ is defined by
$$
c_1^*:=\left(\frac{p\left(2(q-b)-N(q-2)\right)}{2C_{N,b}(p-q)(N-b)}\right)^{\frac{N}{2-b}}.
$$
\end{thm}

The proofs of Theorems \ref{thm5} and \ref{thm6} can be completed by taking advantage of the similar spirit as the one of Theorem \ref{thm4} but with nontrivial calculations.

Next we investigate some properties of the function $c \mapsto \sigma(c)$ for $c>0$. In this respect, our main results are addressed as follows. Here we first define a constant $c_0>0$ by 
\begin{align*}
c_0:=\left\{
\begin{aligned}
0, &\quad \mbox{if} \,\, q>2+2(2-b)/N \,\, \mbox{and}\,\, \mu=1\,\, \mbox{or}\,\, p>2+2(2-b)/N\,\, \mbox{and}\,\,\mu=-1, \\ 
c_1, &\quad \mbox{if} \,\, p=2+2(2-b)/N \,\, \mbox{and}\,\, \mu=-1.
\end{aligned}
\right.
\end{align*}

\begin{prop} \label{prop2}
Let ${N \geq 1}$, $2<q<p<{2(N-b)}/{(N-2)^+}$ and $0<b<\min\{2, N\}$.
\begin{itemize}
\item [$(\textnormal{i})$] If $q>2+2(2-b)/N$ and $\mu=1$ or $p \geq 2+2(2-b)/N$ and $\mu=-1$, then the function $c \mapsto \sigma(c)$ is nonincreasing on $(c_0, \infty)$. Moreover, the function $c \mapsto \sigma(c)$ is continuous for any $c>c_0$. 
\item [$(\textnormal{ii})$] If $q>2+2(2-b)/N$ and $\mu=1$ or $p \geq 2+2(2-b)/N$ and $\mu=-1$, then $\sigma(c) \to \infty$ as $c \to (c_0)^+$.
\item [$(\textnormal{iii})$]  If $q>2+2(2-b)/N$ and $\mu=1$, then $\sigma(c) \to 0$ as $c \to \infty$.
\item [$(\textnormal{iv})$] If $p>2+2(2-b)/N$ and $\mu=-1$, then there exists a constant $\sigma_0>0$ such that $\sigma(c) \to \sigma_0$ as $c \to \infty$. In addition, if $q \geq 2+2(2-b)/N$ and $ N=3$ or $q>2+2(2-b)/N$ and $N=4$, then $\sigma(c) > \sigma_0$ for any $c>0$. In addition, if $2<q<2+2(2-b)/N$ and $ N=3$ or $2<q \leq 2+2(2-b)/N$ and $N=4$ or $q>2$ and $N \geq 5$, then there exist a constant $c_{\infty}>0$ such that $\sigma(c) =\sigma_0$ for any $c \geq c_{\infty}$, {where $\sigma_0>0$ is the ground state energy of solutions to \eqref{zequ00}.}
\item [$(\textnormal{v})$] If $p=2+2(2-b)/N$, $\mu=-1$ and $N \geq 3$, then there exists a constant $c_{\infty}>c_1$ such that $\sigma(c) =\sigma_0$ for any $c \geq c_{\infty}$.
\end{itemize}
\end{prop}

To discuss the asymptotical behaviors of the function $c \mapsto \sigma(c)$ as $c \to \infty$ for the case $\mu=-1$, we need to investiagte the following zero mass equation,
\begin{align} \label{zequ00}
-\Delta u + |x|^{-b}|u|^{q-2} u = |x|^{-b}|u|^{p-2} u \quad \mbox{in} \,\, \R^N.
\end{align}
Here the natural Sobolve space $X$ related to \eqref{zequ00} is defined by the completion of $C_0^{\infty}(\R^N)$ under the norm
$$
\|u\|_X:=\left(\int_{\R^N}|\nabla u|^2 \,dx\right)^{\frac 12} + \left(\int_{\R^N}|x|^{-b}|u|^q \, dx\right)^{\frac 1 q}.
$$
Actually, the quantity $\sigma_0>0$ arising in Proposition \ref{prop2} is the ground state energy of solutions to \eqref{zequ00}. The proofs of the assertions $(\textnormal{v})$, $(\textnormal{vi})$ and $(\textnormal{vii})$ mainly benefit from the decay estimates of solutions to \eqref{zequ00} presented in the proposition below.

\begin{prop}\label{zestimate}
Let $N \geq 3$, $2<q<p<{2(N-b)}/{(N-2)^+}$ and $0<b<2$, then there exists a positive, radially symmetric and decreasing ground state $u \in X$ to \eqref{zequ00}. Moreover,  $u \in C(\R^N) \cap C^2(\R^N\backslash\{0\})$ and there holds that
\begin{align*} 
u(x) \underset{|x| \to \infty}{\sim}\left\{
\begin{aligned}
&|x|^{-\alpha} \quad &\mbox{if }\,\, q\neq (2N-2-b)/(N-2),\\
&|x|^{2-N} \left(\ln |x|\right)^{\frac{2-N}{2-b}}\quad  &\mbox{if }\,\, q = (2N-2-b)/(N-2),
\end{aligned}
\right.
\end{align*}
where $\alpha:=\max\{(2-b)/(q-2), N-2\}>0$.
\end{prop}

To show the radially symmetric and decreasing property of ground states to \eqref{zequ}, we shall adopt the polarization arguments developed in \cite{BS} as well as \cite[Theorem 1.1]{BZ}, because the moving plane method seems unavailable in the current situation.

\begin{rem}
As a simple application of Proposition \ref{prop2}, we can obtain that, if $2<q<2+2(2-b)/N$ and $ N=3$ or $2<q \leq 2+2(2-b)/N$ and $N=4$ or $q>2$ and $N \geq 5$, then any ground state to \eqref{zequ00} belongs to $L^2(\R^N)$. 
\end{rem}

\begin{rem}
It follows from Proposition \ref{zestimate} that ground states to \eqref{zequ} only admit algebraical decay at infinity. This is in contrary to the case for $\lambda>0$, where ground states to \eqref{equ} indeed enjoy exponentially decay at infinity.
\end{rem}

Finally, we turn to investigate dynamical behaviors of solutions to the Cauchy problem for \eqref{equt}. As a direct consequence of the compactness of any minimizing sequence and Lemma \ref{localwp}, we first obtain orbital stability of the set of (local) minimizers in the following sense.

\begin{defi}\label{def1}
We say that a set $\mathcal{G} \subset H^1(\R^N)$ is orbitally stable, if for any $\eps>0$ there exists a constant $\delta>0$ such that if $\psi_0 \in H^1(\R^N)$ satisfies $\inf_{u \in \mathcal{G}} \| \psi_0-u\| \leq \delta$, then 
$$
\sup_{t \geq 0}\inf_{u \in \mathcal{G}} \| \psi(t)-u\| \leq \eps,
$$ 
where $\psi(t) \in C([0, \infty); H^1(\R^N))$ is the solution to the Cauchy problem for \eqref{equt} with initial datum $\psi_0$.
\end{defi}

\begin{thm} \label{thmstable}
Under the assumptions of Theorem \ref{thm1} or Theorem \ref{thm3}, then the set of minimizers obtained in Theorem \ref{thm1} or the set of local minimizers obtained in Theorem \ref{thm3} is orbitally stable.
\end{thm}

If $2<q<p<2+2(2-b)/N$ and $0<b<\min\{2, N\}$, by Gagliardo-Nirenberg’s inequality \eqref{GN}, then any solution to the Cauchy problem for \eqref{equt} exists globally in time. However, if $p \geq 2+2(2-b)/N$ and $0<b<\min\{2, N\}$, then solutions to the Cauchy problem for \eqref{equt} may blow up in finite time. In this situation, we have global existence versus blowup dichotomy of solutions to the problem for initial data below the ground state energy threshold.

\begin{thm}\label{globalexistence}
Under the assumptions of Theorem \ref{thm4}, Theorem \ref{thm41}, Theorem \ref{thm5} or Theorem \ref{thm6}, if $\psi_0 \in \mathcal{K}^+(c)$, then the solution to the Cauchy problem for \eqref{equt} with initial datum $\psi_0$ exists globally time, where
$$
\mathcal{K}^+(c):=\{ u\in S(c) : E(u) < \gamma(c), Q(u) >0\}.
$$
\end{thm}

To prove Theorem \ref{globalexistence}, we shall make use of the variational characterization of the ground state energy $\gamma(c)$ and the conservation laws in Lemma \ref{localwp}.

\begin{rem}
If $\psi_0 \in \mathcal{K}(c)$, then the solution to the Cauchy problem for \eqref{equt} with initial datum $\psi_0$ scatter under certain assumptions. Such a topic will be discussed in future publication.
\end{rem}

\begin{thm}\label{blowup}
Under the assumptions of Theorem \ref{thm4}, Theorem \ref{thm41}, Theorem \ref{thm5} with $2<q \leq 2+2(2-b)/N$ or Theorem \ref{thm6}, if $\psi_0 \in \mathcal{K}^-(c)$ with $|x| \psi_0 \in L^2(\R^N)$ or $N \geq 2$, $2<p \leq 6$ and $\psi_0$ is radially symmetric, then the solution to the Cauchy problem for \eqref{equt} with initial datum $\psi_0$ blows up in finite time, where
$$
\mathcal{K}^-(c):=\{ u\in S(c) : E(u) < \gamma(c), Q(u) <0\}.
$$
\end{thm}

To prove Theorem \ref{blowup}, the essential argument is to deduce the evolution of (localized) virial quantity along time, see Lemmas \ref{virial} and \ref{viriall2}.

{
\begin{rem}
In the proof of Theorem \ref{blowup}, the essential ingredient to derive blowup of solutions in finite time for radial data is the well-known Strauss inequality \eqref{sti}. This inequality holds only for $N \geq 2$. Then it is unknown to us if blowup of solutions in finite time for radial data occurs as well for $N=1$.
\end{rem}}

As an immediate consequence of Theorem \ref{blowup}, we have the following corollary.

\begin{cor}\label{instability}
Under the assumptions Theorem \ref{thm4}, Theorem \ref{thm41}, Theorem \ref{thm5} with $2<q \leq 2+2(2-b)/N$ or Theorem \ref{thm6}, let $u \in S(c)$ be a ground state to \eqref{equ}-\eqref{mass} at the level $\sigma(c)$, if $|x|u \in L^2(\R^N)$ or $N \geq 2$ and $2<p<6$, then it is unstable by blowup in finite time, i.e. for any $\eps>0$, there exists $v \in H^1(\R^N)$ such that $\|v-u\| \leq \eps$ and the solution to the Cauchy problem for \eqref{equt} with initial datum $v$ blows up in finite time.
\end{cor}

\subsection*{Structure of the Paper} This paper is organized as follows. In Section \ref{pre}, we present some preliminary results used to prove our main results and give the proof of Proposition \ref{radial}. In Section \ref{sub}, we consider the mass subcritical case and show the proofs of Proposition \ref{prop11} and Theorem \ref{thm1}. Later, in Section \ref{critical}, we investigate the mass critical case and prove Theorem \ref{thm2}. Section \ref{sup1} is devoted to the study of the mass supercritical case with focusing perturbation and contains the proofs of Theorems \ref{thm3}, \ref{thm4} and \ref{thm41}. And Section \ref{sup2} is devoted to the study of the mass supercritical case with defocusing perturbation and contains the proof of Theorem \ref{thm5}. In Section \ref{rcritical}, we revisit the mass critical case and establish Theorem \ref{thm6}. Successively, in Section \ref{prop}, we consider asymptotical behaviors of the function $c \mapsto \sigma(c)$ and show the proofs of Propositions \ref{prop2} and \ref{zestimate}. Finally, in Section \ref{dynamics}, we investigate dynamical behaviors of solutions to the Cauchy problem for \eqref{equt} and prove Theorems \ref{thmstable}, \ref{globalexistence} and \ref{blowup} as well as Corollary \ref{instability}.

\subsection*{Notations}  Throughout the paper, for any $1 \leq p < \infty$, we denote by $L^p(\R^N)$ the usual Lebesgue
space equipped with the norm 
$$
\|u\|_p=\left(\int_{\R^N} |u|^p \, dx \right)^{\frac 1p}.
$$ 
We denote by $H^1(\R^N)$ the usual Sobolev space equipped with the standard norm
$$
\|u\|=\|u\|_2 +\|\nabla u\|_2.
$$
In addition, we use the letter $C$ for a generic positive constant, whose value may change from line to line. We use the notation $o_n(1)$ for any quantity which tends to zero as $n \to \infty$. For simplicity, we define
$$
2^*_b:=\frac {2(N-b)} {(N-2)^+},
$$
where $2^*=\infty$ if $N=1, 2$ and $2^*= {2(N-b)}/{(N-2)}$ if $N \geq 3$.
\section{Preliminaries} \label{pre}

In this section, we are going to present some preliminary results used to establish our main results. Let us first introduce the well-known Gagliardo-Nirenberg iequality {in \cite{G1}}.

\begin{lem}
Let $N \geq 1$, $2<p<2_b^*$ and $0<b<\min\{2, N\}$, then there exists a constant $C_{N,p,b}>0$ such that, for any $u \in H^1(\R^N)$,
\begin{align} \label{GN}
\int_{\R^N} |x|^{-b} |u|^{p} \, dx \leq C_{N,p,b} \left(\int_{\R^N} |\nabla u|^2 \, dx\right)^{\frac{N(p-2)}{4} + \frac b 2} \left(\int_{\R^N} |u|^2 \,dx\right)^{\frac p2 -\frac{N(p-2)}{4} - \frac b 2},
\end{align}
where the optimal constant $C_{N, p, b}>0$ is given by
\begin{align} \label{cnp}
C_{N,p,b}=\left(\frac{(p-2)N+2b}{2p-((p-2)N+2b)}\right)^{\frac{4-((p-2)N+2b)}{4}} \frac{2p}{((p-2)N+2b)\|Q_p\|_2^{p-2}}.
\end{align}
Moreover, $Q_{p,b} \in H^1(\R^N)$ is the ground state to the equation
\begin{align} \label{equ1}
-\Delta Q_{p,b}+Q_{p,b}=|x|^{-b}|Q_{p,b}|^{p-2} Q_{p,b} \quad \mbox{in} \,\, \R^N.
\end{align}
\end{lem}

\begin{rem}
The radial symmetry of ground states to \eqref{equ1} can be directly achieved by invoking the symmetric-decreasing arrangement arguments in \cite{LL} and \cite[Theorem 1.1]{BZ}. For the uniqueness of ground states to \eqref{equ1}, we refer the readers to \cite{G, T, Y}.
\end{rem}

\begin{lem} \label{inequ}
Let ${N \geq 1}$, $2<q<p<2_b^*$ and $0<b<\min\{2, N\}$, then there exists a constant $C>0$ depending only on $N, p, q$ and $b$ such that, for any $u \in H^1(\R^N)$,
$$
\int_{\R^N} |x|^{-b}|u|^p \,dx \leq C \left(\int_{\R^N} |x|^{-b}|u|^q \,dx\right)^{1-\theta} \left(\int_{\R^N} |\nabla u|^2 \, dx\right)^{ \frac{\theta \left(N(r-2)+2b\right)}{4}} \left(\int_{\R^N} |u|^2 \,dx\right)^{\frac{\theta\left(2(r-b)-N(r-2)\right)}{4}},
$$
where $p=(1-\theta) q + \theta r$, $2<q<p<r <2_b^*$ if $N=1,2$ and $2<q<p<r =2_b^*$ if $N \geq 3$.
\end{lem}
\begin{proof}
Using H\"older's inequality and \eqref{GN}, we are able to derive that
\begin{align*}
\int_{\R^N} |x|^{-b}|u|^p \,dx &\leq  \left(\int_{\R^N} |x|^{-b}|u|^q \,dx\right)^{1-\theta}  \left(\int_{\R^N} |x|^{-b}|u|^r \,dx\right)^{\theta}  \\
& \leq C \left(\int_{\R^N} |x|^{-b}|u|^q \,dx\right)^{1-\theta}\left(\int_{\R^N} |\nabla u|^2 \, dx\right)^{ \frac{\theta \left(N(r-2)+2b\right)}{4}} \left(\int_{\R^N} |u|^2 \,dx\right)^{\frac{\theta\left(2(r-b)-N(r-2)\right)}{4}}.
\end{align*}
Thus the proof is completed.
\end{proof}

\begin{lem} \cite[Lemma 2.1]{AC} \label{cembedding}
Let ${N \geq 1}$ and $0<b<\min\{2, N\}$, then $H^1(\R^N)$ is compactly embedded into $L^p(\R^N, |x|^{-b}dx)$ for any $2<p<2_b^*$. 
\end{lem}

\begin{lem} \label{ph}
Let ${N \geq 1}$, $2<q<p<2^*_b$, $0<b<\min\{2, N\}$ and $\mu=\pm 1$. If $u \in S(c)$ is a solution to \eqref{equ}-\eqref{mass}, then $Q(u)=0$.
\end{lem}
\begin{proof}
Since $u \in S(c)$ is a solution to \eqref{equ}-\eqref{mass}, then there exists $\lambda \in \R$ such that $u \in S(c)$ satisfies the equation
\begin{align} \label{equ2}
-\Delta u + \lambda u=\mu |x|^{-b}|u|^{q-2} u + |x|^{-b}|u|^{p-2} u \quad \mbox{in} \,\, \R^N.
\end{align}
Multiplying \eqref{equ2} by $x \cdot \nabla u$ and integrating on $B_R(0)$, we first have that
\begin{align} \label{ph11}
\begin{split}
&-\int_{B_R(0)} \Delta u (x \cdot \nabla u ) \, dx + \lambda \int_{B_R(0)} u (x \cdot \nabla u ) \, dx \\
&=\int_{B_R(0)}|x|^{-b}|u|^{q-2} u (x \cdot \nabla u) \,dx +\int_{B_R(0)}|x|^{-b}|u|^{p-2} u (x \cdot \nabla u) \,dx.
\end{split}
\end{align} 
With help of the divergence theorem, it is simple to compute that
\begin{align*}
-\int_{B_R(0)} \Delta u (x \cdot \nabla u ) \, dx&= \int_{B_R(0)} \nabla u \cdot \nabla (x \cdot \nabla u) \, dx-R\int_{\partial B_R(0)} \left| \nabla u \cdot \bf{n} \right|^2 \,dS\\
&=\frac{2-N}{2} \int_{B_R(0)}|\nabla u|^2\, dx-\frac R 2 \int_{\partial B_R(0)} \left| \nabla u \cdot \bf{n} \right|^2 \,dS
\end{align*}
and
\begin{align*}
\int_{B_R(0)} u (x \cdot \nabla u ) \, dx=\frac 12 \int_{B_R(0)} x \cdot \nabla \left(|u|^2\right) \, dx=-\frac N 2 \int_{B_R(0)} |u|^2 \, dx + \frac  R 2 \int_{\partial B_R(0)} |u|^2 \, dS,
\end{align*}
where the vector $\bf{n}$ denotes the outward normal to $\partial B_R(0)$. In addition, we are able to show that
\begin{align*}
\int_{B_R(0)}|x|^{-b}|u|^{q-2} u (x \cdot \nabla u) \,dx&=\frac 1 q \int_{B_R(0)} \left(|x|^{-b}x\right) \cdot \nabla\left(|u|^q\right) \, dx\\
&=\frac {b-N}{q} \int_{B_R(0)} |x|^{-b}|u|^q \, dx-\frac  R q \int_{\partial B_R(0)} |x|^{-b} |u|^q \, dS
\end{align*}
and
\begin{align*}
\int_{B_R(0)}|x|^{-b}|u|^{p-2} u (x \cdot \nabla u) \,dx=\frac {b-N}{p} \int_{B_R(0)} |x|^{-b}|u|^p \, dx-\frac  R p \int_{\partial B_R(0)} |x|^{-b} |u|^p \, dS.
\end{align*}
As a result, from \eqref{ph11}, we derive that
\begin{align} \label{ph1}
\begin{split}
&\frac{N-2}{2} \int_{B_R(0)}|\nabla u|^2\, dx +\frac {\lambda N}{2} \int_{B_R(0)} |u|^2 \, dx \\
&=\frac {\mu(N-b)}{q} \int_{B_R(0)} |x|^{-b}|u|^q \, dx+\frac {N-b}{p} \int_{B_R(0)} |x|^{-b}|u|^p \, dx +I_R,
\end{split}
\end{align} 
where
$$
I_R:=-\frac R 2 \int_{\partial B_R(0)} \left| \nabla u \cdot \bf{n} \right|^2 \,dS+\frac  R 2 \int_{\partial B_R(0)} |u|^2 \, dS +\frac  {\mu R} {q} \int_{\partial B_R(0)} |x|^{-b} |u|^q \, dS +\frac  R p \int_{\partial B_R(0)} |x|^{-b} |u|^p \, dS.
$$
Thanks to $u \in H^1(\R^N)$, then
\begin{align*}
&\int_{\R^N} |\nabla u|^2 + |u|^2 + |x|^{-b} |u|^q + |x|^{-b} |u|^p \, dx \\
&= \int_0^{\infty} \int_{\partial B_R(0)}|\nabla u|^2 + |u|^2 + |x|^{-b} |u|^q + |x|^{-b} |u|^p \,dSdR<\infty.
\end{align*}
It then follows that there exists a sequence $\{R_n\} \to \infty$ as $n \to \infty$ such that $I_{R_n}=o_n(1)$.  Making use of \eqref{ph1} with $R=R_n$ and taking the limit as $n \to \infty$, we then get that
\begin{align} \label{ph2}
\frac{N-2}{2} \int_{\R^N}|\nabla u|^2\, dx +\frac {\lambda N}{2} \int_{\R^N} |u|^2 \, dx=\frac {\mu(N-b)}{q} \int_{\R^N} |x|^{-b}|u|^q \, dx+\frac {N-b}{p} \int_{\R^N} |x|^{-b}|u|^p \, dx.
\end{align}
On the other hand, multiplying \eqref{equ2} by $u$ and integrating on $\R^N$, we can obtain that
\begin{align}\label{ph3}
\frac{N}{2} \int_{\R^N}|\nabla u|^2\, dx +\frac {\lambda N}{2} \int_{\R^N} |u|^2 \, dx=\frac {\mu N}{2} \int_{\R^N} |x|^{-b}|u|^q \, dx+\frac N 2 \int_{\R^N} |x|^{-b}|u|^p \, dx.
\end{align}
Thereby, combining \eqref{ph2} and \eqref{ph3}, we conclude that
$$
\int_{\R^N}|\nabla u|^2\, dx=\frac {\mu \left(N(q-2) +2b\right)}{2q} \int_{\R^N} |x|^{-b}|u|^q \, dx+\frac{N(p-2)+2b}{2p} \int_{R^N} |x|^{-b}|u|^p \, dx.
$$
Thus the proof is completed.
\end{proof}

\begin{prop} \label{radial}
Let ${N \geq 1}$, $2<q<p<2^*_b$, $0<b<\min\{2, N\}$ and $\mu =\pm 1$. If $u \in H^1(\R^N)$ is a positive solution to the equation
\begin{align} \label{equ111}
-\Delta u + \lambda u= \mu |x|^{-b}|u|^{q-2} u + |x|^{-b}|u|^{p-2} u \quad \mbox{in} \,\, \R^N,
\end{align}
where $\lambda>0$, then it is radially symmetric and decreasing in the radial direction.
\end{prop}
\begin{proof}
By applying standard bootstrap arguments, we first have that $u \in C(\R^N) \cap C^2(\R^N\backslash\{0\})$ and $u(x) \to 0$ as $|x| \to \infty$. In order to adapt the moving plane method {introduced in \cite{CLO}}, we need to introduce some notations as follows. For $s \in \R$, we define
$$
\Sigma_s:=\{x \in \R^N : x_1 \geq s\}, \quad
\Sigma_{s, u}:=\{x \in \Sigma_s : u_s(x) >u(x)\},
$$
where $u_s(x):=u(x_s)$ and $x_s:=(2s-x_1, x_2, \cdots, x_N)$. Let $u \in H^1(\R^N)$ be a positive solution to \eqref{equ111}, by \cite[Theorem 6.23]{LL}, then
$$
u(x)=\int_{\R^N} G_{\lambda}(x-y) \left(\mu |y|^{-b}|u(y)|^{q-2} + |y|^{-b}|u(y)|^{p-2}\right)u(y) \, dy,
$$
where $G_{\lambda}$ is the Yukawa potential defined by 
$$
G_{\lambda}(x):=\int_0^{\infty} (4\pi t)^{-\frac N 2} \textnormal{exp} \left(-\frac{|x|^2}{4t}-\lambda t\right) \, dt>0.
$$
Observe that
\begin{align*}
u(x)&=\int_{\Sigma_s} G_{\lambda}(x-y) \left(\mu |y|^{-b}|u(y)|^{q-2} + |y|^{-b}|u(y)|^{p-2}\right)u(y) \, dy\\
& \quad + \int_{\R^N \backslash \Sigma_s}  G_{\lambda}(x-y) \left(|\mu y|^{-b}|u(y)|^{q-2} + |y|^{-b}|u(y)|^{p-2}\right)u(y) \, dy
\end{align*}
and
\begin{align*}
&\int_{\R^N \backslash \Sigma_s}  G_{\lambda}(x-y) \left(|\mu y|^{-b}|u(y)|^{q-2} + |y|^{-b}|u(y)|^{p-2}\right)u(y) \, dy\\
&=\int_{\Sigma_s} G_{\lambda}(x-y_s) \left(\mu |y_s|^{-b}|u(y_s)|^{q-2} + |y_s|^{-b}|u(y_s)|^{p-2}\right)u(y_s) \, dy\\
&= \int_{\Sigma_s} G_{\lambda}(x_s-y) \left(\mu |y_s|^{-b}|u(y_s)|^{q-2} + |y_s|^{-b}|u(y_s)|^{p-2}\right)u(y_s) \, dy,
\end{align*}
where we used the fact that $|x_s-y|=|x-y_s|$. Therefore, we get that
\begin{align} \label{s1}
\begin{split}
u(x)&=\int_{\Sigma_s} G_{\lambda}(x-y) \left(\mu |y|^{-b}|u(y)|^{q-2} + |y|^{-b}|u(y)|^{p-2}\right)u(y) \, dy\\
& \quad +\int_{\Sigma_s} G_{\lambda}(x_s-y) \left(\mu |y_s|^{-b}|u(y_s)|^{q-2} + |y_s|^{-b}|u(y_s)|^{p-2}\right)u(y_s) \, dy.
\end{split}
\end{align}
Making use of \eqref{s1}, we then obtain that 
\begin{align} \label{s11}
\begin{split}
u_s(x)=u(x_s)&=\int_{\Sigma_s} G_{\lambda}(x_s-y) \left(\mu |y|^{-b}|u(y)|^{q-2} + |y|^{-b}|u(y)|^{p-2}\right)u(y) \, dy dz \\
& \quad + \int_{\Sigma_s} G_{\lambda}(x-y) \left(\mu |y_s|^{-b}|u(y_s)|^{q-2} + |y_s|^{-b}|u(y_s)|^{p-2}\right)u(y_s) \, dy.
\end{split}
\end{align}
Therefore, combining \eqref{s1} and \eqref{s11}, we deduce that
\begin{align*}
u_s(x)-u(x)
&=\int_{\Sigma_s} \left(G_{\lambda}(x-y)-G_{\lambda}(x_s-y) \right)\left(\mu |y_s|^{-b}|u(y_s)|^{q-2} + |y_s|^{-b}|u(y_s)|^{p-2}\right)u(y_s)\\
&\quad +\int_{\Sigma_s} \left(G_{\lambda}(x_s-y)-G_{\lambda}(x-y) \right)\left(\mu |y|^{-b}|u(y)|^{q-2} + |y|^{-b}|u(y)|^{p-2}\right)u(y)\, dy.
\end{align*}
Next we claim that there exists a constant $s_0<0$ with $-s_0>0$ large enough such that $\Sigma_{s,u}=\emptyset$ for any $s<s_0$. In the following, we shall assume that $s<0$. Since $|x-y|\leq|x_s-y|$ for any $x, y \in \Sigma_s$, from the definition of $G_{\lambda}$, then
$$
0<G_{\lambda}(x_s-y) \leq G_{\lambda}(x-y).
$$
Hence there holds that
\begin{align*}
u_s(x)-u(x) &= \int_{\Sigma_s}\left(G_{\lambda}(x-y)-G_{\lambda}(x_s-y) \right)\left(\mu |y_s|^{-b}|u(y_s)|^{q-2} + |y_s|^{-b}|u(y_s)|^{p-2}\right)\left(u(y_s)-u(y)\right) \, dy\\
&\quad + \int_{\Sigma_s}\left(G_{\lambda}(x-y)-G_{\lambda}(x_s-y) \right)\left(\mu |y_s|^{-b}|u(y_s)|^{q-2} + |y_s|^{-b}|u(y_s)|^{p-2}\right)u(y) \, dy\\ 
& \quad +\int_{\Sigma_s}\left(G_{\lambda}(x_s-y)-G_{\lambda}(x-y) \right) \left(\mu |y|^{-b}|u(y)|^{q-2} + |y|^{-b}|u(y)|^{p-2}\right)u(y) \, dy \\
&\leq \int_{\Sigma_{s, u}} G_{\lambda}(x-y)\left(\mu |y_s|^{-b}|u(y_s)|^{q-2} + |y_s|^{-b}|u(y_s)|^{p-2}\right)\left(u_s(y)-u(y)\right) \, dy \\
&\quad + \mu \int_{\Sigma_{s ,u}} G_{\lambda}(x-y)\left( |y_s|^{-b} |u(y_s)|^{q-2} - |y_s|^{-b} |u(y)|^{q-2} \right)u(y)\, dy  \\
& \quad + \int_{\Sigma_{s ,u}} G_{\lambda}(x-y) \left(|y_s|^{-b}|u(y_s)|^{p-2} -|y_s|^{-b}|u(y)|^{p-2} \right)u(y)\, dy :=I_1(x)+I_2(x) +I_3(x).
\end{align*}
In view of \cite[Theorem 6.23]{LL}, we know that $G_{\lambda} \in L^1(\R^N)$. Using Young's inequality, we then get that
\begin{align*}
\|I_1\|_{L^r(\Sigma_{s, u})} & \leq \left\|\left(u_s(x)-u(x)\right)\left(|x_s|^{-b}|u(x_s)|^{q-2} + |x_s|^{-b}|u(x_s)|^{p-2}\right) \right\|_{L^r(\Sigma_{s,u})} \\
& \leq \|u_s-u\|_{L^r(\Sigma_{s,u})} \left\| |x_s|^{-b}|u(x_s)|^{q-2} + |x_s|^{-b}|u(x_s)|^{p-2}\right\|_{L^{\infty}(\Sigma_{s,u})} \\
& \leq \|u_s-u\|_{L^r(\Sigma_{s,u})} \left\||x|^{-b}|u(x)|^{q-2} + |x|^{-b}|u(x)|^{p-2}\right\|_{L^{\infty}(\R^N \backslash \Sigma_s)} \\
& \leq (-s)^{-b}\|u_s-u\|_{L^r(\Sigma_{s,u})} \left(\|u\|_{L^{\infty}(\R^N \backslash \Sigma_s)}^{q-2}+\|u\|_{L^{\infty}(\R^N \backslash \Sigma_s)}^{p-2}\right),
\end{align*}
where $r>1$. In addition, we are able to deduce that 
\begin{align*}
\|I_2\|_{L^r(\Sigma_{s, u})}  &\leq C_q \|u_s-u\|_{L^r(\Sigma_{s,u})} \left\| |x_s|^{-b}|u(x_s)|^{q-2} \right\|_{L^{\infty}(\Sigma_{s,u})} \\
& \leq C_q \|u_s-u\|_{L^r(\Sigma_{s,u})} \left\||x|^{-b}|u(x)|^{q-2}\right\|_{L^{\infty}(\R^N \backslash \Sigma_{s,u})} \\
&\leq C_q (-s)^{-b} \|u_s-u\|_{L^r(\Sigma_{s,u})} \left\|u\right\|^{q-2}_{L^{\infty}(\R^N \backslash \Sigma_{s,u})}.
\end{align*}
Similarly, we have that
$$
\|I_3\|_{L^r(\Sigma_{s, u})}  \leq  C_p (-s)^{-b} \|u_s-u\|_{L^r(\Sigma_{s,u})} \left\|u\right\|^{p-2}_{L^{\infty}(\R^N \backslash \Sigma_{s,u})}.
$$
As a consequence, we get that
\begin{align*}
\|u_s-u\|_{L^r(\Sigma_{s,u})} \leq C (-s)^{-b} \|u_s-u\|_{L^r(\Sigma_{s,u})} \left(\|u\|_{L^{\infty}(\R^N \backslash \Sigma_{s,u})}^{q-2}+\|u\|_{L^{\infty}(\R^N \backslash \Sigma_{s,u})}^{p-2}\right).
\end{align*}
Thus the claim follows by taking $-s>0$ large enough. Therefore, there holds that $u_s(x) \leq u(x)$ for any $x \in \Sigma_{s}$ and $s<s_0<0$. Suppose that there exists $s_1<s_0$ such that $u_{s_1}(x) \leq u(x)$ for any $x \in \Sigma_{s_1}$ and $u_{s_1}\not\equiv u$ in $\Sigma_{s_1}$. Then, reasoning as before, we can derive that there exists $\eps>0$ small enough such that $u_{s}(x) \leq u(x)$ for any $x \in \Sigma_{s}$ and $s \in [s_1, s_1+\eps)$.
This in turn implies that if the moving plane process stops, then $u_{s_1}(x) \leq u(x)$ for any $x \in \Sigma_{s_1}$ and $u_{s_1} \equiv u$ in $\Sigma_{s_1}$. Up to translations, we may assume that 
$$
u(0)=\sup_{x \in \R^N} u(x).
$$
Hence the moving plane process starting from any direction has to stop at the origin. Thus we have the conclusion and the proof is completed.
\end{proof}

\section{Mass subcritical case} \label{sub}

In this section, we consider the case $2<q<p<2+2(2-b)/N$ and present the proofs of Proposition \ref{prop11} and Theorem \ref{thm1}.

\begin{proof} [Proof of Proposition \ref{prop11}]
Let us first verify the assertion $(\textnormal{i})$. In virtue of \eqref{GN}, we have that
\begin{align} \label{b1}
\hspace{-1cm}E(u) \geq \frac 12 \|\nabla u\|_2^2-\frac{C_{N,q}}{q}\|\nabla u\|_2^{\frac{N(q-2)}{2} + \frac b 2} c^{\frac q2 -\frac{N(q-2)}{4} -\frac b 2}- \frac{C_{N,p}}{p}\|\nabla u\|_2^{\frac{N(p-2)}{2} + \frac b 2} c^{\frac p2 -\frac{N(p-2)}{4} - \frac b 2}.
\end{align}
Notice that 
$$
0<\frac{N(q-2)}{2} + \frac b 2<\frac{N(p-2)}{2} + \frac b 2<2.
$$
It then follows from \eqref{b1} that $m(c)>-\infty$ for any $c>0$. Furthermore, in view of \eqref{scaling1}, we see that $E(u_t) \to 0$ as $t \to 0^+$ for any $u \in S(c)$. This infers that $m(c) \leq 0$. Therefore, we derive that $-\infty<m(c) \leq 0$. Next we show that the subadditivity inequality holds. {By the definition of $m(c)$ and the fact that $C_0^{\infty}(\R^N)$ is dense in $H^1(\R^N)$, we know that, for any $\eps>0$, there exist $u_1 \in S(c_1) \cap C_0^{\infty}(\R^N) $ and $u_2 \in S(c_2) \cap C_0^{\infty}(\R^N)$ such that 
$$
E(u_1) \leq m(c_1) + \frac {\eps}{2}, \quad E(u_2) \leq m(c_2) + \frac {\eps}{2}.
$$}
Note that $E$ is invariant under any translation in $\R^N$. Then we may assume that $\mbox{supp}\,u_1 \cap \mbox{supp}\, u_2=\emptyset$. Therefore, we have that
$$
m(c_1+c_2) \leq E(u_1+u_2) \leq E(u_1)+E(u_2) \leq  m(c_1) +m(c_2) + \eps,
$$
from which we have the desired conclusion. This further gives that the function $c \mapsto m(c)$ is nonincreasing on $(0, \infty)$.

We now prove the assertion $(\textnormal{ii})$. Let $c>0$ and $\{c_n\} \subset \R^+$ be a sequence such that $c_n \to c$ as $n \to \infty$. We shall deduce that $m(c_n) \to m(c)$ as $n \to \infty$. From the definition of $m(c)$, then there exists a sequence $\{u_n\} \subset S(c_n)$ such that $E(u_n) = m(c_n)  +o_n(1)$. Define
$$
u=\frac{u_n}{\|u_n\|_2} c_n^{\frac 12} \in S(c).
$$
{Observe that $E(u_n) \leq C$ for some $C>0$, by \eqref{b1}, then $\{u_n\}$ is bounded in $H^1(\R^N)$.} This shows that $E(u)=E(u_n) +o_n(1)$, because of $c_n \to c$ as $n \to \infty$. Hence we obtain that $m(c) \leq m(c_n) +o_n(1)$. Similarly, we can show that $m(c_n) \leq m(c) +o_n(1)$. Therefore, the desired conclusion follows.

Finally we infer the assertion $(\textnormal{iii})$. Let us assume there exists a constant $c_0^*>0$ such that $m(c)<0$ for any $0<c<c_0^*$. If not, then there exists some $c_0>0$ such that $m(c_0)=0$. It then follows from the assertion $(\textnormal{i})$ that $m(c)=0$ for any $0<c<c_0$. This already indicates that $m(c) \to 0$ as $c \to 0^+$. If $m(c)<0$, then there exists $u_c \in S(c)$ such that $m(c) \leq E(u_c)<0$. In light of \eqref{GN}, then
$$
\frac 12 \|\nabla u_c\|_2^2 \leq \frac{C_{N,q}}{q}\|\nabla u_c\|_2^{\frac{N(q-2)}{2} + \frac b 2} c^{\frac q2 -\frac{N(q-2)}{4} -\frac b 2}+ \frac{C_{N,p}}{p}\|\nabla u_c\|_2^{\frac{N(p-2)}{2} + \frac b 2} c^{\frac p2 -\frac{N(p-2)}{4} - \frac b 2}.
$$
Since $2<q<p<2+2(2-b)/N$, then $\|\nabla u_c\|_2 \to 0$ as $c \to 0^+$. This in turn gives that $E(u_c) \to 0$ as $c \to 0^+$. As a result, we have that $m(c) \to 0$ as $c \to 0^+$. Next we prove that $m(c) \to -\infty$ as $c \to \infty$. To do this, we first define
\begin{align*}
u_{c}(x):=c^{\frac{2-b}{2(2-b)-N(p-2)}} u(c^{\frac{p-2}{2(2-b)-N(p-2)}} x), \quad x \in \R^N.
\end{align*}
It is straightforward to find that $u_{c} \in S(c)$ if $u \in S(1)$ and
\begin{align*}
E(u_{c})&=c^{\frac{2(2-b)-(N-2)(p-2)}{2(2-b)-N(p-2)}} \left( \frac 12 \int_{\R^N}|\nabla u|^2 \, dx- \frac 1 p \int_{\R^N}|x|^{-b}|u|^p \, dx \right)- \frac{c^{\frac{q(2-b)-(N-b)(p-2)}{2(2-b)-N(p-2)}}}{q}\int_{\R^N}|x|^{-b}|u|^q \, dx.
\end{align*}
Let $u \in S(1)$ be such that
$$
\frac 12 \int_{\R^N}|\nabla u|^2 \, dx< \frac 1 p \int_{\R^N}|x|^{-b}|u|^p \, dx.
$$
Since $2(2-b)-(N-2)(p-2)>q(2-b)-(N-b)(p-2)>0$, then $E(u_c) \to -\infty$ as $c \to \infty$. Thereby, we obtain the desired result. Thus the proof is completed.
\end{proof}

\begin{proof}[Proof of Theorem \ref{thm1}]
Let us first prove the assertion $(\textnormal{i})$. Since $2<q<2+2(2-b)/N$, then $0<N(q-2)/2 + b <2$. In virtue of \eqref{scaling1}, we then see that $E(u_t)<0$ for any $t>0$ small enough, because of $\mu>0$. This immediately infers that $m(c)<0$ for any $c>0$. Next we discuss the compactness of any minimizing sequence to \eqref{gmin} in $H^1(\R^N)$. For this, we shall apply the concentration compactness principle from \cite{Li1, Li2}. Therefore, it suffices to rule out the possibilities of vanishing and dichotomy of any minimizing sequence. Since $2<q<p<2(N-b)/(N-2)^+$, then there exists $r_1, r_2>1$ such that
$$
2<qr_1<2^*, \quad 0<\frac{br_1}{r_1-1}<N, \quad 2<pr_2<2^*, \quad 0<\frac{br_2}{r_2-1}<N,
$$
where $2^*:=2N/(N-2)^+$. As a result of H\"older's inequality, we then know that
\begin{align} \label{v1}
\begin{split}
\int_{\R^N} |x|^{-b}|u_n|^q \, dx &=\int_{B_R(0)} |x|^{-b}|u_n|^q \, dx  + \int_{\R^N \backslash B_R(0)} |x|^{-b}|u_n|^q \, dx \\
& \leq \left(\int_{B_R(0)} |x|^{-\frac{br_1}{r_1-1}}\,dx \right)^{\frac{r_1-1}{r_1}}\left(\int_{B_R(0)} |u_n|^{qr_1} \, dx \right)^{\frac{1}{r_1}} + R^{-b}\int_{\R^N \backslash B_R(0)} |u_n|^q \, dx
\end{split}
\end{align}
and
\begin{align} \label{v2}
\int_{\R^N} |x|^{-b}|u_n|^p \, dx \leq \left(\int_{B_R(0)} |x|^{-\frac{br_1}{r_2-1}}\,dx \right)^{\frac{r_2-1}{r_2}} \left(\int_{B_R(0)} |u_n|^{pr_2} \, dx \right)^{\frac{1}{r_2}} + R^{-b}\int_{\R^N \backslash B_R(0)} |u_n|^p \, dx,
\end{align}
where $R>0$ is a constant. Note that
$$
\left(\int_{B_R(0)} |x|^{-\frac{br_1}{r_1-1}}\,dx \right)^{\frac{r_1-1}{r_1}}<\infty ,\quad \left(\int_{B_R(0)} |x|^{-\frac{br_2}{r_2-1}}\,dx \right)^{\frac{r_2-1}{r_2}}<\infty.
$$ 
Due to $m(c)<0$, from \eqref{v1}-\eqref{v2} and \cite[Lemma I.1]{Li2}, then vanishing can be excluded. It now remains to exclude dichotomy. To do this, we need to establish the strict subadditivity inequality $m(c_1+c_2)<m(c_1)+m(c_2)$ for any $c_1, c_2>0$. It can be done by deducing that $m(\theta c) <\theta m(c)$ for any $\theta>1$. Indeed, by the definition of $m(c)$, then, for any $\eps>0$, there exists $u \in S(c)$ such that $E(u) \leq m(c) + \eps$. Hence there holds that $\theta^{1/2} u \in S(\theta u)$ and 
\begin{align*}
m(\theta c) \leq E(\theta^{\frac 12} u)&=\theta \left(\frac 12 \int_{\R^N} |\nabla u|^2 \,dx -\frac{\theta^{\frac q 2 -1}}{q} \int_{\R^N} |x|^{-b}|u|^q \,dx -\frac {\theta^{\frac  p2 -1}}{p} \int_{\R^N} |x|^{-b} |u|^p \, dx \right) \\
& < \theta E(u) \leq \theta m(c) + \eps,
\end{align*}
from which we have that $m(\theta c) <\theta m(c)$ for any $\theta>1$. Thus we can obtain the desired conclusion. 

We now demonstrate the assertion $(\textnormal{ii})$. Using Lemma \ref{inequ} and taking $r=2+2(2-b)/N$, we then derive that
\begin{align*}
E(u) &\geq \frac 12 \int_{\R^N} |\nabla u|^2 \,dx +\frac{1}{q} \int_{\R^N} |x|^{-b}|u|^q \,dx -C \left(\int_{\R^N} |x|^{-b}|u|^q \,dx\right)^{1-\theta} \left(\int_{\R^N} |\nabla u|^2 \, dx\right)^{\theta} c^{\frac{2-b}{N}\theta} \\
& \geq \left(\frac 12 -C c^{\frac{2-b}{N}\theta}\right)  \int_{\R^N} |\nabla u|^2 \,dx + \left(\frac 1 q -C c^{\frac{2-b}{N}\theta} \right) \int_{\R^N} |x|^{-b}|u|^q \,dx,
\end{align*}
where we used Young's inequality to derive that
$$
\left(\int_{\R^N} |x|^{-b}|u|^q \,dx\right)^{1-\theta} \left(\int_{\R^N} |\nabla u|^2 \, dx\right)^{\theta} \leq(1-\theta)\int_{\R^N} |x|^{-b}|u|^q \,dx + \theta \int_{\R^N} |\nabla u|^2 \, dx.
$$
This readily suggests that $m(c) \geq 0$ for any $c>0$ small enough, because of $0<b<2$. On the other hand, from \eqref{scaling1}, we know that $E(u_t) \to 0$ as $t \to 0^+$ for any $u \in S(c)$. Therefore, we can conclude that there exists a constant $\tilde{c}_0>0$ such that $m(c) = 0$ for any $0<c<\tilde{c}_0$.  Next we prove that there exists a constant $\hat{c}_0>\tilde{c}_0>0$ such that $m(c)<0$ for any $c>\hat{c}_0$. According to \eqref{scaling1}, we see that, for any $u \in S(c)$,
\begin{align*}
E(u_t)&=t^{\frac{N}{2}(q-2)+b}\left(\frac{t^{2-\frac{N}{2}(q-2)-b}}{2} \int_{\R^N} |\nabla u|^2 \,dx+\frac 1 q\int_{\R^N}|x|^{-b}|u|^q \, dx -\frac{t^{\frac{N(p-q)}{2}}}{p} \int_{\R^N}|x|^{-b}|u|^p \, dx\right)\\
&=:t^{\frac{N}{2}(q-2)+b}F_u(t),
\end{align*}
where 
$$
F_u(t):=\frac{t^{2-\frac{N}{2}(q-2)-b}}{2} \int_{\R^N} |\nabla u|^2 \,dx+\frac 1 q\int_{\R^N}|x|^{-b}|u|^q \, dx -\frac{t^{\frac{N(p-q)}{2}}}{p} \int_{\R^N}|x|^{-b}|u|^p \, dx.
$$
It is easy to compute that
\begin{align} \label{minv}
\min_{t > 0} F_u(t)=\frac 1 q\int_{\R^N}|x|^{-b}|u|^q \, dx -C_{N,p,q,b}\frac{\left(\int_{\R^N}|x|^{-b}|u|^p \, dx\right)^{\frac{2(2-b)-N(q-2)}{2(2-b)-N(p-2)}}}{\left(\int_{\R^N} |\nabla u|^2 \,dx\right)^{\frac{N(p-q)}{2(2-b)-N(p-2)}}},
\end{align}
where the constant $C_{N,p,q,b}>0$ is defined by
$$
C_{N,p,q,b}:=\frac{2(2-b)-N(p-2)}{N(p-q)}2^{\frac{N(p-q)}{2(2-b)-N(p-2)}}p^{\frac{2(2-b)-N(q-2)}{N(p-2)-2(2-b)}}\left(\frac{2(2-b)-N(q-2)}{N(p-q)}\right)^{\frac{2(2-b)-N(q-2)}{N(p-2)-2(2-b)}}.
$$
Let us now take
$$
u=\frac{Q_{p,b}}{\|Q_{p,b}\|_2}c^{\frac 12} \in S(c),
$$
where $Q_{p,b} \in H^1(\R^N)$ is the ground state to \eqref{equ1}. Therefore, it follows from \eqref{GN} and \eqref{minv}  that 
\begin{align} \label{min11}
\begin{split}
\min_{t > 0} F_u(t)&=\frac {c^{\frac q 2}}{q\|Q_p\|_2^q}\int_{\R^N}|x|^{-b}|Q_p|^q \, dx -\frac{\widetilde{C}_{N,p,q,b}  c^{\frac{p\left(2(2-b)-N(q-2)\right)}{2\left(2(2-b)-N(p-2)\right)}-\frac{N(p-q)}{2(2-b)-N(p-2)}}}{\|Q_p\|_2^{\left(\frac{N(p-2)}{2}+ b\right)\frac{2(2-b)-N(q-2)}{2(2-b)-N(p-2)}-\frac{2N(p-q)}{2(2-b)-N(p-2)}}}\\
& \quad \times \left(\int_{\R^N} |\nabla Q_p|^2 \,dx\right)^{\left(\frac{N(p-2)}{4}+ \frac b2\right)\frac{2(2-b)-N(q-2)}{2(2-b)-N(p-2)}-\frac{N(p-q)}{2(2-b)-N(p-2)}},
\end{split}
\end{align}
where the constant $\widetilde{C}_{N,p,q,d}>0$ is defined by
$$
\widetilde{C}_{N,p,q,b}:=C_{N,p,q,b}C_{N,p}^{\frac{2(2-b)-N(q-2)}{2(2-b)-N(p-2)}}.
$$
Since $2<q<p$ and $0<b<2$, then
$$
\frac q2 <\frac{p\left(2(2-b)-N(q-2)\right)}{2\left(2(2-b)-N(p-2)\right)}-\frac{N(p-q)}{2(2-b)-N(p-2)}.
$$
As a consequence, from \eqref{min11}, there exists a constant $\hat{c}_0>\tilde{c}_0>0$ such that $\min_{t>0}F_u(t)<0$ for any $c>\hat{c}_0$. This indicates that $m(c)<0$ for any $c>\hat{c}_0$. Finally we show the compactness of any minimizing sequence to \eqref{gmin} in $H^1(\R^N)$ for any $c>\hat{c}_0$. For this, we adapt again the Lions concentration compactness principle. As discussed previously, the essential argument is to prove that $m(\theta c) < \theta m(c)$ for any $\theta>1$. By the definition of $m(c)$, for any $\eps>0$, there exists $u \in S(c)$ such that $E(u) \leq m(c) + \eps <0$. Define
$$
u_{\theta}(x):=\theta^{\frac{2-b}{2(2-b)-N(p-2)}} u(\theta^{\frac{p-2}{2(2-b)-N(p-2)}} x), \quad x \in \R^N.
$$
We see that $u_{\theta} \in S(\theta c)$ and
\begin{align*}
E(u_{\theta})&= \frac{\theta^{\frac{2(2-b)-(N-2)(p-2)}{2(2-b)-N(p-2)}}}{2} \int_{\R^N}|\nabla u|^2 \, dx + \frac{\theta^{\frac{q(2-b)-(N-b)(p-2)}{2(2-b)-N(p-2)}}}{q}\int_{\R^N}|x|^{-b}|u|^q \, dx\\
&\quad -\frac{\theta^{\frac{2(2-b)-(N-2)(p-2)}{2(2-b)-N(p-2)}}}{p}\int_{\R^N}|x|^{-b}|u|^p \, dx.
\end{align*}
Let us now define $f_{u}(\theta):=E(\theta u)-\theta E(u)$ for any $\theta>1$, then
\begin{align*}
f_u(\theta)&=\frac{\theta^{\frac{2(2-b)-(N-2)(p-2)}{2(2-b)-N(p-2)}}-\theta}{2} \int_{\R^N}|\nabla u|^2 \, dx +\frac{\theta^{\frac{q(2-b)-(N-b)(p-2)}{2(2-b)-N(p-2)}}-\theta}{q}\int_{\R^N}|x|^{-b}|u|^q \, dx\\
&\quad -\frac{\theta^{\frac{2(2-b)-(N-2)(p-2)}{2(2-b)-N(p-2)}}-\theta}{p}\int_{\R^N}|x|^{-b}|u|^p \, dx.
\end{align*}
It is simple to compute that 
\begin{align*}
\frac{d}{d \theta} f_u(\theta)&=\left(\frac{2(2-b)-(N-2)(p-2)}{2(2-b)-N(p-2)}\theta^{\frac{2(p-2)}{2(2-b)-N(p-2)}}-1\right)\left(\frac 12 \int_{\R^N}|\nabla u|^2 \, dx-\frac 1 p \int_{\R^N}|x|^{-b}|u|^p \, dx \right) \\
&\quad +\frac 1 q \left(\frac{q(2-b)-(N-b)(p-2)}{2(2-b)-N(p-2)}\theta^{\frac{b(p-2)}{2(2-b)-N(p-2)}}-1\right)\int_{\R^N}|x|^{-b}|u|^q \, dx
\end{align*}
and
\begin{align*}
&\frac{d^2}{d\theta^2} f_u(\theta) \\
&=\left(\frac{2\left(2(2-b)-(N-2)(p-2)\right)(p-2)}{\left(2(2-b)-N(p-2)\right)^2}\theta^{\frac{2(p-2)}{2(2-b)-N(p-2)}-1}\right) \left(\frac 12 \int_{\R^N}|\nabla u|^2 \, dx-\frac 1 p \int_{\R^N}|x|^{-b}|u|^p \, dx \right) \\
&\quad +\frac 1 q \left(\frac{b\left(q(2-b)-(N-b)(p-2)\right)(p-2)}{\left(2(2-b)-N(p-2)\right)^2}\theta^{\frac{b(p-2)}{2(2-b)-N(p-2)}-1}\right)\int_{\R^N}|x|^{-b}|u|^q \, dx.
\end{align*}
Observe that $\frac{d^2}{d\theta^2} f_u(\theta)<0$ if and only if
$$
\theta> \theta_0:=\left(\frac{b\left(q(2-b)-(N-b)(p-2)\right)}{2\left(2(2-b)-(N-2)(p-2)\right)}\frac{\frac 1q \int_{\R^N}|x|^{-b}|u|^q \, dx}{\frac 1 p \int_{\R^N}|x|^{-b}|u|^p \, dx-\frac 12 \int_{\R^N}|\nabla u|^2 \, dx}\right)^{\frac{2(2-b)-N(p-2)}{(2-b)(p-2)}}.
$$
Let us remark that $q(2-b)-(N-b)(p-2)>0$, due to $2<q<p<2+2(2-b)/N$. Therefore, we are able to deduce that
$$
b\left(q(2-b)-(N-b)(p-2) \right) \leq 2\left(2(2-b)-(N-2)(p-2)\right).
$$
Moreover, since $E(u)<0$, then
$$
0<\frac 1q \int_{\R^N}|x|^{-b}|u|^q \, dx<\frac 1 p \int_{\R^N}|x|^{-b}|u|^p \, dx-\frac 12 \int_{\R^N}|\nabla u|^2 \, dx.
$$
As a result, we obtain that $0<\theta_0<1$. Since $E(u)<0$, then it is straightforward to see that
$$
\frac{d}{d \theta} f_u(\theta){\mid_{\theta=1}}<\frac{2(p-2)}{2(2-b)-N(p-2)} E(u)<0.
$$
Consequently, we derive that $f_u(\theta)<0$ for any $\theta>1$. This then leads to
$$
m(\theta c) \leq E(u_{\theta})<\theta E(u) \leq \theta m(c) + \theta \eps,
$$
from which we get that $m(\theta c) <\theta m(c)$ for any $\theta>1$. It then yields the desired conclusion. Thus the proof is completed.
\end{proof}

\section{Mass critical case} \label{critical}

In this section, we study the case $2<q<p=2+2(2-b)/N$. The principal aim in this section is to establish Theorem \ref{thm2}.

\begin{proof}[Proof of Theorem \ref{thm2}]
Let us first prove the assertion $(\textnormal{i})$. In light of \eqref{GN} with $p=2+2(2-b)/N$, we have that, for any $u \in S(c)$,
$$
E(u) \geq \frac 12 \left(1 -\frac{NC_{N,b}c^{\frac{2-b}{N}}}{N+2-b}\right) \int_{\R^N} |\nabla u|^2 \, dx - \frac {C_{N,q}}{q}\left(\int_{\R^N} |\nabla u|^2 \, dx\right)^{\frac{N(q-2)}{4} + \frac b 2} c^{\frac p2 -\frac{N(p-2)}{4}- \frac b 2}.
$$
Since 
$$
\frac{N(q-2)}{4} + \frac b 2<1,
$$
then $m(c)>-\infty$ for any $0<c<c_1$. Furthermore, from \eqref{scaling1}, we know that $E(u_t)<0$ for any $t>0$ small enough. As a result, we obtain that $m(c)<0$ for any $0<c<c_1$. Let $Q_{p,b} \in H^1(\R^N)$ be the ground state to \eqref{equ1} with $p=2+2(2-b)/N$. Define
\begin{align}\label{defw}
w:=\frac{Q_{p,b}}{\|Q_{p,b}\|_2} c^{\frac 12} \in S(c).
\end{align}
Then we can conclude that $w_t \in S(c)$ and 
\begin{align*}
E(w_t)&=\frac{t^2 c}{2\|Q_{p,b}\|_2^2} \int_{\R^N} |\nabla Q_{p,b}|^2 \, dx-\frac{t^{\frac{N}{2}(q-2)+b}c^{\frac q 2}}{q\|Q_{p,b}\|_2^q} \int_{\R^N}|x|^{-b}|Q_{p,b}|^q \, dx-\frac{t^{2}c^{\frac p 2}}{p\|Q_{p,b}\|_2^p} \int_{\R^N}|x|^{-b}|Q_{p,b}|^{p} \, dx \\
&=\frac{t^2 c}{\|Q_{p,b}\|_2^2} \left(\frac 12 -\frac{C_{N,b}c^{\frac{2-b}{N}}}{p}\right)\int_{\R^N} |\nabla Q_{p,b}|^2 \, dx-\frac{t^{\frac{N}{2}(q-2)+b}c^{\frac q 2}}{q\|Q_{p,b}\|_2^q} \int_{\R^N}|x|^{-b}|Q_{p,b}|^q \, dx.
\end{align*}
This indicates that $E(w_t) \to -\infty$ as $t \to \infty$ for any $c \geq c_1$. Thus we derive that $m(c)=-\infty$ for any $c \geq c_1$. Applying the Lions concentration compactness principle and arguing as the proof of Theorem \ref{thm1}, we can get the compactness of any minimizing sequence to \eqref{gmin} in $H^1(\R^N)$. 

We next demonstrate the assertion $(\textnormal{ii})$. Utilizing \eqref{GN} with $p=2+2(2-b)/N$, we derive that, for any $u \in S(c)$,
$$
E(u) \geq \left(\frac 12 -\frac{C_{N,p}c^{\frac{2-b}{N}}}{p}\right) \int_{\R^N} |\nabla u|^2 \, dx +\frac 1 q \int_{\R^N}|x|^{-b}|u|^q \, dx.
$$
This results in $m(c) \geq 0$ for any $0<c \leq c_1$. On the other hand, we know that $E(u_t) \to 0$ as $t \to 0^+$ for any $c \geq 0$. Consequently, we have that $m(c)=0$ for any $0<c\leq c_1$. Reasoning as before, we can deduce 
$$
E(w_t)=\frac{t^2 c}{\|Q_p\|_2^2} \left(\frac 12 -\frac{C_{N,p}c^{\frac{2-b}{N}}}{p}\right)\int_{\R^N} |\nabla Q_p|^2 \, dx+\frac{t^{\frac{N}{2}(q-2)+b}c^{\frac q 2}}{q\|Q_p\|_2^q} \int_{\R^N}|x|^{-b}|Q_p|^q \, dx.
$$
It then follows that $m(c)=-\infty$ for any $c > c_1$. We now prove that \eqref{equ}-\eqref{mass} has no solutions for any $0<c \leq c_1$. If $u \in S(c)$ is a solution to \eqref{equ}-\eqref{mass} for some $0<c \leq c_1$, it then follows from Lemma \ref{ph} and \eqref{GN} that
\begin{align*}
\int_{\R^N} |\nabla u|^2 \, dx + \frac{N(q-2)+2b}{2q} \int_{\R^N}|x|^{-b}|u|^q \, dx&=\frac{N(p-2)+2b}{2p} \int_{\R^N}|x|^{-b}|u|^p \, dx \\
& \leq \frac{2C_{N,p}c^{\frac{2-b}{N}}}{p}\int_{\R^N} |\nabla u|^2 \, dx.
\end{align*}
We then reach a contradiction, because of $0<c \leq c_1$. Thus the proof is completed.
\end{proof}

\section{Mass supercritical case with focusing perturbation} \label{sup1}

In this section, we investigate the existence of solutions to \eqref{equ}-\eqref{mass} in the mass supercritical case $p>2+ 2(2-b)/N$ and $\mu=1$.

\begin{lem} \label{la}
Let ${N \geq 1}$, $2<q<p<2^*_b$, $0<b<\min\{2, N\}$ and $\mu=1$. If $u \in S(c)$ is a solution to the equation
\begin{align} \label{equ121}
-\Delta u + \lambda u= |x|^{-b}|u|^{q-2} u + |x|^{-b}|u|^{p-2} u \quad \mbox{in} \,\, \R^N,
\end{align}
then $\lambda>0$ for any $c>0$.
\end{lem}
\begin{proof}
Since $u \in H^1(\R^N)$ is a solution to \eqref{equ121}, then
$$
\int_{\R^N} |\nabla u|^2\,dx + \lambda \int_{\R^N} |u|^2 \,dx=\int_{\R^N} |x|^{-b}|u|^{q} \, dx+\int_{\R^N} |x|^{-b}|u|^{p} \, dx.
$$
On the other hand, from Lemma \ref{ph}, we get that $Q(u)=0$, i.e.
$$
\int_{\R^N}|\nabla u|^2\, dx=\frac {N(q-2) +2b}{2q} \int_{\R^N} |x|^{-b}|u|^q \, dx+\frac{N(p-2)+2b}{2p} \int_{R^N} |x|^{-b}|u|^p \, dx.
$$
Therefore, we conclude that
$$
 \lambda \int_{\R^N} |u|^2 \,dx=\frac{2(q-b)-N(q-2)}{2q}\int_{\R^N} |x|^{-b}|u|^q \, dx+\frac{2(p-b)-N(q-2)}{2p}\int_{\R^N} |x|^{-b}|u|^p \, dx>0,
$$
because of $2<q<p<2^*_b$ and $u \neq 0$. This completes the proof.
\end{proof}

\subsection{Focusing mass subcritical perturbation}
We devote this subsection to the study of the case $2<q<2+2(2-b)/N<p$ and $\mu=1$. The goal of this subsection is to prove Theorem \ref{thm3}.

\begin{lem} \label{cod}
Let ${N \geq 1}$, $2<q<2+2(2-b)/N<p$, $0<b<\min\{2, N\}$ and $\mu=1$, then there exists a constant $\tilde{c}_2>0$ such that $P_0(c)=\emptyset$ and $P(c)$ is a smooth manifold of codimension 2 in $H^1(\R^N)$ for any $0<c<\tilde{c}_2$, where $\tilde{c}_2>0$ is defined by
$$
\tilde{c}_2:=\left(\frac{2p\left(2(2-b)-N(q-2)\right)}{C_{N,p,b}N\left(N(p-2)+2b\right)(p-q)}\right)^{\frac{2(2-b)-N(q-2)}{(2-b)(p-q)}}\left(\frac{2q \left(N(p-2)-2(2-b)\right)}{C_{N,q,b}N\left(N(q-2)+2b\right)(p-q)}\right)^{\frac{N(p-2)-2(2-b)}{(2-b)(p-q)}}.
$$
\end{lem}
\begin{proof}
We argue by contradiction that $P_0(c) \neq \emptyset$ for some $0<c<\tilde{c}_2$. Thus, for any $u \in P_0(c)$, we have that $Q(u)=0$ and $\Psi(u)=0$. Therefore, we conclude that 
\begin{align*}
\int_{\R^N} |\nabla u|^2 \, dx&=\frac{N\left(N(p-2)+2b\right)(p-q)}{2p\left(2(2-b)-N(q-2)\right)}\int_{\R^N}|x|^{-b}|u|^p \, dx \\
& \leq \frac{C_{N,p,b}N\left(N(p-2)+2b\right)(p-q)}{2p\left(2(2-b)-N(q-2)\right)} \left(\int_{\R^N} |\nabla u|^2 \, dx\right)^{\frac{N(p-2)}{4} + \frac b 2} c^{\frac p2 -\frac{N(p-2)}{4} - \frac b 2},
\end{align*}
where we used \eqref{GN} for the inequality. This indicates that
\begin{align} \label{empty1}
\int_{\R^N} |\nabla u|^2 \, dx  \geq  \left(\frac{2p\left(2(2-b)-N(q-2)\right)}{C_{N,p,b}N\left(N(p-2)+2b\right)(p-q)}\right)^{\frac{4}{N(p-2)-2(2-b)}} c^{\frac{N(p-2)-2(p-b)}{N(p-2)-2(2-b)}}.
\end{align}
Correspondingly, there holds that
\begin{align*}
\int_{\R^N} |\nabla u|^2 \, dx&=\frac{N\left(N(q-2)+2b\right)(p-q)}{2q\left(N(p-2)-2(2-b)\right)}\int_{\R^N}|x|^{-b}|u|^q \, dx \\
&\leq \frac{C_{N,q,b}N\left(N(q-2)+2b\right)(p-q)}{2q\left(N(p-2)-2(2-b)\right)} \left(\int_{\R^N} |\nabla u|^2 \, dx\right)^{\frac{N(q-2)}{4} + \frac b 2} c^{\frac q2 -\frac{N(q-2)}{4} - \frac b 2}.
\end{align*}
This infers that
\begin{align}\label{empty2}
\int_{\R^N} |\nabla u|^2 \, dx \leq \left(\frac{C_{N,q,b}N\left(N(q-2)+2b\right)(p-q)}{2q \left(N(p-2)-2(2-b)\right)}\right)^{\frac{4}{2(2-b)-N(q-2)}} c^{\frac{2(q-b)-N(q-2)}{2(2-b)-N(q-2)}}.
\end{align}
Combining \eqref{empty1} and \eqref{empty2}, we then have that $c \geq \tilde{c}_2$.
This is impossible. Thus we get that $P_0(c)=\emptyset$ for any $0<c<\tilde{c}_2$. 

We now turn to prove the remaining assertion. Let $G(u):=\|u\|_2^2-c$, then 
$$
P(c)=\{u \in H^1(\R^N) : Q(u)=0, G(u)=0\}.
$$
Clearly, $Q$ and $G$ are of class $C^1$. To see that $P(c)$ is codimension $2$ in $H^1(\R^N)$, it suffices to verify that $(dQ, dG): P(c) \to \R^2$ is surjective for any $0<c<\tilde{c}_1$. We suppose by contradition that there exists a constant $\nu \in \R$ such that $dQ(u)=\nu dG(u)$ for some $u \in P(c)$ and $0<c<\tilde{c}_2$. This immediately indicates that $u \in H^1(\R^N)$ satisfies the equation
$$
-\Delta u + \nu u=\frac{N(q-2)+2b}{4} |x|^{-b} |u|^{q-2} u + \frac{N(p-2)+2b}{4} |x|^{-b} |u|^{p-2} u \quad \mbox{in} \,\,\R^N.
$$
Accordingly, we can conclude that $u$ enjoys the following Pohozaev identity,
\begin{align} \label{ph11}
\int_{\R^N} |\nabla u|^2 \, dx = \frac{\left(N(q-2)+2b\right)^2}{8q} \int_{\R^N}|x|^{-b}|u|^q \, dx+\frac{\left(N(p-2)+2b\right)^2}{8p} \int_{\R^N}|x|^{-b}|u|^p \, dx.
\end{align}
Since $Q(u)=0$, from \eqref{ph11}, then $\Psi(u)=0$. This leads to $u \in P_0(c)$ for $0<c<\tilde{c}_2$. We then reach a contradiction. Thus the proof is completed.
\end{proof}

\begin{lem} \label{locmp}
Let ${N \geq 1}$, $2<q<2+2(2-b)/N<p$, $0<b<\min\{2, N\}$ and $\mu=1$, then there exists a constant $\hat{c}_2>0$ such that, for any $0<c<\hat{c}_2$, there exists a constant $\rho>0$ such that
\begin{align} \label{loc}
\inf_{u \in V_{\rho}(c)} E(u) <0 \leq \inf_{\partial V_{\rho}(c)} E(u),
\end{align}
where $V_{\rho}(c):=\{u \in S(c) : \|\nabla u\|^2_2 < \rho\}$ and $\hat{c}_2>0$ is defined by 
$$
\hat{c}_2:=\left(\frac{q\left(N(p-2)-2(2-b)\right)}{2NC_{N,q,b}(p-q)}\right)^{\frac{N(p-2)-2(2-b)}{(p-q)(2-b)}}\left(\frac{p\left(2(2-b)-N(q-2)\right)}{2NC_{N,p,b}(p-q)}\right)^{\frac{2(2-b)-N(q-2)}{(p-q)(2-b)}}.
$$
\end{lem}
\begin{proof}
By means of \eqref{GN}, we derive that, for any $u \in S(c)$, 
\begin{align*}
E(u) &\geq \frac 12 \|\nabla u\|_2^2 -\frac{C_{N,q,b}}{q}\|\nabla u\|_2^{\frac{N(q-2)}{2} + b } c^{\frac q2 -\frac{N(q-2)}{4} - \frac b 2} -\frac{C_{N,p,b}}{p}\|\nabla u\|_2^{\frac{N(p-2)}{2} + b} c^{\frac p2 -\frac{N(p-2)}{4} - \frac b 2} \\
&=\|\nabla u\|_2^{\frac{N(q-2)}{2} + b}\left(\frac 12 \|\nabla u\|_2^{\frac{2(2-b)-N(q-2)}{2}}-\frac{C_{N,q,b}}{q}c^{\frac q2 -\frac{N(q-2)}{4}- \frac b 2} -\frac{C_{N,p,b}}{p}\|\nabla u\|_2^{\frac{N(p-q)}{2}} c^{\frac p2 -\frac{N(p-2)}{4} - \frac b 2} \right).
\end{align*}
For any $t>0$, we define
$$
g(t):=\frac 12 t^{\frac{2(2-b)-N(q-2)}{2}}-\frac{C_{N,q,b}}{q}c^{\frac{2(q-b)-N(q-2)}{4}}-\frac{C_{N,p,b}}{p}c^{\frac{2(p-b)-N(p-2)}{4}}t^{\frac{N(p-q)}{2}}.
$$
Since $2<q<2+2(2-b)/N<p$, then $g(0)<0$ and $g(t) \to -\infty$ as $t \to \infty$. In addition, we can compute that $g$ has only one maximum point $t_{max}\in (0, \infty)$ given by
$$
t_{max}:=\left(\frac{p\left(2(2-b)-N(q-2)\right)}{2NC_{N,p,b}(p-q)c^{\frac{2(p-b)-N(p-2)}{4}}}\right)^{\frac{2}{N(p-2)-2(2-b)}}
$$
such that
\begin{align*}
\max_{t>0} g(t)&=g(t_{max}) \\
&=\frac{N(p-2)-2(2-b)}{2N(p-q)}\left(\frac{p\left(2(2-b)-N(q-2)\right)}{2NC_{N,p,b}(p-q)c^{\frac{2(p-b)-N(p-2)}{4}}}\right)^{\frac{2(2-b)-N(q-2)}{N(p-2)-2(2-b)}}-\frac{C_{N,q,b}}{q}c^{\frac{2(q-b)-N(q-2)}{4}}.
\end{align*}
If $0<c<\hat{c}_2$, then $\max_{t>0} g(t)>0$. On the other hand, we see that $E(u_t)<0$ for any $t>0$ small enough. Therefore, it follows that, for any $0<c<\hat{c}_2$, there exists a constant $\rho>0$ such that \eqref{loc} holds. Thus the proof is completed.
\end{proof}

With the help of \eqref{loc}, we are now able to introduce the following local minimization problem, for any $0<c<\hat{c}_2$,
\begin{align} \label{lmin}
M(c):=\inf_{u \in V_{\rho}(c)} E(u)<0.
\end{align}

\begin{lem} \label{exist1}
Let ${N \geq 1}$, $2<q<2+2(2-b)/N<p$, $0<b<\min\{2, N\}$ and $\mu=1$, then any minimizing sequence to \eqref{lmin} is compact in $H^1(\R^N)$ up to translation for any $0<c<\hat{c}_2$. In particular, there exists a positive solution to \eqref{equ}-\eqref{mass} at the energy level $M(c)<0$ for any $0<c<\hat{c}_2$.
\end{lem}
\begin{proof}
To prove this, we shall adapt the Lions concentration compactness principle as before. Let $\{u_n\} \subset V_{\rho}(c)$ be a minimizing sequence to \eqref{lmin}. Due to $M(c)<0$, then vanishing cannot occur as the proof of Theorem \ref{thm1}. Moreover, by scaling technique, it is not hard to derive that
$$
M(c)<M(c-\bar{c})+M(\bar{c}),
$$
where $0<\bar{c}<\hat{c}_2$. This infers that dichotomy cannot occur, either. Therefore, we obtain that $\{u_n\}$ is compact in $H^1(\R^N)$ up to translations. Since $E(|u|) \leq E(u)$ for any $u \in S(c)$, then $M(c)$ is achieved by some nonnegative $u \in V_{\rho}(c)$ for $0<c<\hat{c}_2$. Via the maximum principle, we can deduce that $u$ is positive. Thus the proof is completed.
\end{proof}

\begin{lem} \label{unique1}
Let ${N \geq 1}$, $2<q<2+2(2-b)/N<p$, $0<b<\min\{2, N\}$ and $\mu=1$, then, for any $u \in S(c)$ and $0<c<\min\{\tilde{c}_2, \hat{c}_2\}$, the function $t \mapsto E(u_t)$ has exactly two critical points $0<t_{1, u}<t_{2,u}<\infty$ such that $u_{t_{1,u}} \in P_+(c)$, $u_{t_{2,u}} \in P_-(c)$ and
$$
E(u_{t_{1,u}})<0<E(u_{t_{2,u}})=\sup_{t>0} E(u_t).
$$
Moreover, the functions $u\mapsto t_{1,u}$ and $u\mapsto t_{2,u}$ are of class $C^1$.
\end{lem}
\begin{proof}
Note first that 
\begin{align}\label{scaling2}
E(u_t)=\frac{t^2}{2} \int_{\R^N} |\nabla u|^2 \,dx-\frac{t^{\frac{N}{2}(q-2)+b}}{q} \int_{\R^N}|x|^{-b}|u|^q \, dx -\frac{t^{\frac{N}{2}(p-2)+b}}{p} \int_{\R^N}|x|^{-b}|u|^p \, dx.
\end{align}
Since 
$$
\frac{N}{2}(q-2)+b<2<\frac{N}{2}(p-2)+b,
$$
then $E(u_t)<0$ for any $t >0$ small enough and $E(u_t)<0$ for any $t<0$ large enough. Furthermore, it follows from \eqref{GN} and \eqref{scaling2} that
\begin{align*}
E(u_t) & \geq \frac{t^2}{2} \|\nabla u\|_2^2-\frac{C_{N,q,b}t^{\frac{N}{2}(q-2)+b}}{q}\|\nabla u\|_2^{\frac{N(q-2)}{2} + b} c^{\frac q2 -\frac{N(q-2)}{4} - \frac b 2} \\
& \quad -\frac{C_{N,p,b}t^{\frac{N}{2}(p-2)+b}}{p}\|\nabla u\|_2^{\frac{N(p-2)}{2} + b} c^{\frac p2 -\frac{N(p-2)}{4} - \frac b 2} \\
&=\|\nabla u\|_2^{\frac{N(q-2)}{2} + b} t^{\frac{N}{2}(q-2)+b} \left(\frac{t^{(2-b)-\frac {N} {2}(q-2)}}{2}\|\nabla u\|_2^{(2-b)-\frac{N}{2} (q-2)} \right.\\
& \quad \left.-\frac{C_{N,q,b}}{q} c^{\frac q2 -\frac{N(q-2)}{4} - \frac b 2}-\frac{C_{N,p,b}t^{\frac{N}{2}(p-q)}}{p}\|\nabla u\|_2^{\frac{N(p-q)}{2}} c^{\frac p2 -\frac{N(p-2)}{4} - \frac b 2} \right).
\end{align*}
For any $t>0$, we define 
\begin{align*}
h_u(t)&:=\frac{t^{(2-b)-\frac {N} {2} (q-2)}}{2}\|\nabla u\|_2^{(2-b)-\frac {N} {2} (q-2)} -\frac{C_{N,q,b}}{q} c^{\frac q2 -\frac{N(q-2)}{4} - \frac b 2} \\
& \quad-\frac{C_{N,p,b}t^{\frac{N}{2}(p-q)}}{p}\|\nabla u\|_2^{\frac{N(p-q)}{2}} c^{\frac p2 -\frac{N(p-2)}{4} - \frac b 2}.
\end{align*}
Through simple calculations, we get that
\begin{align*}
\max_{t>0} h_u(t)&=\frac{N(p-2)-2(2-b)}{2N(p-q)}\left(\frac{p\left(2(2-b)-N(q-2)\right)}{2C_{N,p,b}N(p-q)c^{\frac{2(p-b)-N(p-2)}{4}}}\right)^{\frac{2(2-b)-N(q-2)}{N(p-2)-2(2-b)}} \\
& \quad-\frac{C_{N,q,b}}{q} c^{\frac{2(p-b)-N(q-2)}{4}}.
\end{align*}
If $0<c<\hat{c}_2$, then $\max_{t>0} h_u(t)>0$. This suggests that $\max_{t>0}E(u_t)>0$ for any $0<c<\hat{c}_2$. 
Observe that
\begin{align*}
\frac{d}{dt} E(u_t)&=t\int_{\R^N} |\nabla u|^2 \,dx-\frac{N(q-2)+2b}{2q}t^{\frac N 2 (q-2)+b-1}\int_{\R^N}|x|^{-b}|u|^q \, dx\\ 
&\quad-\frac{N(p-2)+2b}{2p}t^{\frac N 2 (p-2)+b-1}\int_{\R^N}|x|^{-b}|u|^p \, dx.
\end{align*}
Then we find that $\frac{d}{dt} E(u_t)<0$ for any $t >0$ small enough and $\frac{d}{dt} E(u_t)<0$ for any $t<0$ large enough. By \eqref{GN},  then
\begin{align*}
\frac{d}{dt} E(u_t) & \geq t \|\nabla u\|_2^2-\frac{C_{N,q,b}(N(q-2)+2b)t^{\frac{N}{2}(q-2)+b-1}}{2q}\|\nabla u\|_2^{\frac{N(q-2)}{2} + b} c^{\frac q2 -\frac{N(q-2)}{4} - \frac b 2} \\
& \quad -\frac{C_{N,p,b}(N(p-2)+b)t^{\frac{N}{2}(p-2)+b}-1}{2p}\|\nabla u\|_2^{\frac{N(p-2)}{2} + b} c^{\frac p2 -\frac{N(p-2)}{4} - \frac b 2} \\
&=2\|\nabla u\|_2^{\frac{N(q-2)}{2} + b} t^{\frac{N}{2}(q-2)+b-1} \left(\frac{t^{(2-b)-\frac {N} {2}(q-2)}}{2}\|\nabla u\|_2^{(2-b)-\frac{N}{2} (q-2)} \right.\\
& \quad \left.-\frac{C_{N,q,b}(N(q-2)+b)}{4q} c^{\frac q2 -\frac{N(q-2)}{4} - \frac b 2} \right.\\
& \quad \left.-\frac{C_{N,p,b}(N(p-2)+b)t^{\frac{N}{2}(p-q)}}{4p}\|\nabla u\|_2^{\frac{N(p-q)}{2}} c^{\frac p2 -\frac{N(p-2)}{4} - \frac b 2} \right).
\end{align*}
Similarly, we can compuet that $\max_{t>0}\frac{d}{dt} E(u_t)>0$ for any $0<c<\tilde{c}_2$, where the constant $\tilde{c}_2>0$ is given by Lemma \ref{cod}. It means that the function $t \mapsto \frac{d}{dt} E(u_t)$ has exactly two zeros on $(0, \infty)$.
Therefore, we are able to conclude that the function $t \mapsto E(u_t)$ has exactly two critical points $0<t_{1,u}<t_{2,u}<\infty$. Further, we can derive that $u_{t_{1,u}} \in P_+(c)$, $u_{t_{2,u}} \in P_-(c)$ satisfying 
$$
E(u_{t_{1,u}})<0<E(u_{t_{2,u}})=\sup_{t>0} E(u_t).
$$
Note that
$$
\frac{d}{dt} E(u_t) \mid_{t={t_{1,u}}}=\frac{d}{dt} E(u_t) \mid_{t={t_{2,u}}}=0, \quad \frac{d^2}{dt^2} E(u_t) \mid_{t={t_{1,u}}}>0, \quad \frac{d^2}{dt^2} E(u_t) \mid_{t={t_{2,u}}}<0,
$$
It then follows from the implicit function theorem that the functions $u\mapsto t_{1,u}$ and $u\mapsto t_{2,u}$ are of class $C^1$. Thus the proof is completed.
\end{proof}

\begin{prop} \label{prop1}
Let ${N \geq 1}$, $2<q<2+2(2-b)/N<p$, $0<b<\min\{2, N\}$ and $\mu=1$, then $M(c)=\sigma(c)$ for any $0<c<\min\{\bar{c}_2, \hat{c}_2\}.$, where
$$
\sigma(c):=\inf_{u \in P(c)} E(u).
$$
\end{prop}
\begin{proof}
Let $u \in S(c)$ be a minimizer to \eqref{lmin}, then $u$ is a solution to \eqref{equ}-\eqref{mass}. In view of Lemma \ref{ph}, then $u \in P(c)$. This yields that $\sigma(c) \leq M(c)$. On the other hand, for any $u \in P(c)$, it follows from Lemmas \ref{locmp} and \ref{unique1} that there exists a unique $t_{u_1}>0$ such that $u_{t_{1, u}} \in V_{\rho}(c)$ and $E(u) \geq E(u_{t_{1,u}})$. Therefore, we have that $\sigma(c) \geq M(c)$. Thus the proof is completed.
\end{proof}

From Lemma \ref{unique1}, we get that $P_-(c) \neq  \emptyset$ for any $0<c<\min\{\tilde{c}_2, \hat{c}_2\}$. As a consequence, we are able to introduce the following minimization problem, for any $0<c<c_2$, 
\begin{align} \label{min1}
\sigma_-(c):=\inf_{u\in P_-(c)}E(u).
\end{align}

\begin{lem} \label{coercive1}
Let ${\geq 1}$, $2<q<2+2(2-b)/N<p$, $0<b<\min\{2, N\}$ and $\mu=1$, then there exists a constant $c_2>0$ such that $\sigma_-(c)>0$ and $E$ restricted on $P_-(c)$ is coercive for any $0<c<c_2$, where $c_2>0$ is defined by
\begin{align*}
c_2:= \left(\frac{2p\left(2(2-b)-N(q-2)\right)}{C_{N,p,b}N\left(N(p-2)+2b\right)(p-q)}\right)^{\frac{2(2-b)-N(q-2)}{(p-q)(2-b)}}\left(\frac{q \left(N(p-2)-2(2-b)\right)}{2C_{N,q,b} N(p-q)} \right)^{\frac{N(p-2)-2(b-2)}{(p-q)(2-b)}}.
\end{align*}
Moreover, there holds that $c_2<\min\{\tilde{c}_2, \hat{c}_2\}$.
\end{lem}
\begin{proof}
For any $u \in P_-(c)$, we know that $Q(u)=0$. This results in 
\begin{align} \label{q1} \nonumber
E(u)&=E(u)-\frac{2}{N(p-2)+2b} Q(u)\\
&=\frac{N(p-2)-2(2-b)}{2\left(N(p-2)+2b\right)} \int_{\R^N} |\nabla u|^2 \,dx - \frac{N(p-q)}{q\left(N(p-2)+2b\right)} \int_{\R^N} |x|^{-b}|u|^{q} \, dx \\ \nonumber
& \geq  \|\nabla u\|_2^{\frac{N(q-2)+2b}{2}} \Bigg(\frac{N(p-2)-2(2-b)}{2\left(N(p-2)+2b\right)} \|\nabla u\|_2^{\frac{2(2-b)-N(q-2)}{2}} -\frac{C_{N,q,b}N(p-q)}{q\left(N(p-2)+2b\right)} c^{\frac{2(q-b)-N(q-2)}{4}} \Bigg),
\end{align}
where we used \eqref{GN} for the inequality. For any $u \in P_-(c)$, we have that $Q(u)=0$ and $\Psi(u)<0$. Arguing as the proof of \eqref{empty1}, we then conclude that
\begin{align*}
\int_{\R^N} |\nabla u|^2 \, dx  \geq  \left(\frac{2p\left(2(2-b)-N(q-2)\right)}{C_{N,p}N\left(N(p-2)+2b\right)(p-q)}\right)^{\frac{4}{N(p-2)-2(2-b)}} c^{\frac{N(p-2)-2(p-b)}{N(p-2)-2(2-b)}}.
\end{align*}
It then follows from \eqref{q1} that
\begin{align*}
E(u)& \geq \|\nabla u\|_2^{\frac{N(q-2)+2b}{2}} \Bigg(\frac{N(p-2)-2(2-b)}{2\left(N(p-2)+2b\right)} \left(\frac{2p\left(2(2-b)-N(q-2)\right)}{C_{N,p}N\left(N(p-2)+2b\right)(p-q)}\right)^{\frac{2(2-b)-N(q-2)}{N(p-2)-2(2-b)}} \\
& \qquad -\frac{C_{N,q,b}N(p-q)}{q\left(N(p-2)+2b\right)} c^{\frac{(p-q)(2-b)}{N(p-2)-2(2-b)}} \Bigg)c^{\frac{\left(N(p-2)-2(p-b)\right)\left(2(2-b)-N(q-2)\right)}{4\left(N(p-2)-2(2-b)\right)}} .
\end{align*}
This yields that $\sigma_-(c)>0$ and $E$ restricted on $P_-(c)$ is coercive for any $0<c<c_2$. Due to $2<q<2+2(2-b)/N<p$, then $c_2<\min\{\tilde{c}_2, \hat{c}_2\}$. Thus the proof is completed.
\end{proof}

\begin{defi}\label{homotopy} \cite[Definition 3.1]{Gh}
Let $B$ be a closed subset of a set $Y \subset H^1(\R^N)$. We say that a class $\mathcal{G}$ of compact subsets of $Y$ is a homotopy stable family with the closed boundary $B$ provided that
\begin{enumerate}
\item [\textnormal{(i)}] Every set in $\mathcal{G}$ contains $B$.
\item [\textnormal{(ii)}] For any $A \in \mathcal{G}$ and $\eta \in C([0, 1] \times Y, Y)$ satisfying $\eta(t, x)=x$ for all $(t, x) \in (\{0\} \times Y) \cup([0, 1] \times B)$, then $\eta(\{1\} \times A) \in \mathcal{G}$.
\end{enumerate}
\end{defi}

\begin{lem}\label{ps}
Let ${N \geq 1}$, $2<q<2+2(2-b)/N<p$, $0<b<\min\{2, N\}$ and $\mu=1$. Let $\mathcal{G}$ be a homotopy stable family of compact subsets of $S(c)$ with closed boundary $B$. Define
\begin{align} \label{ming}
\sigma_{\mathcal{G}}(c):=\inf_{A\in \mathcal{G}}\max_{u \in A} F(u),
\end{align}
where $F(u):=E(u_{t_{2,u}})=\max_{t>0}E(u_t)$ and $0<c<c_2$, where $c_2>0$ is the constant given in Lemma \ref{coercive1}. Suppose that $B$ is contained on a connected component of $P_-(c)$ and $\max\{\sup F(B), 0\}<\sigma_{\mathcal{G}}(c)<\infty$. Then there exists a Palais-Smale sequence $\{u_n\} \subset P_-(c)$ satisfying $(u_n)^-=o_n(1)$ for $E$ restricted on $S(c)$ at the level $\sigma_{\mathcal{G}}(c)$, where $u^-$ denotes the negative part of $u$ defined by $u^-(x):=\max \{-u(x), 0\}$ for $x \in \R^N$.
\end{lem}
\begin{proof}
The proof of this lemma is inspired by elements presented in \cite{BS1, BS2}. Let us first define a function $\eta: [0,1] \times S(c) \to S(c)$ by $\eta(s, u)=u_{1-s+st_{2,u}}$. In view of Lemma \ref{unique1}, we know that, for any $u \in P_-(c)$, $t_{2,u}=1$.  Notice that $B \subset P_-(c)$, then $\eta(s ,u)=u$ for any $(s, u) \in (\{0\} \times S(c)) \cup([0, 1] \times B)$. In addition, from Lemma \ref{unique1}, we have that $\eta \in C([0,1] \times S(c), S(c))$. Let us now suppose that $\{A_n\} \subset \mathcal{G}$ is a minimizing sequence to \eqref{ming}. 
Note that
$$
F(|u|)=\max_{t>0} E(|u|_t) \leq \max E(u_t)=F(u).
$$
Then, without restriction, we may assume that $A_n$ is nonnegative. In light of Definition \ref{homotopy}, we then get that
$$
D_n:=\eta(\{1\} \times A_n)=\{u_{t_{2,u}} : u \in A_n\} \in \mathcal{G}.
$$
Record that $D_n \subset P_-(c)$, then 
$$
\max_{v \in D_n}F(v)=\max_{u \in A_n}F(u).
$$ 
This readily suggests that there exists another minimizing sequence $\{D_n\} \subset P_-(c)$ to \eqref{ming}. Applying \cite[Theorem 3.2]{Gh}, we then deduce that there exists a Palais-Smale sequence $\{\tilde{u}_n\} \subset S(c)$ for $F$ at the level $\sigma_{\mathcal{G}}(c)$ such that $(\tilde{u}_n)^-=o_n(1)$ and
\begin{align} \label{dist}
\mbox{dist}_{H^1}(\tilde{u}_n, D_n)=o_n(1).
\end{align}
Note that $D_n$ is compact for any $n\in \N$, then there exists $w_n \in D_n$ such that 
\begin{align} \label{achieve}
\mbox{dist}_{H^1}(\tilde{u}_n, D_n)=\|\tilde{u}_n-w_n\|=o_n(1).
\end{align}

For simplicity, we shall write $t_n=t_{2, \tilde{u}_n}$ and $u_n=(\tilde{u}_n)_{t_n}$ in what follows. As an immediate consequence, we have that $(u_n)^-=o_n(1)$. Next we claim that there exists a constant $C>0$ such that $1/C \leq t_n \leq C$. Indeed, observe first that
$$
t_n^2=\frac{\int_{\R^N} |\nabla u_n|^2 \, dx}{\int_{\R^N} |\nabla \tilde{u}_n|^2 \, dx}.
$$
Since $E(u_n)=F(\tilde{u}_n)=\sigma_{\mathcal{G}}(c)+o_n(1)$, $\{u_n\} \subset P_-(c)$ and $0<\sigma_{\mathcal{G}}(c)<\infty$, from Lemma \ref{coercive1}, it then follows that there exists a constant $C_1>0$ such that $1/C_1 \leq \|u_n\|\leq C_1$. On the other hand, since $\{D_n\} \subset P_-(c)$ is a minimizing sequence to \eqref{ming}, from Lemma \ref{coercive1}, it then yields that $\{D_n\}$ is bounded in $H^1(\R^N)$. Thanks to \eqref{dist}, then $\{\tilde{u}_n\}$ is bounded in $H^1(\R^N)$, i.e. there exists a constant $C_2>0$ such that $\|\tilde{u}_n\| \leq C_2$. 
Utilizing the assumption $\sigma_{\mathcal{G}}(c)>0$, we can get that $\|w_n\| \geq 1/C_2$. By \eqref{achieve}, then 
$$
\|\tilde{u}_n\| \geq \|w_n\|-\|\tilde{u}_n-w_n\|\geq \frac{1}{C_2}+o_n(1).
$$
Therefore, the claim follows.

We now demonstrate that $\{u_n\} \subset P_-(c)$ is a Palais-Smale sequence for $E$ restricted on $S(c)$ at the level $\sigma_{\mathcal{G}}(c)$. We denote by $\|\cdot\|_{*}$ the dual norm of $(T_u S(c))^*$. Accordingly, there holds that
\begin{align*}
\|dE(u_n)\|_*=\sup_{\psi \in T_{u_n}S(c), \|\psi\|\leq 1}|dE(u_n)[\psi]|=\sup_{\psi \in T_{u_n}S(c), \|\psi\|\leq 1} |dE(u_n)[(\psi_{\frac{1}{t_n}})_{t_n}]|.
\end{align*}
From straightforward calculations, we are able to derive that the mapping $T_uS(c) \to T_{u_{t_u}}S(c)$ defined by $\psi \mapsto \psi_{t_u}$ is an isomorphism. Furthermore, it is simple to check that $dF(u)[\psi]=dE(u_{t_{2,u}})[\psi_{t_{2,u}}]$ for any $u \in S(c)$ and $\psi \in T_uS(c)$. As a result, we have that
$$
\|dE(u_n)\|_*=\sup_{\psi \in T_{u_n}S(c), \|\psi\| \leq 1}|dF(\tilde{u}_n)[\psi_{\frac{1}{t_n}}]|.
$$
Since $\{\tilde{u}_n\} \subset S(c)$ is a Palais-Smale sequence for $F$ at the level $\sigma_{\mathcal{G}}(c)$, we then  apply the claim to deduce that $\{u_n\} \subset P_-(c)$ is a Palais-Smale sequence for $E$ restricted on $S(c)$ at the level $\sigma_{\mathcal{G}}(c)$. Thus the proof is completed.
\end{proof}

\begin{lem} \label{pss1}
Let ${N \geq 1}$, $2<q<2+2(2-b)/N<p$, $0<b<\min\{2, N\}$ and $\mu=1$, then there exists a Palais-Smale sequence $\{u_n\} \subset P_-(c) $ satisfying $(u_n)^-=o_n(1)$ for $E$ restricted on $S(c)$ at the level $\sigma_-(c)$ for any $0<c<c_2$. 
\end{lem}
\begin{proof}
Let $B=\emptyset$ and let $\mathcal{G}$ be all singletons in $S(c)$. In virtue of \eqref{ming}, we then have that
$$
\sigma_{\mathcal{G}}(c)=\inf_{u \in S(c)} \sup_{t \geq 0} E(u_t).
$$
We are now going to prove that $\sigma_{\mathcal{G}}(c)=\sigma_-(c)$. From Lemma \ref{unique1}, we know that, for any $u \in S(c)$, there exists a unique $t_{2,u}>0$ such that $u_{t_{2,u}} \in P_-(c)$ and $E(u_{t_{2,u}})=\max_{t \geq 0}E(u_t)$. This then implies that
$$
\inf_{u \in S(c)} \sup_{t \geq 0} E(u_t)  \geq \inf_{u \in P_-(c)} E(u).
$$
On the other hand, for any $u \in P_-(c)$, by Lemma \ref{unique1}, we derive that $E(u)=\max_{t \geq 0}E(u_t)$. This then gives that
$$
\inf_{u \in S(c)} \sup_{t \geq 0} E(u_t)  \leq \inf_{u \in P_-(c)} E(u).
$$
Therefore, we have that $\sigma_{\mathcal{G}}(c)=\sigma_-(c)$. It then follows from Lemma \ref{ps} that the result of this lemma holds and the proof is completed.
\end{proof}

\begin{lem} \label{exist2}
Let ${N \geq 1}$, $2<q<2+2(2-b)/N<p$, $0<b<\min\{2, N\}$ and $\mu=1$, then there exists a positive solution to \eqref{equ}-\eqref{mass}  at the level $\sigma(c)>0$ for any $0<c<c_2$.
\end{lem}
\begin{proof}
In view of Lemma \ref{pss1}, then there exists a Palais-Smale sequence $\{u_n\} \subset P_-(c)$ for $E$ restricted on $S(c)$ at the level $\sigma_-(c)$ for any $0<c<c_2$. From Lemma \ref{coercive1}, then $\{u_n\}$ is bounded in $H^1(\R^N)$. Therefore, there exists $u \in H^1(\R^N)$ such that $u_n \wto u$ in $H^1(\R^N)$ as $n \to \infty$. Furthermore, by Lemma \ref{cembedding}, then $u_n \to u$ in $L^p(\R^N, |x|^{-b} dx)$ and $L^q(\R^N, |x|^{-b} dx)$ as $n \to \infty$. As a consequence, then $u \neq 0$. If not, by $Q(u_n)=0$, then $\sigma_-(c)=0$. This is impossible, see Lemma \ref{coercive1}. {Since $\{u_n\}\subset P_-(c)$ is a Palais-Smale sequence for $E$ restricted on $S(c)$ at the level $\sigma_-(c)$, arguing as the proof of \cite[Lemma 3]{BeLi}, we then have that} 
\begin{align} \label{equ11}
-\Delta u_n + \lambda_n u_n= |x|^{-b}|u_n|^{q-2} u_n + |x|^{-b}|u_n|^{p-2} u_n +o_n(1) \quad \mbox{in} \,\, \R^N,
\end{align}
where 
$$
\lambda_n:=\frac 1 c \left(\int_{\R^N} |x|^{-b}|u_n|^q \, dx + \int_{\R^N} |x|^{-b}|u_n|^p \, dx-\int_{\R^N} |\nabla u|^2 \, dx\right).
$$
Then $\{\lambda_n\} \subset \R$ is bounded, because $\{u_n\}$ is bounded in $H^1(\R^N)$. As a result, there exists $\lambda \in \R$ such that $\lambda_n \to \lambda$ in $\R$ as $n \to \infty$. Note that $u_n \wto u$ in $H^1(\R^N)$ as $n \to \infty$, from \eqref{equ11}, then $u$ satisfies the equation
\begin{align} \label{equ12}
-\Delta u + \lambda u= |x|^{-b}|u|^{q-2} u + |x|^{-b}|u|^{p-2} u \quad \mbox{in} \,\, \R^N.
\end{align}
Further, we have that $\lambda>0$ by Lemma \ref{la}. Using Lemma \ref{cembedding} and \eqref{equ11}-\eqref{equ12}, we then conclude that $u_n \to u$ in $H^1(\R^N)$ as $n \to \infty$. Furthermore, notice that $(u_n)^-=o_n(1)$, then $u$ is nonnegative. By the maximum principle, then $u$ is positive. Thus proof is completed.
\end{proof}

\begin{proof}[Proof of Theorem \ref{thm3}]
As a consequence of Theorems \ref{exist1} and \ref{exist2}, we derive the existence of two positive solutions. One is a local minimizer to \eqref{lmin} and the other is a mountain pass type solution as the level $\sigma_-(c)$. Invoking Lemma \ref{radial}, we now have that the solutions are radially symmetric and decreasing in the radial direction. This completes the proof.
\end{proof}

\subsection{Focusing mass critical perturbation} We next consider the case $2<q=2+2(2-b)/N<p$ and $mu=1$. In this subsection, we are going to establish Theorem \ref{thm4}.

\begin{lem} 
Let ${N \geq 1}$, $2<q=2+2(2-b)/N<p$, $0<b<\min\{2, N\}$ and $\mu=1$, then $P_0(c)=\emptyset$ and $P(c)$ is a smooth manifold of codimension $2$ in $H^1(\R^N)$ for any $c>0$.
\end{lem}
\begin{proof}
We only prove that $P_0(c)=\emptyset$ for any $c>0$. The rest proof can be done by following the ideas of that of Lemma \ref{cod}. Argue by contradiction that $P_0(c) \neq \emptyset$ for some $c>0$, then there would exist $u \in P_0(c)$, i.e. $Q(u)=\Psi(u)=0$. Since $N(q-2)+2b=4$, we then deduce that
$$
\int_{\R^N} |x|^{-b}|u|^p \, dx=0.
$$
This is impossible. Thus we have completed the proof.
\end{proof}

\begin{lem}\label{unique2}
Let ${N \geq 1}$, $2<q=2+2(2-b)/N<p$, $0<b<\min\{2, N\}$ and $\mu=1$, then, for any $u \in S(c)$, there exists a unique $t_u>0$ such that $u_{t_u} \in P(c)$ and $E(u_t)=\max_{t \geq 0} E(u_t)$ for any $0<c<c_1$, where $c_1>0$ is the constant decided in Theorem \ref{thm1}. Moreover, the function $t \mapsto E(u_t)$ is concave on $[t_u, \infty)$ and the function $u \mapsto t_u$ is of class $C^1$.
\end{lem}
\begin{proof}
By using \eqref{scaling1} and \eqref{GN} , we first obtain that, for any $u \in S(c)$,
\begin{align*}
E(u_t) &=\frac{t^2}{2} \int_{\R^N} |\nabla u|^2 \,dx-\frac{t^2}{q} \int_{\R^N}|x|^{-b}|u|^q \, dx -\frac{t^{\frac{N}{2}(p-2)+b}}{p} \int_{\R^N}|x|^{-b}|u|^p \, dx\\
&\geq \frac {t^2}{2} \left(1-\frac{NC_{N,b}c^{\frac{2-b}{N}}}{N+2-b}\right) \int_{\R^N} |\nabla u|^2 \,dx -\frac {t^{\frac N 2 (p-2)+b}}{p}\int_{\R^N} |x|^{-b}|u|^p \,dx.
\end{align*}
Due to $p>2+2(2-b)/N$ and $0<c<c_1$, then $E(u_t)>0$ for any $t>0$ small enough. In addition, we have that $E(u_t)<0$ for any $t>0$ large enough. 
Observe that 
$$
\frac{d}{dt}E(u_t)=t\left(\int_{\R^N} |\nabla u|^2 \,dx-\frac{2}{q}\int_{\R^N}|x|^{-b}|u|^q \, dx \right)-\frac{N(p-2)+2b}{2p}t^{\frac{N}{2}(p-2)+b-1}\int_{\R^N}|x|^{-b}|u|^p \, dx.
$$
and
$$
\frac 12 \int_{\R^N} |\nabla u|^2 \,dx-\frac{1}{q} \int_{\R^N}|x|^{-b}|u|^q \, dx>0,
$$
because of $0<c<c_1$. Then we see there exists a unique $t_u>0$ such that $\frac{d}{dt} E(u_t)\mid_{t=t_u}=0$, where $t_u>0$ is given by
\begin{align} \label{deftu1}
t_u:=\left(\frac{2pq\int_{\R^N} |\nabla u|^2 \,dx-4p\int_{\R^N}|x|^{-b}|u|^q \, dx}{q(N(p-2)+2b)\int_{\R^N}|x|^{-b}|u|^p \, dx}\right)^{\frac{N(p-2)-2(2-b)}{2}}.
\end{align}
Furthermore, we find that $\frac{d}{dt} E(u_t)>0$ for any $0<t<t_u$ and $\frac{d}{dt} E(u_t)<0$ for any $t>t_u$. Therefore, the function $t \mapsto E(u_t)$ has a unique critical point $t_u>0$ such that $u_{t_u} \in P(c)$ and $E(u_t)=\max_{t \geq 0} E(u_t)$. Note that
\begin{align*}
\frac{d^2}{dt^2}E(u_t) &=\int_{\R^N} |\nabla u|^2 \,dx-\frac{2}{q}\int_{\R^N}|x|^{-b}|u|^q \, dx \\
& \quad -\frac{(N(p-2)+2b)((N(p-2)-2(1-b))}{4p}t^{\frac{N}{2}(p-2)+b-2}\int_{\R^N}|x|^{-b}|u|^p \, dx.
\end{align*}
Hence we have that $\frac{d^2}{dt^2}E(u_t)<0$ for any $t>t_u$. In view of \eqref{deftu1}, then the function $u \mapsto t_u$ is of class $C^1$. This completes the proof.
\end{proof}

As a result of Lemma \ref{unique2}, we see that $P(c) \neq \emptyset$ for any $0<c<c_1$. Let us now consider the following minimization problem, for any $0<c<c_1$,
\begin{align}\label{min2}
\sigma(c):=\inf_{u \in P(c)} E(u).
\end{align}

\begin{lem}\label{coercive2}
Let ${N \geq 1}$, $2<q=2+2(2-b)/N<p$, $0<b<\min\{2, N\}$ and $\mu=1$, then $\sigma(c)>0$ and $E$ restricted on $P(c)$ is coercive for any $0<c<c_1$.
\end{lem}
\begin{proof}
For any $u \in P(c)$, we have that $Q(u)=0$. This gives that
\begin{align} \label{c2}
E(u)=E(u)-\frac 12 Q(u)=\frac{N(p-2)-2(2-b)}{4p} \int_{\R^N} |x|^{-b} |u|^p \,dx.
\end{align}
Furthermore, utilizing \eqref{GN}, we obtain that if $Q(u)=0$, then
\begin{align} \label{c22}
\begin{split}
\int_{\R^N} |\nabla u|^2 \, dx &=\frac{2}{q} \int_{\R^N}|x|^{-b}|u|^q \, dx+\frac{N(p-2)+2b}{2p} \int_{\R^N}|x|^{-b}|u|^p \, dx \\
& \leq \frac{2C_{N,b} c^{\frac{2-b}{N}}}{q}\int_{\R^N} |\nabla u|^2 \, dx + \frac{N(p-2)+2b}{2p} \int_{\R^N}|x|^{-b}|u|^p \, dx.
\end{split}
\end{align}
Since $0<c<c_1$, then
\begin{align} \label{c21}
0<\frac{2p\left(q-2 C_{N,b} c^{\frac{2-b}{N}}\right)}{q\left(N(p-2)+2b\right)} \int_{\R^N} |\nabla u|^2 \, dx \leq \int_{\R^N}|x|^{-b}|u|^p \, dx.
\end{align}
Making use of \eqref{GN} and \eqref{c22}, we can derive that
\begin{align*}
\int_{\R^N} |\nabla u|^2 \, dx & \leq \frac{2C_{N,b} c^{\frac{2-b}{N}}}{q}\int_{\R^N} |\nabla u|^2 \, dx \\
& \quad +\frac{C_{N,p,b}\left(N(p-2)+2b\right)}{2p}\left(\int_{\R^N} |\nabla u|^2 \, dx \right)^{\frac{N(p-2)+2b}{4}} c^{\frac p2 -\frac{N(p-2)+2b}{4}},
\end{align*}
from which we then get that
\begin{align} \label{c23}
\int_{\R^N} |\nabla u|^2\,dx \geq \left(\frac{2p\left(q-2C_{N,b}c^{\frac{2-b}{N}}\right)}{qC_{N,p,b}\left(N(p-2)+2b\right)}\right)^{\frac{4}{N(p-2)-2(2-b)}}c^{\frac{N(p-2)-2(p-b)}{N(p-2)-2(2-b)}}.
\end{align}
Taking into account \eqref{c2}, \eqref{c22} and \eqref{c23}, we then get that $\sigma(c)>0$ for any $0<c<c_1$. From \eqref{c2} and \eqref{c21}, we then have that $E$ restricted on $P(c)$ is coercive for any $0<c<c_1$. This completes the proof.
\end{proof}

\begin{lem} \label{pss2}
Let ${N \geq 1}$, $2<q=2+2(2-b)/N<p$, $0<b<\min\{2, N\}$ and $\mu=1$, then there exists a Palais-Smale sequence $\{u_n\} \subset P(c)$ satisfying $(u_n)^-=o_n(1)$ for $E$ restricted on $S(c)$ at the level $\sigma(c)$ for any $0<c<c_1$.
\end{lem}
\begin{proof}
Applying Lemmas \ref{unique2} and \ref{coercive2} and arguing as the proof of Lemmas \ref{ps} and \ref{pss1} through replacing the role of $P_-(c)$ by $P(c)$, we are able to finish the proof.
\end{proof}

\begin{proof}[Proof of Theorem \ref{thm4}]
Using Lemmas \ref{coercive2} and \ref{pss2} and following the way of the proof of Lemma \ref{exist2}, we can complete the proof.
\end{proof}

\subsection{Focusing mass supcritical perturbation} Let us now turn to investigate the case $2<2+2(2-b)/N<q<p$ and $\mu=1$ and present the proof of Theorem \ref{thm41}.

\begin{lem} 
Let ${N \geq 1}$, $2+2(2-b)/N<q<p$, $0<b<\min\{2, N\}$ and $\mu=1$, then $P_0(c)=\emptyset$ and $P(c)$ is a smooth manifold of codimension $2$ in $H^1(\R^N)$ for any $c>0$.
\end{lem}
\begin{proof}
For simplicity, we only verify that $P_0(c)=\emptyset$. If $u \in P_0(c)$, then $Q(u)=\Psi(u)=0$. This readily yields that
\begin{align*}
&\frac{N(q-2)+2b}{2q} \left(\frac{N(q-2)+2(b-1)}{2}-1\right) \int_{\R^N}|x|^{-b}|u|^q \, dx\\
& \quad +\frac{N(p-2)+2b}{2p} \left(\frac{N(p-2)+2(b-1)}{2}-1\right) \int_{\R^N}|x|^{-b}|u|^p \, dx=0.
\end{align*}
Since $p>q>2$, $b>0$ and
$$
\frac{N(p-2)+2(b-1)}{2}>\frac{N(p-2)+2(b-1)}{2}>1,
$$
then $u=0$. We the reach a contradiction. It in turn yields that $P_0(c)=\emptyset$ for any $c>0$. This completes the proof.
\end{proof}

\begin{lem}\label{unique3}
Let ${N \geq 1}$, $2+2(2-b)/N<q<p$, $0<b<\min\{2, N\}$ and $\mu=1$, then, for any $u \in S(c)$, there exists a unique $t_u>0$ such that $u_{t_u} \in P(c)$ and $E(u_t)=\max_{t \geq 0} E(u_t)$ for any $c>0$. Moreover, the function $t \mapsto E(u_t)$ is concave on $[t_u, \infty)$ and the function $u \mapsto t_u$ is of class $C^1$.
\end{lem}
\begin{proof}
In virtue of \eqref{scaling1}, we know that $E(u_t)>0$ for any $t>0$ small enough and $E(u_t)<0$ for any $t>0$ large enough. Due to $2+2(2-b)/N<q<p$, then there exists a unique $t_u>0$ such that $\frac{d}{dt} E(u_t)\mid_{t=t_u}=0$. Furthermore, there holds that $\frac{d}{dt}E(u_t)>0$ for any $0<t<t_u$ and $\frac{d}{dt}E(u_t)<0$ for any $t>t_u$. As a consequence, we can conclude that $u_{t_u} \in P(c)$ and $\max_{t \geq 0} E(u_t)=E(u_{t_u})$. The smoothness of the function $u \mapsto t_u$ benefits directly from the implicit function theorem as the proof of Lemma \ref{unique1} . Thus the proof is completed.
\end{proof}

In this subsection, we introduce the following minimization problem, for any $c>0$,
\begin{align}\label{min3}
\sigma(c):=\inf_{u \in P(c)} E(u).
\end{align}

\begin{lem}\label{coercive31}
Let ${N \geq 1}$, $2+2(2-b)/N<q<p$, $0<b<\min\{2, N\}$ and $\mu=1$, then $\sigma(c)>0$ and $E$ restricted on $P(c)$ is coercive for any $c>0$.
\end{lem}
\begin{proof}
For any $u \in P(c)$, there holds that $Q(u)=0$. It then follows that
\begin{align} \label{c3}
\begin{split}
E(u)&=E(u)-\frac 12 Q(u)\\
&=\frac{N(q-2)-2(2-b)}{4q} \int_{\R^N} |x|^{-b}|u|^q \,dx + \frac{N(p-2)-2(2-b)}{4p} \int_{\R^N} |x|^{-b}|u|^p \,dx.
\end{split}
\end{align}
Since $Q(u)=0$, by \eqref{GN}, then
\begin{align} \label{c3111}
\begin{split}
\int_{\R^N} |\nabla u|^2 \,dx & =\frac{N(q-2)+2b}{2q} \int_{\R^N} |x|^{-b}|u|^q \,dx + \frac{N(p-2)+2b}{2p} \int_{\R^N} |x|^{-b}|u|^p \,dx \\
& \leq \frac{C_{N,q,b}\left(N(q-2)+2b\right)}{2q}\left(\int_{\R^N} |\nabla u|^2 \, dx \right)^{\frac{N(q-2)+2b}{4}} c^{\frac q2 -\frac{N(q-2)+2b}{4}} \\
& \quad + \frac{C_{N,p,b}\left(N(p-2)+2b\right)}{2p}\left(\int_{\R^N} |\nabla u|^2 \, dx \right)^{\frac{N(p-2)+2b}{4}} c^{\frac p2 -\frac{N(p-2)+2b}{4}}.
\end{split}
\end{align}
Note that $2+2(2-b)/N<q<p$, from \eqref{c3111}, then there exists a constant $C>0$ such that $\|\nabla u\|_2 \geq C$. Due to $Q(u)=0$, then
$$
\frac{N(q-2)+2b}{2q} \int_{\R^N} |x|^{-b}|u|^q \,dx + \frac{N(p-2)+2b}{2p} \int_{\R^N} |x|^{-b}|u|^p \,dx \geq C.
$$
This together with \eqref{c3} implies that $\sigma(c)>0$ for any $c>0$. Since $Q(u)=0$, by \eqref{c3}, we then derive that $E$ restricted on $P(c)$ is coercive for any $c>0$. Thus the proof is completed.
\end{proof}

\begin{lem} \label{pss3}
Let ${N \geq 1}$, $2+2(2-b)/N<q<p$, $0<b<\min\{2, N\}$ and $\mu=1$, then there exists a Palais-Smale sequence $\{u_n\} \subset P(c)$ satisfying $(u_n)^-$ for $E$ restricted on $S(c)$ at the level $\sigma(c)$ for any $c>0$.
\end{lem}
\begin{proof}
To prove this, one can closely follow the strategy of the proof of Lemma \ref{pss1}. 
\end{proof}

\begin{proof}[Proof of Theorem \ref{thm41}]
Using Lemmas \ref{coercive31} and \ref{pss3} and following the way of the proof of Lemma \ref{exist2}, we are able to complete the proof.
\end{proof}

\section{Mass supercritical case with defocusing perturbation} \label{sup2}

In this section, we study the existence of solutions to \eqref{equ}-\eqref{mass} in the mass supercritical case $p>2+ 2(2-b)/N$ and $\mu=-1$.


\begin{lem} \label{cod11}
Let ${N \geq 1}$, $2<q<p$, $p>2+2(2-b)/N$, $0<b<\min\{2, N\}$ and $\mu=-1$, then $P(c)$ is a smooth manifold of codimension $2$ in $H^1(\R^N)$ for any $c>0$.
\end{lem}
\begin{proof}
Let $G(u):=\|u\|_2^2-c$, then
$$
P(c)=\{u\in H^1(\R^N) : G(u)=0, Q(u)=0\}.
$$
The smoothness of $P(c)$ is obviuos. To show that $P(c)$ is codimension $2$ in $H^1(\R^N)$, it suffices to verify that $(dG, dQ) : P(c) \to \R^2$ is surjective for any $c>0$. If not, then there exist $u \in P(c)$ and $\nu \in \R^N$ such that $dQ(u)=\nu G(u)$ for some $c>0$. Therefore, $u$ satisfies the equation
$$
-\Delta u + \nu u =-\frac{N(q-2)+2b}{4} |x|^{-b} |u|^{q-2} u + \frac{N(p-2)+2b}{4} |x|^{-b} |u|^{p-2} u \quad \mbox{in} \,\,\R^N.
$$
In the spirit of Lemma \ref{ph}, then $u$ enjoys the following Pohozaev identity,
\begin{align}\label{ph3}
\hspace{-1cm}\int_{\R^N} |\nabla u|^2 \,dx=-\frac{\left(N(q-2)+2b\right)^2}{8q} \int_{\R^N} |x|^{-b} |u|^q \,dx + \frac{\left(N(p-2)+2b\right)^2}{8p} \int_{\R^N} |x|^{-b} |u|^p \, dx.
\end{align}
Since $Q(u)=0$, then
\begin{align*}
&\frac{N(q-2)+2b}{2q} \left(1-\frac{N(q-2)+2b}{4}\right) \int_{\R^N} |x|^{-b}|u|^{q} \, dx\\ 
& \quad + \frac{N(p-2)+2b}{2p} \left(\frac{N(p-2)+2b}{4}-1\right) \int_{\R^N} |x|^{-b}|u|^{p}\, dx=0. 
\end{align*}
If $2<q \leq 2+2(2-b)/N<p$, then $u=0$. This is impossible. We now consider the case $2+2(2-b)/N < q<p$. In this case, by using the fact $Q(u)=0$ and \eqref{ph3}, we can deduce that
$$
\left(1-\frac{N(q-2)+2b}{4}\right) \int_{\R^N} |\nabla u|^2\,dx=\frac{N\left(N(p-2)+2b\right)(p-q)}{8p}\int_{\R^N} |x|^{-b}|u|^{p}\, dx.
$$
It follows that $u=0$. This is impossible, either. Thus the proof is completed.
\end{proof}

\begin{lem}\label{unique4}
Let ${N \geq 1}$, $2<q<p$, $p>2+2(2-b)/N$, $0<b<\min\{2, N\}$ and $\mu=-1$, then, for any $u \in S(c)$, there exists a unique $t_u>0$ such that $u_{t_u} \in P(c)$ and $E(u_t)=\max_{t \geq 0} E(u_t)$ for any $c>0$. Moreover, the function $t \mapsto E(u_t)$ is concave on $[t_u, \infty)$ and the function $u \mapsto t_u$ is of class $C^1$.
\end{lem}
\begin{proof}
In light of \eqref{scaling1}, we easily obtain that $E(u_t)>0$ for any $t>0$ small enough and $E(u_t)<0$ for any $t>0$ large enough. If $2<q \leq 2+2(2-b)/N<p$, we write
\begin{align*}
\frac{d}{dt} E(u_t)&=t^{\frac N2 (q-2) +b-1} \left(t^\frac{2(2-b)-N(q-2)}{2} \int_{\R^N} |\nabla u|^2 \, dx +\frac{N(q-2)+2b}{2q} \int_{\R^N} |x|^{-b}|u|^q \, dx \right. \\
& \quad \left.-\frac{{\left(N(p-2)+2b\right)}}{2p}t^\frac{N(p-q)}{2}\int_{\R^N} |x|^{-b}|u|^p \, dx\right).
\end{align*}
It is simple to check that there exists a unique $t_u>0$ such that $\frac{d}{dt} E(u_t)\mid_{t=t_u}=0$. Furthermore, we find that $\frac{d}{dt} E(u_t)>0$ if $0<t<t_u$ and $\frac{d}{dt} E(u_t)<0$ if $t>t_u$. Observe that
\begin{align*}
\frac{d^2}{dt^2} E(u_t)&=\int_{\R^N} |\nabla u|^2 \,dx+\frac{(N(q-2)+2b)((N(q-2)-2(1-b))}{4q}t^{\frac{N}{2}(q-2)+b-2}\int_{\R^N}|x|^{-b}|u|^q \, dx \\
& \quad -\frac{(N(p-2)+2b)((N(p-2)-2(1-b))}{4p}t^{\frac{N}{2}(p-2)+b-2}\int_{\R^N}|x|^{-b}|u|^p \, dx.
\end{align*}
This infers that the function $t \mapsto E(u_t)$ is concave on $[t_u, \infty)$. Using the implicit function theorem, we can also have that the function $t\mapsto t_u$ is of class $C^1$. If $2+2(2-b)/N < q<p$, we write
\begin{align*}
\frac{d}{dt} E(u_t)&=t \left(\int_{\R^N} |\nabla u|^2 \, dx +\frac{N(q-2)+2b}{2q} t^{\frac{N(q-2)-2(2-b)}{2}}\int_{\R^N} |x|^{-b}|u|^q \, dx \right. \\
& \quad \left.-\frac{N(p-2)+2b}{2p}t^{\frac{N(p-2)-2(2-b)}{2}}\int_{\R^N} |x|^{-b}|u|^p \, dx\right).
\end{align*}
By a similar way, we can also get the desired results. Thus the proof is completed.
\end{proof}

\begin{lem}\label{coercive3}
Let ${N \geq 1}$, $2<q<p$, $p>2+2(2-b)/N$, $0<b<\min\{2, N\}$ and $\mu=-1$, then $\sigma(c)>0$ and $E$ restricted on $P(c)$ is coercive for any $c>0$, where
$$
\sigma(c):=\inf_{u \in P(c)} E(u).
$$
\end{lem}
\begin{proof}
For any $u \in P(c)$, we have that $Q(u)=0$. Therefore, we get that
\begin{align} \label{c31}
\begin{split}
E(u)&=E(u)-\frac{2}{N(p-2)+2b} Q(u)\\ 
&=\frac{N(p-2)-2(2-b)}{2\left(N(p-2)+2b\right)} \int_{\R^N} |\nabla u|^2 \,dx +\frac{N(p-q)}{q\left(N(p-2)+2b\right)}\int_{\R^N} |x|^{-b}|u|^q\,dx. 
\end{split}
\end{align}
Since $Q(u)=0$, by \eqref{GN}, then
\begin{align*}
&\int_{\R^N} |\nabla u|^2 \,dx+\frac{N(q-2)+2b}{2q} \int_{\R^N} |x|^{-b} |u|^q \,dx = \frac{N(p-2)+2b}{2p} \int_{\R^N} |x|^{-b} |u|^p \, dx \\
& \leq  \frac{C_{N,p,b} \left(N(p-2)+2b\right)}{2p}\left(\int_{\R^N} |\nabla u|^2 \, dx \right)^{\frac{N(p-2)+2b}{4}} c^{\frac p2 -\frac{N(p-2)+2b}{4}},
\end{align*}
from which we have that
\begin{align} \label{below111}
\int_{\R^N} |\nabla u|^2 \,dx \geq \left(\frac{2p}{C_{N,p,b} \left(N(p-2)+2b\right)c^{\frac{2(p-b)-N(p-2)}{4}}}\right)^{\frac{4}{N(p-2)-2(2-b)}}.
\end{align}
In view of \eqref{c31}, we then find that
$\sigma(c)>0$ for any $c>0$. Thanks to $p>2+2(2-b)/N$, it then follows from \eqref{c31} that $E$ restricted on $P(c)$ is coercive for any $c>0$. Thus the proof is completed.
\end{proof}

\begin{lem} \label{pss4}
Let ${N \geq 1}$, $2<q<p$, $p>2+2(2-b)/N$, $0<b<\min\{2, N\}$ and $\mu=-1$, then there exists a Palais-Smale sequence $\{u_n\} \subset P(c) $ with $(u_n)^-=o_n(1)$ for $E$ restricted on $S(c)$ at the level $\sigma(c)$ for any $c>0$. 
\end{lem}
\begin{proof}
The lemma can be proved by applying Lemmas \ref{unique4} and \ref{coercive3} and adapting the strategy of the proof of Lemma \ref{pss1}.
\end{proof}

\begin{lem}\label{la1}
Let ${N \geq 1}$, $2<q<p$, $p>2+2(2-b)/N$, $0<b<\min\{2, N\}$ and $\mu=-1$. If $u \in S(c)$ is a solution to the equation
\begin{align}\label{equ3}
-\Delta u + \lambda u= -|x|^{-b}|u|^{q-2} u + |x|^{-b}|u|^{p-2} u \quad \mbox{in} \,\, \R^N,
\end{align}
then there exists a constant $c_3>0$ such that $\lambda>0$ for any $0<c<c_3$, where $c_3>0$ is defined by
$$
c_3:=\left(\frac{q\left(N(p-2)-2(2-b)\right)}{2C_{N,q,b} \left(N(p-q)(N-b)\right)}\right)^{\frac{N(p-2)-2(2-b)}{(p-q)(2-b)}}\left(\frac{2p}{C_{N,p,b}\left(N(p-2)+2b\right)}\right)^{\frac{2(2-b)-N(q-2)}{(p-q)(2-b)}}.
$$
\end{lem}
\begin{proof}
Since $u \in H^1(\R^N)$ is a solution to \eqref{equ3}, then
$$
\int_{\R^N} |\nabla u|^2 \,dx + \lambda \int_{\R^N} |u|^2 \,dx + \int_{\R^N} |x|^{-b}|u|^q \,dx =\int_{\R^N} |x|^{-b}|u|^p\,dx.
$$ 
On the other hand, from Lemma \ref{ph}, we see that
$$
\int_{\R^N} |\nabla u|^2 \,dx + \frac{N(q-2)+2b}{2q}\int_{\R^N} |x|^{-b}|u|^q \,dx = \frac{N(p-2)+2b}{2p}\int_{\R^N} |x|^{-b}|u|^p \,dx.
$$
Therefore, we conclude that
\begin{align*} 
&\frac{\lambda\left(N(p-2)+2b\right)}{2p} \int_{\R^N} |u|^2 \,dx = \frac{2(p-b)-N(p-2)}{2p} \int_{\R^N} |\nabla u|^2 \,dx -\frac{(p-q)(N-b)}{pq} \int_{\R^N} |x|^{-b}|u|^q \,dx \\
&\geq \frac{2(p-b)-N(p-2)}{2p} \|\nabla u\|_2^2 - \frac{C_{N,q,b} (p-q)(N-b)}{pq}\|\nabla u\|_2 ^{\frac{N(q-2)+2b}{2}} c^{\frac q2 -\frac{N(q-2)+2b}{4}}\\
& =\frac 1p \|\nabla u\|_2^{\frac{N(q-2)+2b}{2}} \left(\frac{2(p-b)-N(p-2)}{2}\|\nabla u\|_2^{\frac{2(2-b)-N(q-2)}{2}} -\frac{C_{N,q,b} (p-q)(N-b)}{q}c^{\frac{2(q-b)-N(q-2)}{4}}\right),
\end{align*}
where we used \eqref{GN} for the inequality. By applying \eqref{below111}, we then have that $\lambda>0$ for any $0<c<c_3$. Thus the proof is completed.
\end{proof}

\begin{proof}[Proof of Theorem \ref{thm5}]
According to Lemma \ref{pss4}, then there exists a Palais-Smale sequence $\{u_n\} \subset P(c)$ with $(u_n)^{-}=o_n(1)$ for $E$ restricted on $S(c)$ at the level $\sigma(c)$. From Lemma \ref{coercive3}, it follows that $\{u_n\}$ is bounded in $H^1(\R^N)$. Therefore, there exists a nonnegative $u \in H^1(\R^N)$ such that $u_n \wto u$ in $H^1(\R^N)$ as $n \to \infty$ up to translations. Thanks to $\sigma(c)>0$, by Lemma \ref{cembedding} and \cite[Lemma I.1] {Li2}, then $u \neq 0$. {Since $\{u_n\} \subset P(c)$ is a Palais-Smale sequence for $E$ restricted on $S(c)$ at the level $\sigma(c)$, arguing as the proof of \cite[Lemma 3]{BeLi}, we then get that}
\begin{align*}
-\Delta u_n + \lambda_n u_n= -|x|^{-b}|u_n|^{q-2} u_n + |x|^{-b}|u_n|^{p-2} u_n +o_n(1) \quad \mbox{in} \,\, \R^N,
\end{align*}
where 
$$
\lambda_n:=\frac 1 c \left(\int_{\R^N} |x|^{-b}|u_n|^p \, dx-\int_{\R^N} |x|^{-b}|u_n|^q \, dx -\int_{\R^N} |\nabla u|^2 \, dx\right).
$$
Since $\{u_n\}$ is bounded in $H^1(\R^N)$, then there exists $\lambda \in \R$ such that $\lambda_n \to \lambda$ in $\R$ as $n \to \infty$. As a consequence, we have that $u$ satisfies the equation
\begin{align*}
-\Delta u + \lambda u= -|x|^{-b}|u|^{q-2} u + |x|^{-b}|u|^{p-2} u \quad \mbox{in} \,\, \R^N.
\end{align*}
From Lemma \ref{ph}, we know that $Q(u)=0$. At this point, using Lemmas \ref{cembedding}  and \ref{la1}, we can conclude that $u_n \to u$ in $H^1(\R^N)$ as $n \to \infty$. Further, by Lemma \ref{radial}, we then have the desired result. Thus the proof is completed.
\end{proof}

\section{Mass critical case revisted} \label{rcritical}

In this section, our purpose is to consider the existence of solutions to \eqref{equ}-\eqref{mass} for the case $2<q<p=2+2(2-b)/N$ and $\mu=-1$. For this aim, we introduce the following minimization problem,
\begin{align}\label{min4}
\sigma(c):=\inf_{u \in \mathcal{P}(c)} E(u),
\end{align}
where $\mathcal{P}(c):=P(c) \cap \mathcal{S}(c)$ and
$$
\mathcal{S}(c):=\left\{u \in S(c) : \int_{\R^N} |\nabla u|^2 \,dx < \frac 2 p \int_{\R^N} |x|^{-b} |u|^p \, dx \right\}.
$$
Clearly, $w \in \mathcal{S}(c)$ for any $c>c_1$, where $w \in S(c)$ is defined by \eqref{defw}. This infers that $\mathcal{S}(c) \neq \emptyset$ for any $c>c_1$.

\begin{lem} 
Let ${N \geq 1}$, $2<q<p=2+2(2-b)/N$, $0<b<\min\{2, N\}$ and $\mu=-1$, then $\mathcal{P}(c)$ is a smooth manifold of codimension $2$ in $H^1(\R^N)$ for any $c>c_1$.
\end{lem}
\begin{proof}
The proof of this lemma is almost identical to that of Lemma \ref{cod11}, then we omit its proof here.
\end{proof}

\begin{lem} \label{unique5}
Let ${N \geq 1}$, $2<q<p=2+2(2-b)/N$, $0<b<\min\{2, N\}$ and $\mu=-1$, then, for any $u \in \mathcal{S}(c)$, there exists a unique $t_u>0$ such that $u_{t_u} \in \mathcal{P}(c)$ and $E(u_t)=\max_{t \geq 0} E(u_t)$ for any $c>c_1$. Moreover, the function $t \mapsto E(u_t)$ is concave on $[t_u, \infty)$ and the function $u \mapsto t_u$ is of class $C^1$.
\end{lem}
\begin{proof}
In view of \eqref{scaling1}, we have that, for any $u \in \mathcal{S}(c)$,
$$
E(u_t)= t^2 \left(\frac 12\int_{\R^N} |\nabla u|^2 \,dx- \frac 1p \int_{\R^N}|x|^{-b}|u|^p \, dx\right)
+\frac{t^{\frac{N}{2}(q-2)+b}}{q} \int_{\R^N}|x|^{-b}|u|^q \, dx.
$$
Since $2<q<2+2(2-b)/N$, then $E(u_t) \to -\infty$ as $t \to \infty$ and $E(u_t)>0$ for any $t>0$ small enough. 
Note that
$$
\frac{d}{dt}E(u_t)= t\left(\int_{\R^N} |\nabla u|^2 \,dx- \frac 2p \int_{\R^N}|x|^{-b}|u|^p \, dx\right)
+\frac{N(q-2)+2b}{2q} t^{\frac{N}{2}(q-2)+b-1} \int_{\R^N}|x|^{-b}|u|^q \, dx.
$$
Then there exists a unique $t_u>0$ such that $\frac{d}{dt}E(u_t)\mid_{t=t_u}=0$, where $t_u>0$ is given by
\begin{align} \label{deftu11}
t_u:=\left(\frac{4q\int_{\R^N}|x|^{-b}|u|^p \, dx-2pq\int_{\R^N} |\nabla u|^2 \,dx}{p(N(p-2)+2b)\int_{\R^N}|x|^{-b}|u|^q \, dx}\right)^{\frac{N(q-2)-2(2-b)}{2}}.
\end{align} 
Furthermore, there holds that $\frac{d}{dt}E(u_t)>0$ for any $0<t<t_u$ and $\frac{d}{dt}E(u_t)<0$ for any $t>t_u$. Accordingly, we are able to deduce that there exists a unique $t_u>0$ such that $u_{t_u} \in \mathcal{P}(c)$ and $E(u_t)=\max_{t \geq 0} E(u_t)$. Note that 
\begin{align*}
\frac{d^2}{dt^2}E(u_t) &= \int_{\R^N} |\nabla u|^2 \,dx- \frac 2p \int_{\R^N}|x|^{-b}|u|^p \, dx \\
& \quad +\frac{(N(q-2)+2b)(N(q-2)-2(1-b))}{4q} t^{\frac{N}{2}(q-2)+b-2} \int_{\R^N}|x|^{-b}|u|^q \, dx.
\end{align*}
Then the function $t \mapsto E(u_t)$ is concave on $[t_u, \infty)$. From \eqref{deftu11}, then the smoothness of the function $t \mapsto t_u$ follows. Thus the proof is completed.
\end{proof}

\begin{lem}\label{coercive5}
Let ${N \geq 1}$, $2<q<p=2+2(2-b)/N$, $0<b<\min\{2, N\}$ and $\mu=-1$, then $\sigma(c)>0$ and $E$ restricted on $\mathcal{P}(c)$ is coercive for any $c>c_1$.
\end{lem}
\begin{proof}
For simplicity, we only show the proof for the case $N \geq 3$. The case $N=1,2$ can handled by an analogous way. Since $u \in \mathcal{P}(c)$, then $Q(u)=0$. Thus we can deduce that, for any $u \in \mathcal{P}(c)$,
\begin{align} \label{c5}
E(u)=E(u)-\frac 12 Q(u)=\frac{2(2-b)-N(q-2)}{4q} \int_{\R^N} |x|^{-b}|u|^q \,dx.
\end{align}
Making use of Lemma \ref{inequ} with $r=2_b^*$, we obtain that if $Q(u)=0$, then
\begin{align*}
\int_{\R^N} |\nabla u|^2 \,dx+\frac{N(q-2)+2b}{2q} \int_{\R^N} |x|^{-b} |u|^q \,dx &= \frac{N}{N+2-b} \int_{\R^N} |x|^{-b} |u|^p \, dx \\
& \leq  C \left(\int_{\R^N} |x|^{-b} |u|^q \,dx\right)^{1-\theta}\left(\int_{\R^N} |\nabla u|^2 \, dx\right)^{\frac{\theta(N-b)}{N-2}},
\end{align*}
where $p=(1-\theta) q+\theta 2_b^*$. Since $q>2$, then $\theta(N-b)/(N-2) <1$ and $\left(1-\theta\right)N+\theta b>2$. Using Young's inequality, we then have that
\begin{align*}
C\left(\int_{\R^N} |x|^{-b} |u|^q \,dx\right)^{1-\theta}\left(\int_{\R^N} |\nabla u|^2 \, dx\right)^{\frac{\theta(N-b)}{N-2}} &\leq \widetilde{C} \left(\int_{\R^N} |x|^{-b} |u|^q \,dx\right)^{\frac{\left(1-\theta\right)\left(N-2\right)}{\left(1-\theta\right)N+\theta b-2}} + \frac 12 \int_{\R^N} |\nabla u|^2 \,dx.
\end{align*}
Then we derive that 
\begin{align} \label{c511}
\int_{\R^N} |x|^{-b} |u|^q \,dx \geq \left(\frac{N(q-2)+2b}{2q \widetilde{C}}\right)^{\frac{\left(1-\theta\right)N+\theta b-2}{\theta(2-b)}}
\end{align}
and
\begin{align} \label{c521}
\int_{\R^N} |\nabla u|^2 \,dx \leq 2 \widetilde{C} \left(\int_{\R^N} |x|^{-b} |u|^q \,dx\right)^{\frac{\left(1-\theta\right)\left(N-2\right)}{\left(1-\theta\right)N+\theta b-2}}.
\end{align}
Coming back to \eqref{c5} and using \eqref{c511} and \eqref{c521}, we then get the desired result. Thus the proof is completed. 
\end{proof}

\begin{lem} \label{pss5}
Let ${N \geq 1}$, $2<q<p=2+2(2-b)/N$, $0<b<\min\{2, N\}$ and $\mu=-1$, then there exists a Palais-Smale sequence $\{u_n\} \subset P(c) $ with $(u_n)^-=o_n(1)$ for $E$ restricted on $\mathcal{S}(c)$ at the level $\sigma(c)$ for any $c>c_1$. 
\end{lem}
\begin{proof}
The proof of this lemma can be completed adapting the ideas to that of Lemma \ref{pss4} and replacing the role of $S(c)$ by $\mathcal{S}(c)$, we omit its proof here.
\end{proof}

\begin{lem} \label{la3}
Let ${N \geq 1}$, $2<q<p=2+2(2-b)/N$, $0<b<\min\{2, N\}$ and $\mu=-1$. If $u \in \mathcal{S}(c)$ is a solution to the equation
\begin{align}\label{equ5}
-\Delta u + \lambda u= -|x|^{-b}|u|^{q-2} u + |x|^{-b}|u|^{p-2} u \quad \mbox{in} \,\, \R^N,
\end{align}
then there exists a constant $c_1^*>0$ such that $\lambda>0$ for any $c_1<c<c_1^*$, where $c_1^*>0$ is defined by
$$
c_1^*:=\left(\frac{p\left(2(q-b)-N(q-2)\right)}{2C_{N,b}(p-q)(N-b)}\right)^{\frac{N}{2-b}}>c_1.
$$
\end{lem}
\begin{proof}
Since $u \in H^1(\R^N)$ is a solution to \eqref{equ5}, then
$$
\int_{\R^N} |\nabla u|^2 \,dx + \lambda \int_{\R^N} |u|^2 \,dx + \int_{\R^N} |x|^{-b}|u|^q \,dx =\int_{\R^N} |x|^{-b}|u|^p\,dx.
$$ 
In addition, from Lemma \ref{ph}, we have that
$$
\int_{\R^N} |\nabla u|^2 \,dx + \frac{N(q-2)+2b}{2q}\int_{\R^N} |x|^{-b}|u|^q \,dx = \frac{N(p-2)+2b}{2p}\int_{\R^N} |x|^{-b}|u|^p \,dx.
$$
As a consequence, we derive that
\begin{align*}
\frac{\lambda\left(N(q-2)+2b\right)}{2q}\int_{\R^N} |u|^2 \,dx&=\frac{2(q-b)-N(q-2)}{2q}\int_{\R^N} |\nabla u|^2 \,dx-\frac{(N-b)(p-q)}{pq}\int_{\R^N} |x|^{-b}|u|^p \,dx \\
& \geq \left(\frac{2(q-b)-N(q-2)}{2q}-\frac{C_{N,b}(N-b)(p-q)}{pq} c^{\frac{2-b}{N}}\right)\int_{\R^N} |\nabla u|^2 \,dx, 
\end{align*}
where we used \eqref{GN} for the inequality. Therefore, we see that $\lambda>0$ for any $c_1<c<c_1^*$. Due to $2<q<p=2+2(2-b)/N$, then $c_1^*>c_1$. Thus the proof is completed.
\end{proof}


\begin{proof}[Proof of Theorem \ref{thm6}]
Taking advantage of previous lemmas in this section and reasoning as the proof of Theorem \ref{thm5}, we are able to derive the desired result.
\end{proof}

\section{Propositions of the function $c \mapsto \sigma(c)$ for $c>0$} \label{prop}

In this section, we investigate some properties of the function $c \mapsto \sigma(c)$ for $c>c_0$, where $\sigma(c) \in \R^+$ is defined by
\begin{align*}
\sigma(c):=
\left\{
\begin{aligned}
&\inf_{u \in P(c)} E(u), \quad \mbox{if} \,\, q>2+2(2-b)/N\,\,\mbox{and} \,\, \mu=1 \,\,\mbox{or} \,\, p>2+2(2-b)/N \,\, \mbox{and} \,\, \mu=-1, \\
&\inf_{u \in \mathcal{P}(c)} E(u), \quad \mbox{if} \,\, p=2+2(2-b)/N\,\, \mbox{and} \,\,\mu=-1.
\end{aligned}
\right.
\end{align*}
Throughout this section, we shall always assume that${N \geq 1}$, $2<q<p<2^*_b$ and $0<b<\min\{2, N\}$. To prove Proposition \ref{prop2}, we are going to establish the following propositions.

\begin{prop} \label{nonincreasing}
Let $q>2+2(2-b)/N$ and $\mu=1$ or $p \geq 2+2(2-b)/N$ and $\mu=-1$, then the function $c \mapsto \sigma(c)$ is nonincreasing on $(c_0, \infty)$.
\end{prop}
\begin{proof}
For simplicity, we only show the proof for the case $q>2+2(2-b)/N$ and $\mu=1$ or $p>2+2(2-b)/N$ and $\mu=-1$. In this case, we have that $c_0=0$. To prove the result, we need to verify that, if $c_1>c_2>0$, then $\sigma(c_1) \leq \sigma(c_2)$. In view of the definition of $\sigma(c)$, we know that, for any $\eps>0$, there exists $u \in P(c_2)$ such that
\begin{align} \label{non1}
E(u) \leq \sigma (c_2) + \frac{\eps}{2}.
\end{align}
Let $\chi \in C_0^{\infty}(\R^N, [0,1])$ be a cut-off function such that $\chi(x) =1$ for $|x| \leq 1$ and $\chi(x)=0$ for $|x|\geq 2$. For $\delta>0$ small, we define $u^{\delta}(x):=u(x) \chi(\delta x)$ for $x \in \R^N$. It is standard to check that $u^{\delta} \to u$ in $H^1(\R^N)$ as $\delta \to 0^+$. Therefore, we have that $(u^{\delta})_{t_{u^{\delta}}} \to u$ in $H^1(\R^N)$ as $\delta \to 0^+$, where $t_{u^{\delta}}>0$ is a constant such that $Q((u^{\delta})_{t_{u^{\delta}}})=0$ and $t_{u_{\delta}} \to 1$ as $\delta \to 0^+$. It then follows that there exists a constant $\delta>0$ small such that
\begin{align} \label{non2}
E((u^{\delta})_{t_{u_{\delta}}}) \leq E(u) + \frac{\eps}{4}.
\end{align}
Let $v \in C^{\infty}_0(\R^N)$ be such that $\mbox{supp}\,v \subset B(0, 1+2/\delta) \backslash B(0, 2/\delta)$ and set
$$
\tilde{v}^{\delta}:=\frac{\left(c_1-\|u^{\delta}\|_2^2\right)^{\frac 12}}{\|v\|_2} v.
$$ 
For $0<t<1$, we now define $w^{\delta}_t:=u^{\delta} + (\tilde{v}^{\delta})_t$, where $ (\tilde{v}^{\delta})_t(x):= t^{N/2}\tilde{v}^{\delta}(tx)$ for $x\in \R^N$. Observe that $\mbox{supp}\, u_{\delta} \cap \mbox{supp}\, (\tilde{v}^{\delta})_t =\emptyset$, then $\|w^{\delta}_t\|_2^2=c_1$. Moreover, it is not hard to infer that $ \| \nabla(\tilde{v}^{\delta})_t\|_2 \to 0$ as $t \to 0^+$. Therefore, we are able to deduce that there exist $t, \delta>0$ small enough such that
\begin{align}\label{non3}
\max_{s>0}E(((\tilde{v}^{\delta})_t)_s) \leq \frac{\eps}{4}.
\end{align}
Accordingly, from \eqref{non1}-\eqref{non3}, we arrive at
\begin{align*}
\sigma(c_1) \leq \max_{s>0} E((w^{\delta}_t)_s)&=\max_{s>0}\left(E((u^{\delta})_s) + E(((\tilde{v}^{\delta})_t)_s) \right) \\
&\leq E((u^{\delta})_{t_{u_{\delta}}})  + \frac{\eps}{4}\leq E(u)+ \frac{\eps}{2} \leq \sigma(c_2) +\eps.
\end{align*}
Since $\eps>0$ is arbitrary, then $\sigma(c_1) \leq \sigma(c_2)$. Thus the proof is completed.
\end{proof}


\begin{prop} \label{continuous}
Let $q>2+2(2-b)/N$ and $\mu=1$ or $p \geq 2+2(2-b)/N$ and $\mu=-1$, then the function $c \mapsto \sigma(c)$ is continuous for any $c>c_0$.
\end{prop}
\begin{proof}
For simplicity, we only prove the result for the case $q>2+2(2-b)/N$ and $\mu=1$ or $p >2+2(2-b)/N$ and $\mu=-1$. Let $c>0$ and $\{c_n\} \subset (0, \infty)$ satisfying $c_n \to c$ as $n \to \infty$, we shall prove that $\sigma(c_n) \to \sigma(c)$ as $n \to \infty$. From the definition of $\sigma(c)$, we have that, for any $\eps>0$, there exists $u \in P(c)$ such that
$$
E(u) \leq \sigma(c) + \frac{\eps}{2}.
$$
Define
$$
u_n:=\left(\frac{c_n}{c}\right)^{\frac 12} v \in S(c_n).
$$
It is not hard to verify that $v_n \to v$ in $H^1(\R^N)$ as $n \to \infty$. Therefore, we are able to deduce that
\begin{align*}
\sigma(c_n) \leq \max_{t>0} E((u_n)_{t}) \leq \max_{t>0} E(u_t) + \frac{\eps}{2}=E(u)+\frac{\eps}{2} \leq \sigma(c) +\eps.
\end{align*}
This implies that
$$
\limsup_{n\to \infty}\sigma(c_n) \leq \sigma(c).
$$
On the other hand, by a similar way, we can prove that
$$
\sigma(c) \leq \displaystyle\liminf_{n\to \infty}\sigma(c_n).
$$
Then the desired conclusion follows. Thus proof is completed.
\end{proof}

\begin{prop} \label{limit1}
Let $q>2+2(2-b)/N$ and $\mu=1$ or $p>2+2(2-b)/N$ and $\mu=-1$, then $\sigma(c) \to \infty$ as $c \to 0^+$.
\end{prop}
\begin{proof}
Here we only consider the case for $q>2+2(2-b)/N$ and $\mu=1$. The other case can be done by a similar way. For any $u_c\in P(c)$, we have that $Q(u_c)=0$. Therefore, we conclude that
\begin{align} \label{c311}
\begin{split}
E(u_c)&=E(u_c)-\frac 12 Q(u_c)\\ 
&=\frac{N(q-2)-2(2-b)}{4q} \int_{\R^N} |x|^{-b}|u_c|^q \,dx + \frac{N(p-2)-2(2-b)}{4p} \int_{\R^N} |x|^{-b}|u_c|^p \,dx.
\end{split}
\end{align}
Using the fact that $Q(u_c)=0$ and \eqref{GN}, we find that
\begin{align*}
\int_{\R^N} |\nabla u_c|^2 \,dx &= \frac{N(q-2)+2b}{2q}\int_{\R^N} |x|^{-b}|u_c|^q \,dx + \frac{N(p-2)+2b}{2p}\int_{\R^N} |x|^{-b}|u_c|^p \,dx \\
& \leq  \frac{C_{N,q,b} \left(N(q-2)+2b\right)}{2q}\left(\int_{\R^N} |\nabla u_c|^2 \,dx\right)^{\frac{N(q-2)+2b}{4}} c^{\frac q2 -\frac{N(q-2)+2b}{4}}  \\
& \quad +\frac{C_{N,p,b} \left(N(p-2)+2b\right)}{2p}\left(\int_{\R^N} |\nabla u_c|^2 \,dx\right)^{\frac{N(p-2)+2b}{4}} c^{\frac p2 -\frac{N(p-2)+2b}{4}}.
\end{align*}
Due to $2+2(2-b)/N<q<p<2^*_b$, then $\|\nabla u_c\|_2 \to \infty$ as $c \to 0^+$. Note that $Q(u_c)=0$, then
$$
\int_{\R^N} |x|^{-b}|u_c|^q \,dx + \int_{\R^N} |x|^{-b}|u_c|^p \,dx \to \infty \quad \mbox{as} \,\, c \to 0^+.
$$
This along with \eqref{c311} leads to the desired result. Thus the proof is completed.
\end{proof}

\begin{prop} \label{limit2}
Let $p=2+2(2-b)/N$ and $\mu=-1$, then $\sigma(c) \to \infty$ as $c \to (c_1)^+$.
\end{prop}
\begin{proof}
For any $u_c \in P(c)$, we know that $Q(u_c)=0$. This then yields that
\begin{align} \label{c51}
E(u_c)=E(u_c)-\frac 12 Q(u_c)=\frac{2(2-b)-N(q-2)}{4q} \int_{\R^N} |x|^{-b}|u_c|^q \,dx.
\end{align}
By using the fact that $Q(u)=0$ and \eqref{GN}, we obtain that
\begin{align} \label{c52}
\begin{split}
\int_{\R^N} |\nabla u_c|^2 \, dx + \frac{N(q-2)+2b}{2q} \int_{\R^N}|x|^{-b}| u_c|^q \, dx&=\frac{N(p-2)+2b}{2p} \int_{\R^N}|x|^{-b}| u_c|^p \, dx \\
& \leq \frac{2C_{N,b}c^{\frac{2-b}{N}}}{p}\int_{\R^N} |\nabla  u_c|^2 \, dx.
\end{split}
\end{align}
According to Proposition \ref{nonincreasing}, we know that $\sigma(c) \geq \sigma(c_1^*)$ for any $c_1<c<c_1^*$. It then yields from \eqref{c51} and \eqref{c52} that $\|\nabla u_c\|_2 \to \infty$ as $c \to c_1^+$. Applying \eqref{c52}, we then derive that
\begin{align} \label{limit0}
\int_{\R^N}|x|^{-b}| u_c|^p \, dx \to \infty \quad \mbox{as} \,\, c \to c_1^+.
\end{align}
By means of \eqref{c52}, we see that
$$
\int_{\R^N} |\nabla u_c|^2 \, dx<\frac{N(p-2)+2b}{2p} \int_{\R^N}|x|^{-b}| u_c|^p \, dx.
$$
Using Lemma \ref{inequ} with $2<q<r<p$ such that $\theta\left(N(r-2)+2b\right)<4$, we then derive from \eqref{limit0} that
$$
\int_{\R^N}|x|^{-b}| u_c|^q \, dx \to \infty \quad \mbox{as} \,\, c \to c_1^+.
$$ 
This together with \eqref{c51} readily implies the desired result. Thus the proof is completed.
\end{proof}

\begin{prop} \label{limit3}
Let $q>2+2(2-b)/N$ and $\mu=1$, then $\sigma(c) \to 0$ as $c \to \infty$.
\end{prop}
\begin{proof}
Let $u \in S(1)$ and define $u_c:=c^{1/2}u \in S(c)$. In view of Lemma \ref{unique3}, there exists a unique $t_c >0$ such that $Q((u_c)_{t_c})=0$. This means that
\begin{align} \label{limit00}
\begin{split}
\int_{\R^N} |\nabla u|^2 \, dx &=\frac{N(q-2)+2b}{2q}(t_c^2 c)^{\frac q 2-1}t_c^{\frac{N}{2}(q-2)+b-q}\int_{\R^N}|x|^{-b}|u|^q \, dx\\
&\quad +\frac{N(p-2)+2b}{2p} (t_c^2 c)^{\frac p 2-1}t_c^{\frac{N}{2}(p-2)+b-p}\int_{\R^N}|x|^{-b}|u|^p \, dx,
\end{split}
\end{align}
from which we have that $t_c \to 0$ and $t_c^2 c \to 0$ as $c \to \infty$, because of $2+2(2-b)/N<q<p<2^*_b$. Observe that
\begin{align*}
\sigma(c) \leq E((u_c)_{t_c}) &=E((u_c)_{t_c})-\frac 12 Q((u_c)_{t_c}) \\
&=\frac{N(q-2)-2(2-b)}{4q} t_c^{\frac{N}{2}(q-2)+b}c^{\frac q 2}\int_{\R^N}|x|^{-b}|u|^q \, dx \\
& \quad +\frac{N(p-2)-2(2-b)}{4p} t_c^{\frac{N}{2}(p-2)+b}c^{\frac p 2}\int_{\R^N}|x|^{-b}|u|^p \, dx.
\end{align*}
As a consequence, it follows from \eqref{limit00} that $\sigma(c) \to 0$ as $c \to \infty$. Thus the proof is completed.
\end{proof}

In order to consider the asymptotic behaviors of the function $c \mapsto \sigma(c)$ as $c \to \infty$ for the case $\mu=-1$ and $N \geq 3$, we need to study the following zero mass equation,
\begin{align} \label{zequ}
-\Delta u + |x|^{-b}|u|^{q-2} u = |x|^{-b}|u|^{p-2} u \quad \mbox{in} \,\, \R^N.
\end{align}
Let us first introduce the natural Sobolev space $X$ used to consider solutions to \eqref{zequ}, which is defined by the completion of $C_0^{\infty}(\R^N)$ under the norm
$$
\|u\|_X:=\left(\int_{\R^N}|\nabla u|^2 \,dx\right)^{\frac 12} + \left(\int_{\R^N}|x|^{-b}|u|^q \, dx\right)^{\frac 1 q}.
$$
It is standard to demonstrate that $X$ is a reflexive Banach space. Next we present the associated embedding result in $X$.

\begin{lem} \label{space}
Let $N \geq 3$, then there exists a constant $C>0$ depending only on $N, p, q$ and $b$ such that, for any $u \in X$,
\begin{align} \label{inequality}
\left(\int_{\R^N} |x|^{-b}|u|^p \,dx \right)^{\frac 1p}\leq C \left(\int_{\R^N} |x|^{-b}|u|^q \,dx\right)^{\frac{1-\theta}{q}} \left(\int_{\R^N} |\nabla u|^2 \, dx\right)^{\frac{\theta}{2}},
\end{align}
where $0 <\theta < 1$ such that $1/p=(1-\theta)/q + \theta/2^*_b$.
\end{lem}
\begin{proof}
From H\"older's inequality, we first have that
$$
\left(\int_{\R^N} |x|^{-b}|u|^p \,dx \right)^{\frac 1p}\leq \left(\int_{\R^N} |x|^{-b}|u|^q \,dx\right)^{\frac{1-\theta}{q}} \left(\int_{\R^N} |x|^{-b} |u|^{2^*_b} \, dx\right)^{\frac{\theta}{2^*_b}}.
$$
In addition, we know that 
$$
\left(\int_{\R^N} |x|^{-b} |u|^{2^*_b} \, dx\right)^{\frac{1}{2^*_b}} \leq S_{N,b} \left(\int_{\R^N} |\nabla u|^2 \, dx \right)^{\frac 12}.
$$
Then we obtain \eqref{inequality} and the proof is completed.
\end{proof}

\begin{lem} \label{embedding}
Let $N \geq 3$, then the embedding $X \hookrightarrow L^p(\R^N, |x|^{-b}dx) $ is continuous for any $q \leq p\leq 2^*_b$ and locally compact for any $q \leq p<2^*_b$.
\end{lem}
\begin{proof}
The continuity of the embedding is directly from Lemma \ref{inequality}. We now prove that the embedding is locally compact. Let $\Omega \subset \R^N$ be a bounded domain and $\{u_n\} \subset X$ be a sequence such that $(u_n)_{\mid \Omega} \wto 0$ in $X$ as $n \to \infty$. We are going to prove that $(u_n)_{\mid \Omega} \to 0$ in $L^p(\R^N, |x|^{-b}dx)$ as $n \to \infty$ for any $q \leq p<2^*_b$. Observe that, for any $u \in X$, there holds that $u \in L^{2^*}(\R^N)$. It then leads to $u_{\mid \Omega} \in H^1(\Omega)$.  Since the embedding $H^1(\R^N) \hookrightarrow L^r(\R^N)$ is locally compact for any $1 \leq r <2^*$, then $(u_n)_{\mid \Omega} \to 0$ in $L^r(\R^N)$ as $n \to \infty$ for any $1 \leq r <2^*$. In view of H\"older's inequality, we then derive that 
\begin{align*}
\int_{\Omega} |x|^{-b}|u_n|^p \,dx \leq \left(\int_{\Omega} |x|^{-bs} \,dx\right)^{\frac 1 s} \left(\int_{\Omega} |u_n|^{\frac{ps}{s-1}} \,dx\right)^{\frac{s-1}{s}}=o_n(1),
\end{align*}
where $s>1$ with $0<bs<N$ and $1<ps/(s-1)<2^*$. Then the desired conclusion follows and the proof is completed.
\end{proof}

\begin{lem} \label{lions}
Let $N \geq 3$ and $\{u_n\}\subset X$ be a bounded sequence satisfying
\begin{align} \label{lions1}
\sup_{y \in \R^N} \int_{B_R(y)} |x|^{-b} |u_n|^q \,dx=o_n(1),
\end{align}
then $u_n \to 0$ in $L^p(\R^N, |x|^{-b}\,dx)$ as $n \to \infty$ for any $q<p<2^*_b$.
\end{lem}
\begin{proof}
For $2<q<p<r<2^*_b$, by H\"older's inequality and Lemma \ref{inequality}, we infer that
\begin{align*}
\int_{B_R(y)}|x|^{-b}|u|^p \, dx & \leq \left(\int_{B_R(y)}|x|^{-b}|u|^q \, dx \right)^{\frac{\alpha p}{q}} \left(\int_{B_R(y)}|x|^{-b}|u|^r \, dx \right)^{\frac{(1-\alpha)p}{r}} \\
& \leq C \left(\int_{B_R(y)}|x|^{-b}|u|^q \, dx \right)^{\frac{\alpha p}{q}}\left(\int_{B_R(y)} |x|^{-b}|u|^q \,dx\right)^{\frac{(1-\theta)(1-\alpha)p}{q}} \|\nabla u\|_{L^2(B_R(y))}^{\theta(1-\alpha)p}\\
& \leq  C \left(\int_{B_R(y)}|x|^{-b}|u|^q \, dx \right)^{\frac{\alpha p}{q}} \left(\left(\int_{B_R(y)}|x|^{-b}|u|^q \, dx\right)^{\frac 1 q}+\|\nabla u\|_{L^2(B_R(y))} \right)^{(1-\alpha)p},
\end{align*}
where $0<\alpha<1$ and $0<\theta<1$ such that  
$$
\frac 1 p =\frac {\alpha }{q} + \frac {1-\alpha} {r}, \quad \frac 1 r =  \frac {1-\theta} {q} +\frac{\theta}{2^*_b}.
$$
Choosing $(1-\alpha)p=1$, we then obtain that
$$
\int_{B_R(y)}|x|^{-b}|u|^p \, dx  \leq C \left(\int_{B_R(y)}|x|^{-b}|u|^q \, dx \right)^{\frac{p-1}{q}} \left(\left(\int_{B_R(y)}|x|^{-b}|u|^q \, dx\right)^{\frac 1 q}+\|\nabla u\|_{L^2(B_R(y))} \right).
$$
Observe that every point in $\R^N$ is contained in at most $N+1$ balls if covering $\R^N$ by balls of radius $R>0$. This results in
$$
\int_{\R^N}|x|^{-b}|u|^p \, dx \leq C(N+1) \left(\sup_{y \in \R^N} \int_{B_R(y)} |x|^{-b} |u|^q \,dx\right)\|u\|_X.
$$
Therefore, if $\{u_n\} \subset X$ is bounded and \eqref{lions1} is valid, then we get that $u_n \to 0$ in $L^p(\R^N, |x|^{-b}\,dx)$ as $n \to \infty$. Thus the proof is completed.
\end{proof}

\begin{lem}\label{existzero}
Let $N \geq 3$, then there exists a nonnegative, symmetric and decreasing ground state to \eqref{zequ} in $X$ at the level $\sigma_0>0$, where $\sigma_0>0$ is defined by \begin{align} \label{zmin}
\sigma_0:=\inf\{E(u) : u \in X \backslash \{0\},  Q(u)=0\}
\end{align}
and
$$
Q(u)=\int_{\R^N} |\nabla u|^2 \, dx + \frac{N(q-2)+2b}{2q} \int_{\R^N}|x|^{-b}|u|^q \, dx-\frac{N(p-2)+2b}{2p} \int_{\R^N}|x|^{-b}|u|^p \, dx.
$$
\end{lem}
\begin{proof}
We first infer that $\sigma_0>0$. Let $u \in X$ be such that $Q(u)=0$. Then we have that
\begin{align} \label{zbe}
\begin{split}
E(u)=E(u)&=E(u)-\frac{2}{N(p-2)+2b} Q(u)\\ 
&=\frac{N(p-2)-2(2-b)}{2\left(N(p-2)+2b\right)} \int_{\R^N} |\nabla u|^2 \,dx +\frac{N(p-q)}{q\left(N(p-2)+2b\right)}\int_{\R^N} |x|^{-b}|u|^q\,dx.
\end{split}
\end{align}
On the other hand, from Lemma \ref{space}, we see that
\begin{align*}
\int_{\R^N} |\nabla u|^2 \,dx+\frac{N(q-2)+2b}{2q} \int_{\R^N} |x|^{-b} |u|^q \,dx &= \frac{N(p-2)+2b}{2p} \int_{\R^N} |x|^{-b} |u|^p \, dx \\
& \leq  \widetilde{C} \left(\int_{\R^N} |x|^{-b}|u|^q \,dx\right)^{\frac{\left(1-\theta\right)p}{q}} \left(\int_{\R^N}|\nabla u|^2 \, dx\right)^{\frac{\theta p}{2}} \\
& \leq\widetilde{C} \left(\int_{\R^N}|\nabla u|^2 \, dx + \int_{\R^N} |x|^{-b}|u|^q \,dx\right)^{\left(\frac{1-\theta}{q}+\frac{\theta}{2} \right)p}, 
\end{align*}
which implies that there exists a constant $\widehat{C}>0$ such that
$$
\int_{\R^N} |\nabla u|^2 \,dx + \int_{\R^N} |x|^{-b}|u|^q \,dx \geq \widehat{C}.
$$
This jointly with \eqref{zbe} infers that $\sigma_0>0$. 

We next demonstrate that $\sigma_0>0$ is attained by a nonnegative $u \in X$. In view of Ekeland's variational principle, then there exists a Palais-Smale sequence $\{u_n\} \subset X$ such that $Q(u_n)=0$ for $E$ at the level $\sigma_0>0$. By applying \eqref{zbe}, we conclude that $\{u_n\}$ is bounded in $X$. From Lemmas \ref{lions} and \ref{embedding}, we then derive that there exists a nontrivial $u \in X$ such that $u_n \wto u$ in $X$ as $n \to \infty$ up to translations, because of $\sigma_0>0$. Further, we are able to obtain that $u \in X$ is a solution to \eqref{zequ}. This then yields that $Q(u)=0$ and $E(u) \geq \sigma_0$. Therefore, by Fatou's lemma, we have that 
\begin{align*}
\sigma_0 \leq E(u)=&E(u)-\frac{2}{N(p-2)+2b} Q(u) \\
&=\frac{N(p-2)-2(2-b)}{2\left(N(p-2)+2b\right)} \int_{\R^N} |\nabla u|^2 \,dx +\frac{N(p-q)}{q\left(N(p-2)+2b\right)}\int_{\R^N} |x|^{-b}|u|^q\,dx \\
& \leq \frac{N(p-2)-2(2-b)}{2\left(N(p-2)+2b\right)} \int_{\R^N} |\nabla u_n|^2 \,dx +\frac{N(p-q)}{q\left(N(p-2)+2b\right)}\int_{\R^N} |x|^{-b}|u_n|^q\,dx +o_n(1) \\
& = E(u_n)-\frac{2}{N(p-2)+2b} Q(u_n)+o_n(1) \\
&=E(u_n)+o_n(1)=\sigma_0+o_n(1).
\end{align*}
As a consequence, we get that $E(u)=\sigma_0$. Observe that $Q(|u|) \leq Q(u)=0$, then there exists a unique constant $0<t_{|u|} \leq 1$ such that $Q(|u|_{t_{|u|}})=0$, from which we are able to obtain that 
\begin{align*}
\sigma_0 \leq E(|u|_{t_{|u|}})&=E(|u|_{t_{|u|}})-\frac{2}{N(p-2)+2b} Q(|u|_{t_{|u|}}) \\
& \leq \frac{N(p-2)-2(2-b)}{2\left(N(p-2)+2b\right)} \int_{\R^N} |\nabla |u||^2 \,dx +\frac{N(p-q)}{q\left(N(p-2)+2b\right)}\int_{\R^N} |x|^{-b}|u|^q\,dx \\
&\leq \frac{N(p-2)-2(2-b)}{2\left(N(p-2)+2b\right)} \int_{\R^N} |\nabla u|^2 \,dx +\frac{N(p-q)}{q\left(N(p-2)+2b\right)}\int_{\R^N} |x|^{-b}|u|^q\,dx \\
&= E(u)-\frac{2}{N(p-2)+2b} Q(u) \\
&=E(u)=\sigma_0.
\end{align*}
It then follows that $E(|u|)=\gamma_0$ and $|u| \in X$ is also a minimizer to \eqref{zmin}. 

We now prove that $\sigma_0>0$ is achieved by some radially symmetric and decreasing function in $X$. To do this, we need to introduce the definition of polarization of measurable functions. Here we denote by $\mathcal{H}$ the family of all affine closed half spaces in $\R^N$ and denote by $\mathcal{H}_0$ the family of all closed half spaces in $\R^N$, i.e. $H \in \mathcal{H}_0$ if and only if $H \in \mathcal{H}$ and $0$ lies in the hyperplane $\partial H$. For $H \in \mathcal{H}$, we denote by $R_H: \R^N \to \R^N$ the reflection with respect to the boundary of $H$. We define the polarization of a measurable function $u : \R^N \to \R$ with respect to $H$ by
\begin{align*}
u_H(x):=\left\{
\begin{aligned}
\max\left\{u(x), u(R_H(x))\right\}, &\quad x \in H, \\
\min \left\{u(x), u(R_H(x))\right\}, &\quad x \in \R^N \backslash H.
\end{aligned}
\right.
\end{align*}
Let $u \in X$ be a nonnegative minimizer to \eqref{zmin}. In view of \cite[Lemmas 2.2-2.3]{BWW}, we can obtain that, for any $H \in \mathcal{H}_0$, 
$$
\int_{\R^N} |\nabla u_H|^2 \,dx =\int_{\R^N} |\nabla u|^2 \, dx
$$
and
$$
 \int_{\R^N} |x|^{-b}|u_H|^q\,dx=\int_{\R^N} |x|^{-b}|u|^q\,dx, \quad  \int_{\R^N} |x|^{-b}|u_H|^p\,dx=\int_{\R^N} |x|^{-b}|u|^p\,dx.
$$
As a consequence, we find that $u_H \in X$ is also a minimizer to \eqref{zmin} for any $H \in \mathcal{H}_0$. It follows from \cite[Theorem 1]{Sch} that there exist a sequence $\{H_n\} \subset \mathcal{H}_0$ and a sequence $\{u_n\} \subset X$ such that $u_n \to u^*$ in $L^r(\R^N)$ as $n \to \infty$ for any $ 1 \leq r <\infty$, where $u^*$ denotes the symmetric-decreasing rearrangement of $u$ and the sequence $\{u_n\}$ is defined by
$$
u_1:=u, \quad u_{n+1}:=(u_n)_{H_1H_2 \cdots H_{n+1}}.
$$
Therefore, we conclude that
$$
\int_{\R^N} |\nabla u^*|^2 \,dx  \leq \int_{\R^N} |\nabla u|^2 \, dx
$$
and
$$
 \int_{\R^N} |x|^{-b}|u^*|^q\,dx=\int_{\R^N} |x|^{-b}|u|^q\,dx, \quad  \int_{\R^N} |x|^{-b}|u^*|^p\,dx=\int_{\R^N} |x|^{-b}|u|^p\,dx.
$$
This yields that $E(u^*) \leq E(u)=\sigma_0$. Note that $Q(u)=0$, then $Q(u^*) \leq 0$. Hence we have that there exists a constant $0<t^* \leq 1$ such that $Q((u^*)_{t^*})=0$. Further, we are able to derive that $u^* \in X$ is a minimizer to \eqref{zmin} and
$$
\int_{\R^N} |\nabla u^*|^2 \,dx=\int_{\R^N} |\nabla u|^2 \, dx.
$$
In view of \cite[Theorem 1.1]{BZ}, then it is not difficult to deduce that $u$ is radially symmetric and decreasing up to translations. Thus the proof is completed.
\end{proof}

\begin{proof}[Proof of Proposition \ref{zestimate}]
Let $u \in X$ be the solution to \eqref{zequ} obtained in Lemma \ref{existzero}. Note first that $X \subset L^{2^*}(\R^N)$ for $N \geq 3$. Then, applying standard bootstrap arguments, we get that $u \in C(\R^N) \cap C^2(\R^N\backslash\{0\})$ and $u(x) \to 0$ as $|x| \to \infty$. Utilizing the maximum principle, we then have that $u>0$. 

We next establish the decay estimates of the solution to \eqref{zequ}. Let us first show that $u(x) \sim |x|^{-\alpha}$ as $|x| \to \infty$. To do this, we shall adapt some ingredients from \cite{DS}. For $R>0$, we know that $-\Delta (|x|^{2-N})=0$ in $\R^N \backslash B_R(0)$. On the other hand, since $q<p$ and $u(x) \to 0$ as $|x| \to \infty$, then $-\Delta u \leq 0$ in $\R^N \backslash B_R(0)$ for $R>0$ large enough. From the maximum principle, we then obtain that $u(x) \leq C |x|^{2-N}$ in $\R^N \backslash B_R(0)$. This means that $u(r) \leq C r^{2-N}$ for any $r>0$. We now prove that $u(r) \leq C r^{\frac{b-2}{q-2}}$ for any $r>0$. Since $u \in X$ is positive and radially symmetric, then \eqref{zequ} can be rewritten as
\begin{align}\label{zequ1}
-u_{rr}-\frac{N-1}{r}u_r+r^{-b}u^{q-1}=r^{-b}u^{p-1}.
\end{align}
It is simple to conclude that $u_{rr}>0$ for any $r>0$. This shows that $u_r$ is increasing and $\lim_{r \to \infty}|u_r(r)|=0$. Multiplying \eqref{zequ1} by $u_r$ and integrating on $[t, \infty)$, we have that
\begin{align} \label{d0}
-\int_t^{\infty} u_{rr}u_r \,dr +\int_t^{\infty}r^{-b}u^{q-1} u_r \, dr=\int_t^{\infty}r^{-b}u^{p-1} u_r \, dr + (N-1)\int_t^{\infty}\frac{u_r^2}{r}\,dr.
\end{align}
Note that $u(r) \to 0$ and $u_r(r) \to 0$ as $r \to \infty$. Therefore, from \eqref{d0}, we obtain that
\begin{align*}
f(t):=\frac 12 u_t^2-\frac 1 q t^{-b}u^q +\frac 1 p t^{-b}u^p&=\int_t^{\infty} \frac{b}{r} \left(\frac{N-1}{b} u_r^2-\frac 1 q r^{-b}u^q + \frac 1 p r^{-b} u^p\right) \, dr,
\end{align*}
from which we get that
\begin{align} \label{nonnegative}
\frac{d}{dr} f(r)= -\frac{b}{r}f(r)-\frac{2(N-1)-b}{2r} u_r^2.
\end{align}
This means that
$$
\frac{d}{dr} \left(r^b f(r) \right)=-\frac{(2(N-1)-b)r^{b-1}}{2} u_r^2<0.
$$
As a result, we derive that $r^b f(r)$ is decreasing for any $r>0$. It then follows that $f(r)$ is decreasing for any $r>0$. Note that $f(r) \to 0$ as $r \to \infty$. Hence there holds that $f(r)>0$ for any $r>0$ large enough. Thus we conclude that $u_r^2 \geq r^{-b} u^q/q$ for any $r>0$ large enough, because of $u(r) \to 0$ as $r \to \infty$ and $q<p$. Then we see that, for any $r>0$ large enough,
\begin{align} \label{d01}
\left(r^{\frac b 2} u^{\frac{2-q}{2}}\right)'=\frac b 2r^{\frac b 2 -1} u^{\frac {2-q}{2}} -\frac{q-2}{2} r^{\frac b 2}u^{-\frac{q}{2}} u_r \geq \frac{q-2}{2} r^{\frac b 2}u^{-\frac{q}{2}} |u_r| \geq \frac{q-2}{2 \sqrt {q}}.
\end{align}
On the other hand, there holds that, for any $r>0$,
\begin{align} \label{d02}
\left(r^{\frac b 2} u^{\frac{2-q}{2}}\right)'=\frac b 2r^{\frac b 2 -1} u^{\frac {2-q}{2}} -\frac{q-2}{2} r^{\frac b 2}u^{-\frac{q}{2}} u_r  \geq \frac b 2r^{\frac b 2 -1} u^{\frac {2-q}{2}}.
\end{align}
Combining \eqref{d01} and \eqref{d02}, we conclude that there exists a constant $C>0$ such that, for any $r>0$,
\begin{align} \label{d1}
\left(r^{\frac b 2} u^{\frac{2-q}{2}}\right)' \geq C.
\end{align}
Therefore, from \eqref{d1}, there holds that $ u(r) \leq C^{\frac{q-2}{2}} r^{\frac{b-2}{q-2}}$ for any $r>0$. Consequently, we have that $u(r) \leq C r^{-\alpha}$ for any $r>0$.

Let us define $v(r)=r^{\alpha}u(r)$ for $r>0$. We are going to prove that there exists a constant $l>0$ such that $v(r) \to l$ as $r \to \infty$. From the discussions above, we see that $v$ is bounded. In addition, by \eqref{zequ1}, it is not hard to verify that $v$ solves the following equation,
\begin{align} \label{zequ2}
v_{rr}-\frac{2\alpha +1-N}{r} v_r=\frac{\alpha \left(N-\alpha -2\right)}{r^2} v+r^{-b+\alpha(2-q)}v^{q-1}-r^{-b+\alpha(2-p)}v^{p-1}.
\end{align}
First we consider the case that $\alpha=N-2$. In this case, by \eqref{zequ2}, then $v$ satisfies the equation
\begin{align*}
v_{rr}-\frac{N-3}{r} v_r=r^{-b+(N-2)(2-q)}v^{q-1}-r^{-b+(N-2)(2-p)}v^{p-1}.
\end{align*}
Define $w(t):=v(r)$ for $t=\beta r^{N-2}/(N-2)$ and 
$
\beta=\left(N-2\right)^{\frac{b-2+q(N-2)}{b-2+(N-2)(q-2)}}>0.
$
Therefore, we have that $w$ enjoys the equation
\begin{align}\label{iddd1}
w''=t^{-\frac{b-2+q(N-2)}{N-2}}w^{q-1}-\left(N-2\right)^{\frac{(N-2)^2(p-q)}{b-2+(N-2)(q-2)}}t^{-\frac{b-2+p(N-2)}{N-2}}w^{p-1}.
\end{align}
It then follows that $w''(t) >0$ for any $t>0$ large enough, because of $p<q$ and $0<b-2+q(N-2)<b-2+q(N-2)$.
Note that $w$ is bounded, then $w'(t) \to 0$ as $t \to \infty$. As a consequence, by \eqref{iddd1}, we get that
$$
-w'(t)=\int_t^{\infty} s^{-\frac{b-2+q(N-2)}{N-2}}w^{q-1}(s)-\left(N-2\right)^{\frac{(N-2)^2(p-q)}{b-2+(N-2)(q-2)}}s^{-\frac{b-2+p(N-2)}{N-2}}w^{p-1}(s) \,ds.
$$
This gives that $w'(t) <0$ for any $t>0$ large enough. Furthermore, there holds that
\begin{align*}
w(t)&=\int_t^{\infty}\int_{\tau}^{\infty} s^{-\frac{b-2+q(N-2)}{N-2}}w^{q-1}(s)-\left(N-2\right)^{\frac{(N-2)^2(p-q)}{b-2+(N-2)(q-2)}}s^{-\frac{b-2+p(N-2)}{N-2}}w^{p-1}(s) \,ds d\tau \\
&=\int_t^{\infty}(s-t)\left(s^{-\frac{b-2+q(N-2)}{N-2}}w^{q-1}(s)-\left(N-2\right)^{\frac{(N-2)^2(p-q)}{b-2+(N-2)(q-2)}}s^{-\frac{b-2+p(N-2)}{N-2}}w^{p-1}(s)\right) \,ds
\end{align*}
Therefore, we find that, for $t>0$ large enough, 
\begin{align} \label{below}
\begin{split}
w(t)&=\int_{t}^{\infty}(s-t)\left(s^{-\frac{b-2+q(N-2)}{N-2}}w^{q-1}(s)-\left(N-2\right)^{\frac{(N-2)^2(p-q)}{b-2+(N-2)(q-2)}}s^{-\frac{b-2+p(N-2)}{N-2}}w^{p-1}(s)\right) \,ds \\
& \leq w^{q-1}(t) \int_{t}^{\infty}(s-t)s^{-\frac{b-2+q(N-2)}{N-2}}\, ds \\
&=w^{q-1}(t) t^{2-\frac{b+q(N-2)}{N-2}} \mathcal{B}\left(\frac{b-2+q(N-2)}{N-2}-2, 2\right),
\end{split}
\end{align}
where $\mathcal{B}(a, b)$ denotes the Beta function for $a,b>0$. If $\beta=N-2$ and $(2-b)/(q-2) \neq N-2$, then
$$
\frac{b-2+q(N-2)}{N-2} > 2,
$$
It then follows from \eqref{below} that there exists a constant $l>0$ such that $w(t) \to l$ as $t \to \infty$. This in turn leads to such that $v(r) \to l$ as $r \to \infty$.

Next we consider the case $\alpha>N-2$.
If $N \geq 4$, we define $w(t):=v(r)$ for $r=\exp\left(\frac{t}{|N-2\alpha -1|}\right)$ and $t>0$. In view of \eqref{zequ2}, we then see that 
\begin{align} \label{idd1}
w''=&\frac{1}{\left(N-2\alpha -1\right)^2} f(w),
\end{align}
where
$$
f(w):=\alpha\left(N-\alpha-2\right) w+\exp\left(\frac{\left(2-b+\alpha(2-q)\right)t}{|N-2\alpha -1|}\right) w^{q-1}-\exp\left(\frac{\left(2-b+\alpha(2-p)\right)t}{|N-2\alpha -1|}\right)w^{p-1}.
$$
Integrating \eqref{idd1} on $[t, \infty)$, we then get that
\begin{align} \label{idd2}
-w'(t) =\frac{1}{\left(N-2\alpha -1\right)^2} \exp{\left(\frac{(2+2\alpha-N)t}{|N-2\alpha -1|}\right)}\int_t^{\infty}\exp{\left(-\frac{(2+2\alpha-N)s}{|N-2\alpha -1|}\right)}f(w(s)) \,ds.
\end{align}
Note that $\alpha=\max\{(2-b)/(q-2), N-2\}$. If $ \alpha > N-2$, then $N-2\alpha-1<0$ and
$$
\frac{2-b+\alpha(2-p)}{|N-2\alpha-1|}<\frac{2-b+\alpha(2-q)}{|N-2\alpha-1|} \leq 0.
$$
Then there holds that, for any $t>0$ large enough,
$$
f(w(t)) \leq -\frac{\alpha\left(\alpha-N+2\right)}{2} w(t)<0.
$$
Taking into account \eqref{idd2}, we then obtain that $w'(t)>0$ for any $t>0$ large enough. Due to $w>0$, then $w(t) \to l$ as $t \to \infty$. This shows that $v(r) \to l$ as $r \to \infty$.
In this case, if $N=3$, then \eqref{zequ2} reduces to
\begin{align} \label{vequ}
v_{rr}-\frac{2(\alpha-1)}{r} v_r=\frac{\alpha \left(1-\alpha\right)}{r^2} v+r^{-b+\alpha(2-q)}v^{q-1}-r^{-b+\alpha(2-p)}v^{p-1}.
\end{align}
Utilizing a similar way as before, we can also prove that  $v(r) \to l$ as $ r \to \infty$ for some $l>0$ when $\alpha=1$ or $\alpha>1$. Hence we have the desired result, i.e. $u(x) \sim |x|^{-\alpha}$ as $|x| \to \infty$.

Let us now turn to demonstrate that $u(x) \sim |x|^{2-N} \left(\ln |x| \right)^{\frac{2-N}{2-b}}$ as $|x| \to \infty$ for $q = (2N-2-b)/(N-2)$. To do this, we shall follow some ideas from \cite{DSW}. Let us first define $v(r):=r^{N-2}u(r)$ for any $r>0$. In virtue of \eqref{zequ2}, we immediately know that $v$ satisfies the following equation,
\begin{align*}
v_{rr}-\frac{N-3}{r} v_r=r^{-2}v^{q-1}-r^{-b+(N-2)(2-p)}v^{p-1}.
\end{align*}
Define $w(t):=v(r)$ and $r=e^t$ for any $t>0$. Then we find that
\begin{align*}
w''-(N-2)w'=w^{q-1}-e^{(2-b)t+(N-2)(2-p)t}w^{p-1}.
\end{align*}
It then follows that
\begin{align} \label{zequ11}
\left(w' e^{-(N-2)t}\right)'=e^{-(N-2)t}\left(w^{q-1}-e^{(2-b)t+(N-2)(2-p)t}w^{p-1}\right).
\end{align}
By integrating \eqref{zequ11} on $[t, \infty)$, we then have that 
$$
-w'(t)=e^{(N-2)t}\int_{t}^{\infty}e^{-(N-2)s}\left(w^{q-1}-e^{(2-b)s+(N-2)(2-p)s}w^{p-1}\right) \, ds.
$$
This then yields that $w'(t)<0$ for any $t>0$ large enough. Therefore, we get that $-w'(t) \leq C w^{q-1}(t)$ for any $t>0$ large enough. As a consequence, we have that $\left(w^{2-q}(t)\right)' \leq C(q-2)$ for any $t>0$ large enough, from which we obtain that 
$$
w(t) \geq \left(C(q-2) t +w^{2-q}(t_0)\right)^{\frac {1}{2-q}} \geq \widetilde{C} t^{\frac{2-N}{2-b}}, \quad t>t_0>0,
$$
because of $q = (2N-2-b)/(N-2)$. This yields that, for any $r>0$ large enough,
$$
u(r) \geq \widetilde{C} r^{2-N} \left(\ln r\right)^{\frac{2-N}{2-b}}.
$$
In the following, we are going to show the upper bound of $u$. To this end, we first set 
$$
S(r):=kr^{2-N} \left(\ln r\right)^{\frac{2-N}{2-b}}, \quad r>0,
$$ 
where $k>0$ is a constant defined by
$$
k:=\left(\frac{2-b}{(N-2)^2}\right)^{\frac{N-2}{2-b}}.
$$
From direct computations, we see that
$$
-\Delta S + |x|^{-b} S^{q-1}=-\frac{(N-b) |x|^{-b}}{(2-b)(N-2)} \frac{S^{q-1}}{\ln |x|}.
$$
Let $\eta>0$ be a constant to be determined later. It yields that
$$
-\Delta (\eta S) +  |x|^{-b} (\eta S)^{q-1}= |x|^{-b}(\eta^{q-1}-\eta)S^{q-1} - \frac{\eta (N-b) |x|^{-b}}{(2-b)(N-2)}\frac{S^{q-1}}{\ln |x|}.
$$
Then we obtain that
\begin{align*}
&-\Delta(u-\eta S) + \frac{ |x|^{-b}\left(u^{q-1}-(\eta S)^{q-1}\right)}{u-\eta S} \left(u-\eta S\right)=-\Delta u+ |x|^{-b} u^{p-1} + \left(\Delta (\eta S)- |x|^{-b}(\eta S)^{q-1}\right) \\
& = |x|^{-b}u^{p-1} - |x|^{-b}(\eta^{q-1}-\eta)S^{q-1} + \frac{\eta (N-b) |x|^{-b}}{(2-b)(N-2)}\frac{S^{q-1}}{\ln |x|} \\
&= |x|^{-b}S^{q-1} \left(\frac{u^{p-1}}{S^{q-1}}-(\eta^{q-1}-\eta)+ \frac{\eta (N-b)}{(2-b)(N-2)\ln |x|}\right).
\end{align*}
Note that $u(x) \leq C |x|^{-\alpha}$ for $x \in \R^N$ and $q<p$,  then
$$
\frac{u^{p-1}(x)}{S^{q-1}(x)} \to 0 \quad \mbox{as} \,\, |x| \to \infty.
$$
As a consequence, there exists a constant $R>0$ large enough such that
$$
-\Delta(u-\eta S) +c(x) \left(u-\eta S\right) <0,
$$
where $\eta>0$ is a constant such that $\eta<\eta^{q-1}$ and 
$$
c(x):=\frac{|x|^{-b}\left(u^{q-1}(x)-(\eta S)^{q-1}(x)\right)}{u(x)-\eta S(x)}, \quad x \in \R^N
$$
Using the maximum principle, we then have that $u(x) \leq \widehat{C} S(x)$ for any $|x| \geq R$. This completes the proof.
\end{proof}

\begin{prop} \label{limit11}
Let $p>2+2(2-b)/N$ and $\mu=-1$. If $q \geq 2+2(2-b)/N$ and $ N=3$ or $q>2+2(2-b)/N$ and $N=4$, then $\sigma(c) \to \gamma_0$ as $c \to \infty$ and $\sigma(c) > \sigma_0$ for any $c>0$.
\end{prop}
\begin{proof}
Let us first prove that $\sigma(c) \to \sigma_0$ as $c \to \infty$. From the definition of $\sigma_0$, we first have that $\sigma(c) \geq \sigma_0$ for any $c>0$. Let $u \in X$ be a solution to \eqref{zequ} with $E(u)=\sigma_0$ and $Q(u)=0$. It follows from Proposition \ref{zestimate} that $u \not\in L^2(\R^N)$. Let $\chi \in C_0^{\infty}(\R^N, [0,1])$ be a cut-off function such that $\chi(x)=1$ for $|x| \leq 1$ and $\chi(x)=0$ for $|x| \geq 2$. For $R>0$, we define $u_R(x):=u(x) \chi_R(x)$, where $\chi_R(x):=\chi(x/R)$ for $x \in \R^N$. It is simple to verify that $\|u_R\|_2 \to \infty$ as $R \to \infty$. Moreover, by Proposition \ref{zestimate}, we are able to check that $\|\nabla u_R\|_2=\|\nabla u\|_2 +o_R(1)$ and
$$
\int_{\R^N}|x|^{-b}|u_R|^q \, dx=\int_{\R^N}|x|^{-b}|u|^q \, dx+o_R(1), \quad \int_{\R^N}|x|^{-b}|u_R|^p \, dx=\int_{\R^N}|x|^{-b}|u|^p \, dx +o_R(1).
$$
where $o_R(1) \to 0$ as $R \to \infty$. Since $Q(u)=0$, then there exists a unique constant $t_R>0$ satisfying $t_R \to 1$ as $R \to \infty$ such that $Q((u_R)_{t_R})=0$. Observe that
\begin{align*}
\sigma(\|u_R\|_2^2) &\leq E((u_R)_{t_R})=E((u_R)_{t_R})-\frac{2}{N(p-2)+2b} Q_0((u_R)_{t_R})\\ 
&=\frac{N(p-2)-2(2-b)}{2\left(N(p-2)+2b\right)} \int_{\R^N} |\nabla (u_R)_{t_R}|^2 \,dx +\frac{N(p-q)}{q\left(N(p-2)+2b\right)}\int_{\R^N} |x|^{-b}|(u_R)_{t_R}|^q\,dx\\
&=\frac{N(p-2)-2(2-b)}{2\left(N(p-2)+2b\right)} \int_{\R^N} |\nabla u|^2 \,dx +\frac{N(p-q)}{q\left(N(p-2)+2b\right)}\int_{\R^N} |x|^{-b}|u|^q\,dx +o_R(1) \\
&=E(u)-\frac{2}{N(p-2)+2b} Q(u) +o_R(1) \\
&=\sigma_0 +o_R(1).
\end{align*}
This shows that $\sigma(c) \to \sigma_0$ as $c \to \infty$. We next assert that $\sigma(c) > \sigma_0$ for any $c>0$. Let us argue by contradiction that there exists a constant $c>0$ such that $\sigma(c)=\sigma_0>0$. At this point, following the proof of Theorem \ref{thm5}, we are able to infer that there exists a Palais-Smale sequence $\{u_n\} \subset P(c)$ for $E$ restricted on $S(c)$ at the level $\sigma(c)>0$. This further leads to $\{u_n\} \subset H^1(\R^N)$ is bounded and there exists a nontrivial $u \in H^1(\R^N)$ such that $u_n \wto u$ in $H^1(\R^N)$ as $n \to \infty$, $Q(u)=0$ and $E(u)=\sigma(c)$. It readily suggests that $E(u)=\sigma_0$. However, in view of Proposition \ref{zestimate}, we have that $u \not\in L^2(\R^N)$. We then reach a contradiction. Thus the desired conclusion holds true and the proof is completed.
\end{proof}


\begin{prop} \label{limit12}
Let $p>2+2(2-b)/N$ and $\mu=-1$. If $2<q<2+2(2-b)/N$ and $ N=3$ or $2<q \leq 2+2(2-b)/N$ and $N=4$ or $N \geq 5$, then there exists a constant $c_{\infty}>0$ such that $\sigma(c) =\sigma_0$ for any $c \geq c_{\infty}$.
\end{prop}
\begin{proof}
Let $u \in X$ be a solution to \eqref{zequ} with $E(u)=\sigma_0$ and $Q(u)=0$. From Proposition \ref{zestimate}, we have that $u \in L^2(\R^N)$. Using Proposition \ref{nonincreasing}, we then derive that $\sigma(c) \leq \sigma(c_{\infty}) \leq E(u)=\sigma_0$ for any $c \geq c_{\infty}$, where $c_{\infty}:=\|u\|_2^2>0$. On the other hand, by the definitions of $\sigma(c)$ and $\sigma_0$, we know that $\sigma(c) \geq \sigma_0$ for any $c>0$. As a result, we get the conclusion and the proof is completed.
\end{proof}

\begin{prop} \label{limit13}
Let $N \geq 3$, $p=2+2(2-b)/N$ and $\mu=-1$, then there exists a constant $c_{\infty}>c_1$ such that $\sigma(c) =\sigma_0$ for any $c \geq c_{\infty}$.
\end{prop}
\begin{proof}
Let $u \in X$ be a solution to \eqref{zequ} with $E(u)=\sigma_0$ and $Q(u)=0$. In this case, by Proposition \ref{zestimate}, there also holds that $u \in L^2(\R^N)$. Then, by a similar way as the proof of Proposition \ref{limit12}, we can complete the proof.
\end{proof}

\begin{proof}[Proof of Proposition \ref{prop2}]
The proof of of Proposition \ref{prop2} follows directly from Propositions \ref{nonincreasing}, \ref{continuous}, \ref{limit1}, \ref{limit2}, \ref{limit3}, \ref{limit11}, \ref{limit12} and \ref{limit13}.
\end{proof}

\section{Dynamical behaviors of solutions to the dispersive equation} \label{dynamics}

In this section, we are going to study dynamical behaviors of solutions to the Cauchy problem for \eqref{equt}. First of all, we present the local well-posedness of solutions to the problem.

\begin{lem} \label{localwp}
Let $N \geq 1$, $2<q<p<{2(N-b)}/{(N-2)^+}$, $0<b<\min\{2, N\}$ and $\mu =\pm 1$, then, for any $\psi_0 \in H^1(\R^N)$, there exist a constant $T>0$ and a unique solution $\psi(t) \in C([0, T); H^1(\R^N))$ to \eqref{equt} with initial datum $\psi_0$, satisfying the conservation of mass and energy, i.e. for any $ t \in [0, T)$,
$$
\|\psi(t, \cdot)\|_2=\|\psi(0, \cdot)\|_2, \quad E(\psi(t, \cdot))=E(\psi(0, \cdot)).
$$
In addition, the solution map $\psi_0 \mapsto \psi$ is continuous from $H^1(\R^N)$ to $C([0, T); H^1(\R^N))$. There also holds that either $T< \infty$ or $lim_{t \to T^-} \|\nabla \psi(t)\|_2=\infty.$
\end{lem}
\begin{proof} 
To prove the local well-posedness of solutions to the problem, we shall take advantage of \cite[Theorem 4.3.1]{Ca} instead of Strichartz estimates, {see also Appendix K in \cite{GS}.} We write
$$
|x|^{-b}|\psi|^{q-2} =\chi_{B_1(0)}(x)|x|^{-b}|\psi|^{q-2}  +\chi_{\R^N \backslash B_1(0)}(x)|x|^{-b}|\psi|^{q-2}
$$ 
and
$$
|x|^{-b}|\psi|^{p-2} =\chi_{B_1(0)}(x)|x|^{-b}|\psi|^{p-2}  +\chi_{\R^N \backslash B_1(0)}(x)|x|^{-b}|\psi|^{p-2}.
$$
Then it is trivial to check that the nonlinearities satisfy the assumptions of \cite[Theorem 4.3.1]{Ca}, then the proof is completed.
\end{proof}

\begin{proof}[Proof of Theorem \ref{thmstable}]
For simplicity, we only show the proof of orbital stability of the set of minimizers under the assumptions Theorem \ref{thm1}. Define
$$
\mathcal{G}(c):=\{u \in S(c) : E(u) =m(c)\}.
$$
In view of Theorem \ref{thm1}, then $\mathcal{G} \neq \emptyset$. We shall argue by contradiction that $\mathcal{G}(c)$ is unstable for some $c>0$. In view of Definition \ref{def1}, then there exist a constant $\eps_0>0$ and  a sequence $\{t_n\} \subset \R^+$ such that $\inf_{u \in \mathcal{G}(c)}\|\psi_n(t_n)-u\| \geq \eps_0$, where $\psi_n(t) \in C([0, T); H^1(\R^N))$ is the solution to the Cauchy problem for \eqref{equt} with initial datum $\psi_{0,n} \in S(c_n)$ satisfying $E(\psi_{0,n})=m(c)+o_n(1) $ for $c_n=c+o_n(1)$. Invoking the conservation laws, we then have that there exists a sequence $\{u_n\} \subset S(c)$ satisfying $\|u_n-\psi_n(t_n)\|=o_n(1)$ such that $E(u_n)=m(c)+o_n(1)$. As a consequence of Theorem \ref{thm1}, we know that $\{u_n\}$ is compact in $H^1(\R^N)$ up to translations. It then follows that $\inf_{u \in \mathcal{G}(c)}\|\psi_n(t_n)-u\| =o_n(1)$. Thus we reach a contradiction and the proof is completed.
\end{proof}

\begin{proof}[Proof of Theorem \ref{globalexistence}]
Let $\psi \in C([0, T); H^1(\R^N))$ be the solution to the Cauchy problem for \eqref{equt} with initial datum $\psi_0 \in \mathcal{K}^+(c)$. We shall argue by contradiction that $\psi$ does not exist globally in time, i.e. $T<\infty$. From Lemma \ref{localwp}, then $\|\nabla \psi(t)\|_2 \to \infty$ as $t \to T^-$. Arguing as the discussions of coerciveness of $E$ restricted on $P(c)$ and using the conservation of energy, we can show that $Q(\psi(t)) \to -\infty$ as $t \to T^-$. Note that $Q(\psi_0)>0$, then there exists a constant $t_0 \in (0, T)$ such that $Q(\psi(t_0))=0$, This leads to $E(\psi(t_0)) \geq \gamma(c)$. We then reach a contradiction, because of $E(\psi(t_0))=E(\psi_0)<\gamma(c)$. Thus the proof is completed.
\end{proof}

To discuss blowup of radially symmetric solutions to the Cauchy problem for \eqref{equt}, we need to introduce a localized virial quantity. Let $\chi : \R^N \to \R$ be a radial smooth function with $\nabla^k \chi \in L^{\infty}(\R^N)$ for $1 \leq k \leq 4$ and such that
\begin{align*}
\chi(r):=
\left\{
\begin{aligned}
&\frac{r^2}{2}, &\quad 0 \leq r \leq 1,\\
&\textnormal{const.}, &\quad r \geq 2,
\end{aligned}
\right.
\end{align*}
and $ \chi''(r) \leq 1 $ for $ r \geq 0$. For $R>0$, we define
$$
\chi_R(r):=R^2 \chi\left( \frac r R\right), \quad r \geq 0.
$$
It is simple to verify that
\begin{align} \label{chiprop}
1-\chi_R''(r) \geq 0, \quad 1-\frac{\chi_R'(r)}{r} \geq 0, \quad N-\Delta \chi_R(r) \geq 0, \quad  r \geq 0.
\end{align}
For $\psi \in H^1(\R^N)$, we now define the localized virial quantity by
$$
V_{\chi_R}[\psi]:=2 \textnormal{Im} \int_{\R^N} \overline{\psi} \nabla \chi_R \cdot \nabla \psi \, dx.
$$
Clearly, $V_{\chi_R}[\psi]$ is well-defined for any $\psi \in H^1(\R^N)$.

\begin{lem}\label{virial}
Let $N \geq 2$ and $\psi \in C([0, T); H^1(\R^N))$ be a radially symmetric solution to \eqref{equt} with radially symmetric initial datum $\psi_0 \in H^1(\R^N)$, then, for any $t \in [0, T)$,
\begin{align*}
\frac{d}{dt} V_{\chi_R}[\psi(t)]=4 Q(\psi(t)) + O(R^{-2}) +O\left(R^{-b-\frac{(N-1)(q-2)}{2}}\|\nabla \psi(t)\|_2^{\frac{q-2}{2}}\right)+O\left(R^{-b-\frac{(N-1)(p-2)}{2}}\|\nabla \psi(t)\|_2^{\frac{p-2}{2}}\right).
\end{align*}
\end{lem}
\begin{proof}
For simplicity, we shall denote $\psi(t)$ by $\psi$ in what follows. Observe first that
\begin{align*}
\frac{d}{dt} V_{\chi_R}[\psi]&=2 \textnormal{Im} \int_{\R^N} \overline{\partial_t \psi} \nabla \chi_R \cdot \nabla \psi \, dx +2 \textnormal{Im} \int_{\R^N} \overline{\psi} \nabla \chi_R \cdot \nabla \partial_t \psi \, dx \\
&=2 \textnormal{Im} \int_{\R^N} \overline{\partial_t \psi} \nabla \chi_R \cdot \nabla \psi \, dx -2 \textnormal{Im} \int_{\R^N} \partial_t \psi  \nabla \cdot(\overline{\psi} \nabla \chi_R)\, dx \\
&= 2\textnormal{Im} \int_{\R^N} \overline{\partial_t \psi} \left(\nabla \chi_R \cdot \nabla \psi + \nabla \cdot(\nabla \chi_R \psi )\right)\, dx\\
&=2\textnormal{Re} \int_{\R^N} \overline{\textnormal{i}\partial_t \psi} \left(\nabla \chi_R \cdot \nabla \psi + \nabla \cdot(\nabla \chi_R \psi )\right)\, dx.
\end{align*}
Since $\psi(t) \in C([0, T); H^1(\R^N))$ is a solution to \eqref{equt}, then
\begin{align*}
\frac{d}{dt} V_{\chi_R}[\psi(t)]&= -2 \textnormal{Re} \int_{\R^N}\Delta \overline{\psi}\left(\nabla \chi_R \cdot \nabla \psi + \nabla \cdot\left(\nabla \chi_R \psi \right)\right)\, dx  \\
&\quad -2 \textnormal{Re} \int_{\R^N}\left(\mu |x|^{-b}|\psi|^{q-2}\overline{\psi} + |x|^{-b}|\psi|^{p-2} \overline{\psi}\right)\left(\nabla \chi_R \cdot \nabla \psi + \nabla \cdot(\nabla \chi_R \psi \right)\, dx \\
&:=I_1+I_2.
\end{align*}
From the divergence theorem and the identity
$$
{\partial_{kl}^2} \chi_R=\left(\delta_{kl} -\frac{x_k x_l}{r^2}\right) \frac{\partial_r \chi_R}{r} +\frac{x_k x_l}{r^2} \partial^2_r\chi_R,
$$
{where $\delta_{kl}=1$ if $k=l$ and $\delta_{kl}=0$ if $k \neq l$}, we get that
\begin{align} \label{i1}
\begin{split}
I_1&=4\sum_{k,l=1}^N \int_{\R^N}\partial_k \overline{\psi} \partial_l \psi (\partial_{kl}^2 \chi_R) \, dx-\int_{\R^N}\left(\Delta^2 \chi_R\right) |\psi|^2 \,dx \\
&=4 \int_{\R^N} |\nabla \psi|^2 (\partial_r^2 \chi_R) \, dx-\int_{\R^N} \left(\Delta^2 \chi_R\right) |\psi|^2 \,dx \\
&=4 \int_{\R^N} |\nabla \psi|^2 \,dx -4\int_{\R^N} |\nabla \psi|^2 (1-\partial_r^2 \chi_R) \, dx-\int_{\R^N} \left(\Delta^2 \chi_R\right) |\psi|^2 \,dx \\
&=4 \int_{\R^N} |\nabla \psi|^2 \,dx -4\int_{\R^N} |\nabla \psi|^2 (1-\partial_r^2 \chi_R) \, dx + O(R^{-2}).
\end{split}
\end{align}
In addition, we have that
\begin{align} \label{i2}
\begin{split}
I_2=&-\frac{2 \mu (q-2)}{q} \int_{\R^N} |x|^{-b} |\psi|^{q} \Delta \chi_R  \, dx - \frac{4\mu b}{q} \int_{\R^N} |x|^{-b-2} \left(x \cdot \nabla \chi_R \right) |\psi|^q\, dx \\
& \quad -\frac{2(p-2)}{p} \int_{\R^N} |x|^{-b} |\psi|^{p} \Delta \chi_R  \, dx - \frac{4b}{p} \int_{\R^N} |x|^{-b-2} \left(x \cdot \nabla \chi_R\right) |\psi|^p\, dx \\
&=-\frac{2 \mu (N(q-2)+2b)}{q}\int_{\R^N} |x|^{-b} |\psi|^{q} \, dx-\frac{2 \mu (q-2)}{q}\int_{\R^N} \left(\Delta \chi_R-N\right) |x|^{-b} |\psi|^{q}  \, dx \\
&\quad -\frac{2(N(p-2)+2b)}{p}\int_{\R^N} |x|^{-b} |\psi|^{p} \, dx-\frac{2(p-2)}{p}\int_{\R^N} \left(\Delta \chi_R-N\right) |x|^{-b} |\psi|^{p}  \, dx \\
&\quad - \frac{4\mu b}{q} \int_{\R^N} |x|^{-b-2} \left(\left(x \cdot \nabla \chi_R \right)-|x|^2\right) |\psi|^q\, dx- \frac{4 b}{p} \int_{\R^N} |x|^{-b-2} \left(\left(x \cdot \nabla \chi_R \right)-|x|^2\right) |\psi|^p\, dx \\
&=-\frac{2 \mu (N(q-2)+2b)}{q}\int_{\R^N} |x|^{-b} |\psi|^{q} \, dx -\frac{2(N(p-2)+2b)}{p}\int_{\R^N} |x|^{-b} |\psi|^{p} \, dx \\
& \quad + O\left(R^{-b-\frac{(N-1)(q-2)}{2}}\|\nabla \psi\|_2^{\frac{q-2}{2}}\right)+O\left(R^{-b-\frac{(N-1)(p-2)}{2}}\|\nabla \psi\|_2^{\frac{p-2}{2}}\right),
\end{split}
\end{align}
where we used the definition of $\chi_R$ and the following Strauss inequality for the last equality,
\begin{align} \label{sti}
|x|^{\frac{N-1}{2}} |\psi(x)| \leq 2 \|\psi\|_2^{\frac 12} \|\nabla \psi\|_2^{\frac 12}, \quad x \neq 0.
\end{align}
Note that $1-\partial_r^2 \chi_R \geq 0$ by \eqref{chiprop}. As a consequence, combining \eqref{i1} and \eqref{i2}, we now conclude that
\begin{align*}
\frac{d}{dt} V_{\chi_R}[\psi(t)] & \leq 4 \int_{\R^N} |\nabla \psi|^2 \,dx-\frac{2 \mu (N(q-2)+2b)}{q}\int_{\R^N} |x|^{-b} |\psi|^{q} \, dx -\frac{2(N(p-2)+2b)}{p}\int_{\R^N} |x|^{-b} |\psi|^{p} \, dx \\
& \quad + O(R^{-2}) +O\left(R^{-b-\frac{(N-1)(q-2)}{2}}\|\nabla \psi\|_2^{\frac{q-2}{2}}\right)+O\left(R^{-b-\frac{(N-1)(p-2)}{2}}\|\nabla \psi\|_2^{\frac{p-2}{2}}\right) \\
&=4 Q(\psi) + O(R^{-2}) +O\left(R^{-b-\frac{(N-1)(q-2)}{2}}\|\nabla \psi\|_2^{\frac{q-2}{2}}\right)+O\left(R^{-b-\frac{(N-1)(p-2)}{2}}\|\nabla \psi\|_2^{\frac{p-2}{2}}\right).
\end{align*}
This completes the proof.
\end{proof}

\begin{lem} \label{viriall2}
Let $N \geq 1$ and $\psi \in C([0, T); H^1(\R^N))$ be the solution to \eqref{equt} with initial datum $\psi_0 \in H^1(\R^N)$ satisfying $|x| \psi_0 \in L^2(\R^N)$, then, for any $t \in [0, T)$,
\begin{align*}
\frac{d}{dt} V_{\infty}[\psi(t)]=4 Q(\psi(t)),
\end{align*}
where
$$
V_{\infty}[\psi]:=2\textnormal{Im} \int_{\R^N} \overline{\psi} x \cdot \nabla \psi \, dx.
$$
\end{lem}
\begin{proof}
Here we have that $\chi_R(x)=|x|^2/2$. Following the computations presented in the proof of Lemma \ref{virial}, we then derive the desired result. Thus the proof is completed. 
\end{proof}

\begin{proof}[Proof of Theorem \ref{blowup}]
Let $u \in S(c)$ be a ground state to \eqref{equ}-\eqref{mass} at the level $\gamma(c)>0$, then $u_{\tau} \in \mathcal{K}^-(c)$ for any $\tau >1$. It follows that $\mathcal{K}^-(c) \neq \emptyset$. Let $\psi_0 \in \mathcal{K}^-(c)$, then $\psi(t) \in \mathcal{K}^-(c)$ for any $t\in [0, T)$. Indeed, if not, then there would exist a constant $t_0 \in (0, T)$ such that $\psi(t_0) \not\in \mathcal{K}^-(c)$. Since $\psi_0 \in \mathcal{K}^-(c)$, by the conservation laws, then $Q(\psi(t_0)) \geq 0$. Due to $Q(\psi_0)<0$, then there exists a constant $t_* \in (0, t_0]$ such that $Q(\psi(t_*))=0$. This then leads to $E(\psi(t_*)) \geq \gamma(c)$. However, from the conservation laws, there holds that $E(\psi(t))=E(\psi_0)<\gamma(c)$ for any $t \in [0, T)$. We then reach a contradiction. This in turns indicates that $\mathcal{K}^-(c)$ is invariant under the flow of the Cauchy problem for \eqref{equt}.

For simplicity, we shall write $\psi=\psi(t)$ in the following. Since $\psi_0 \in \mathcal{K}^-(c)$, then $Q(\psi)<0$ by the fact that $\mathcal{K}^-(c)$ is invariant under the flow of the Cauchy problem for \eqref{equt}. Hence we find that there exists a constant {$0<\tau_{\psi}<1$ such that $Q(\psi_{\tau_{\psi}})=0$}. Moreover, the function {$\tau \mapsto E(\psi_{\tau})$ is concave on$[\tau_{\psi}, 1]$}. As a result, we see that
$$
{E(\psi_{\tau_{\psi}})-E(\psi) \leq (\tau_{\psi}-1) \frac{d}{dt} E(\psi_{\tau})\mid_{\tau=1}=(\tau _{\psi}-1) Q(\psi)}.
$$
It then follows that
\begin{align}\label{negative}
{Q(\psi) \leq (1-\tau_{\psi}) Q(\psi) \leq E(\psi)-E(\psi_{\tau_{\psi}}) \leq E(\psi)-\gamma(c).}
\end{align}
If $|x| \psi_0 \in L^2(\R^N)$, by Lemma \ref{viriall2} and the conservation of energy, then
$$
\frac{d}{dt}V_{\infty}[\psi] \leq 4(E(\psi)-\gamma(c))=4(E(\psi_0)-\gamma(c))<0.
$$
Therefore, we have that $\psi$ blows up in finite time. Next we consider the case that $\psi_0$ is radially symmetric. 
We shall assume that $\psi$ exists globally in time. In light of \eqref{negative}, we first get that
\begin{align} \label{bddg1}
\|\nabla \psi\|_2^2 \leq \frac{\mu(N(q-2)+2b)}{2q}\int_{\R^N} |x|^{-b} |\psi|^{q} \, dx +\frac{N(p-2)+2b}{2p}\int_{\R^N} |x|^{-b} |\psi|^{p} \, dx + E(\psi)-\gamma(c).
\end{align}
Since $2<q < p \leq 6$, from Lemma \ref{virial} and Young's inequality, we then infer that, for any $\eps>0$, there exists a constant $R>0$ large enough such that
\begin{align*}
\frac{d}{dt} V_{\chi_R}[\psi] &\leq 4Q(\psi) + \eps \|\nabla \psi\|_2^2 + \eps \\
&= 8 E(\psi)-\frac{2\mu(N(q-2)+2b-4)}{q}\int_{\R^N} |x|^{-b} |\psi|^{q} \, dx-\frac{2(N(p-2)+2b-4)}{q}\int_{\R^N} |x|^{-b} |\psi|^{p} \, dx \\
& \quad + \eps \|\nabla \psi\|_2^2 +\eps.
\end{align*}
Combining \eqref{bddg1}, we further derive that
\begin{align}  \label{bp}
\begin{split}
\frac{d}{dt} V_{\chi_R}[\psi] &\leq 8 E(\psi) -\frac{4 \mu(N(q-2)+2b-4)-\eps\mu(N(q-2)+2b)}{2q}\int_{\R^N} |x|^{-b} |\psi|^{q} \, dx \\
& \quad -\frac{4(N(p-2)+2b-4)-\eps(N(p-2)+2b)}{2q}\int_{\R^N} |x|^{-b} |\psi|^{p} \, dx + \eps(E(\psi)-\gamma(c)) +\eps.
\end{split}
\end{align}
Note that $Q(\psi_{t_{\psi}})=0$ and $0<t_{\psi} <1$, then
\begin{align*} 
E(\psi) <\gamma(c) \leq E(\psi_{t_{\eps}})&=E(\psi_{t_{\psi}})-\frac 12 Q(\psi_{t_{\psi}})\\
& <\frac{\mu(N(q-2)+2b-4)}{4q}\int_{\R^N} |x|^{-b} |\psi|^{q} \, dx +\frac{N(p-2)+2b-4}{4p}\int_{\R^N} |x|^{-b} |\psi|^{p} \, dx \\
&=\frac{4\mu(N(q-2)+2b-4)-\eps \mu (N(q-2)+2b)}{16 q}\int_{\R^N} |x|^{-b} |\psi|^{q} \, dx \\
& \quad +  \frac{4(N(p-2)+2b-4)-\eps(N(p-2)+2b)}{16 p}\int_{\R^N} |x|^{-b} |\psi|^{p} \, dx \\
& \quad + \frac{\eps \mu (N(q-2)+2b)}{16 q}\int_{\R^N} |x|^{-b} |\psi|^{q} \, dx  +\frac{\eps(N(p-2)+2b)}{16 p}\int_{\R^N} |x|^{-b} |\psi|^{q} \, dx.
\end{align*}
In addition, by the conservation of energy, there holds that $E(\psi)=E(\psi_0)<(1-\delta) \gamma(c)$ for some $\delta_0>0$. If there exists a constant $C>0$ such that, for any $t>0$,
\begin{align} \label{blowup1}
\frac{\mu(N(q-2)+2b)}{16 q}\int_{\R^N} |x|^{-b} |\psi|^{q} \, dx  +\frac{N(p-2)+2b}{16 p}\int_{\R^N} |x|^{-b} |\psi|^{q} \, dx \leq C,
\end{align}
then, from \eqref{bp}, we can obtain that, for any $\eps>0$ small enough,
$$
\frac{d}{dt} V_{\chi_R}[\psi] \leq -\frac{\gamma(c)}{2} \delta.
$$
If \eqref{blowup1} does not hold, then there exist a sequence $\{t_n\} \subset \R^+$ such that 
\begin{align*}
\frac{\mu(N(p-2)+2b)}{16 q}\int_{\R^N} |x|^{-b} |\psi(t_n)|^{q} \, dx  +\frac{N(p-2)+2b}{16 p}\int_{\R^N} |x|^{-b} |\psi(t_n)|^{q} \, dx \to \infty \,\, \mbox{as} \,\, n \to \infty.
\end{align*}
In this case, from \eqref{bp}, then there holds that, for any $\eps>0$ small enough,
$$
\frac{d}{dt} V_{\chi_R}[\psi(t_n)] \to -\infty \,\, \mbox{as} \,\, n \to \infty.
$$
Consequently, in both cases, we can conclude that $\psi$ cannot exist globally in time, i.e. $T<\infty$. Thus the proof is completed.
\end{proof}

\begin{proof}[Proof of Corollary \ref{instability}]
Since $u \in S(c)$ is a ground state to \eqref{equ}-\eqref{mass} at the level $\gamma(c)$, then $u_{\tau} \in \mathcal{K}^-(c)$ for any $\tau>1$ close to 1. Note that $\|u_{\tau} -u\| \to 0$ in $H^1(\R^N)$ as $\tau \to 1$. Therefore, from Theorem \ref{blowup}, we have the desired result. Thus the proof is completed. 
\end{proof}


\begin{thebibliography}{10}  

\bibitem{AB} {J. Albert, S. Bhattarai:} \textit{Existence and stability of a two-parameter family of solitary waves for an NLS-KdV system},  Adv. Differential Equations 18 (11-12) (2013) 1129-1164.

\bibitem{AC} {A.H. Ardila, M. Cardoso:} {\it Blow-up solutions and strong instability of ground states for the inhomogeneous nonlinear Schr\"odinger equation},  Commun. Pure Appl. Anal. 20 (1) (2021) 101-119.

\bibitem{AIKN} {T. Akahori, S. Ibrahim, H. Kikuchi, H. Nawa:} {\it Global dynamics above the ground state energy for the combined power-type nonlinear Schr\"odinger equations with energy-critical growth at low frequencies}, Mem. Amer. Math. Soc. 272 (1331) (2021) v+130 pp. 

\bibitem{BS1} {T. Bartsch, N. Soave:} {\it A natural constraint approach to normalized solutions of nonlinear Schr\"odinger equations and systems}, J. Funct. Anal. 272 (12) (2017) 4998-5037.

\bibitem{BS2} {T. Bartsch, N. Soave:} {\it Multiple normalized solutions for a competing system of Schr\"odinger equations}, Calc. Var. Partial Differential Equations 58 (1) (2019), Paper No. 22, 24 pp.

\bibitem{BMRV} T{. Bartsch, R. Molle, M. Rizzi, G. Verzini:} \textit{Normalized solutions of mass supercritical Schr\"odinger equations with potential}, Comm. Partial Differential Equations 46 (9) (2021) 1729-1756.

\bibitem{BJS} {T. Bartsch, L. Jeanjean, N. Soave:} \textit{Normalized solutions for a system of coupled cubic Schr\"odinger equations on $\R^3$}, J. Math. Pures Appl. 106 (2016) 583-614.

\bibitem{BV} {T. Bartsch, S. Valeriola:} \textit{Normalized solutions of nonlinear Schr\"odinger equations}, Arch. Math. 100 (2013) 75-83.

\bibitem{BWW} {T. Bartsch,  T. Weth, Tobias , M. Willem:} {\it Partial symmetry of least energy nodal solutions to some variational problems}, J. Anal. Math. 96 (2005) 1-18.

\bibitem{BZZ} {T. Bartsch, X. Zhong, W. Zou:} \textit{Normalized solutions for a coupled Schr\"odinger system}, Math. Ann. 380 (3-4) (2021) 1713-1740.

\bibitem{BFG} {J. Bellazzini, L. Forcella, V. Georgiev:} {\it Ground state energy threshold and blow-up for NLS with competing nonlinearities}, arXiv:2012.10977.

\bibitem{BJ} {J. Bellazzini, L. Jeanjean:} \textit{On dipolar quantum gases in the unstable regime}, SIAM J. Math. Anal. 48 (2016) 2028-2058.

\bibitem{BJT} {J. Bellazzini, L. Jeanjean, T. Luo:} \textit{Existence and instability of standing waves with prescribed norm for a class of Schr\"odinger-Poisson equations}, Proc. Lond. Math. Soc. 107 (2013) 303-339.



{\bibitem{BPVT} {J. Belmonte-Beitia, V.M. P\'erez-Garc\'ia, V. Vekslerchik, P.J. Torres:} {\it Lie symmetries and solitons in nonlinear systems with spatially inhomogeneous nonlinearities}, Phys. Rev. Lett. 98 (6) (2007) 064102.}

\bibitem{BCGJ} {D. Bonheure, J.-B. Casteras, T. Gou, L. Jeanjean:} \textit{Normalized solutions to the mixed dispersion nonlinear Schr\"odinger equation in the mass critical and supercritical regime}, Trans. Amer. Math. Soc. 372 (3) (2019) 2167-2212.

\bibitem{BS} {F. Brock, A.Yu. Solynin:} {\it An approach to symmetrization via polarization}, Trans. Amer. Math. Soc. 352 (4) (2000) 1759-1796.

\bibitem{BZ} {J.E. Brothers, W.P. Ziemer:} {\it Minimal rearrangements of Sobolev functions}, J. Reine Angew. Math. 384 (1988) 153-179.

{\bibitem{BeLi} {H. Berestycki, P.-L. Lions:} {\it Nonlinear scalar field equations. II. Existence of a ground state}, Arch. Ration. Mech. Anal. 82 (4) (1983) 313-345.}

\bibitem{C1} {L. Campos:} {\it Scattering of radial solutions to the inhomogeneous nonlinear Schr\"odinger equation}, Nonlinear Anal. 202 (2021), Paper No. 112118, 17 pp. 

\bibitem{CC1} {L. Campos, M. Cardoso:} {\it A Virial-Morawetz approach to scattering for the non-radial inhomogeneous NLS}, Proc. Amer. Math. Soc. 150 (5) (2022)  2007-2021.

\bibitem{CC2} {L. Campos, M. Cardoso:} {\it Blowup and scattering criteria above the threshold for the focusing inhomogeneous nonlinear Schr\"odinger equation}, NoDEA Nonlinear Differential Equations Appl. 28 (6) (2021), Paper No. 69, 33 pp.

\bibitem{CCW} {D. Cao, I.-L. Chern, J.-C. Wei:} \textit{On ground state of spinor Bose-Einstein condensates},  NoDEA Nonlinear Differential Equations Appl. 18 (2011) 427-445.

\bibitem{CDSS} {I. Catto, J. Dolbeault, O. S\'anchez, J. Soler:}\textit{Existence of steady states for the Maxwell-Schr\"odinger-Poisson system: exploring the applicability of the concentration-compactness principle,} Math. Models Methods Appl. Sci. 23 (2013) 1915-1938.

\bibitem{Ca} {T. Cazenave:} {\it Semilinear Schr\"odinger Equations}, Courant Lecture Notes in Mathematics, AMS, 2003.

\bibitem{CP} {T. Cazenave, P.-L. Lions:} \textit{Orbital stability of standing waves for some nonlinear Schr\"odinger equations}, Commun. Math. Phys. 85 (1982) 549-561.

{\bibitem{CLO} {W.-X. Chen, C.-M. Li, B. Ou:} \textit{Classification of solutions for an integral equation}, Commun. Pure Appl. Math. 59 (3) (2006) 330–343.}

\bibitem{C} {X. Cheng:} {\it Scattering for the mass super-critical perturbations of the mass critical nonlinear Schr\"odinger equations}, Illinois J. Math. 64 (1) (2020) 21-48. 

\bibitem{CMZ} {X. Cheng, C. Miao, L. Zhao:} {\it Global well-posedness and scattering for nonlinear Schr\"odinger equations with combined nonlinearities in the radial case}, J. Differential Equations 261 (6) (2016) 2881-2934.

\bibitem{CJ} {S. Cingolani, L. Jeanjean:} \textit{Stationary waves with prescribed $L^2$-norm for the planar Schr\"odinger-Poisson system}, SIAM J. Math. Anal. 51 (4) (2019) 3533-3568.

\bibitem{DS} {E. Dancer, S. Santra:} {\it Singular perturbed problems in the zero mass case: asymptotic behavior of spikes},
Ann. Mat. Pura Appl. 189 (2) (2010) 185-225.

\bibitem{DSW} {E. Dancer, S. Santra, J. Wei:} {\it Asymptotic behavior of the least energy solution of a problem with competing powers}, J. Funct. Anal. 261 (8) (2011) 2094-2134.

\bibitem{DBF} {A. De Bouard, R. Fukuizumi:} {\it Stability of standing waves for nonlinear Schr\"odinger equations with inhomogeneous nonlinearities}, Ann. Henri Poincar\'e 6 (2005) (6) 1157-1177.

\bibitem{D1} {V.D. Dinh:} {\it Blowup of $H^1$ solutions for a class of the focusing inhomogeneous nonlinear Schro\"dinger equation}, Nonlinear Anal. 174 (2018), pp. 169-188.

\bibitem{D2} {V. D. Dinh:} {\it Energy scattering for a class of the defocusing inhomogeneous nonlinear Schro\"dinger equation}, J. Evol. Equ. 19(2) (2019) 411-434, 

\bibitem{D3}{V. D. Dinh:} {\it Energy scattering for a class of inhomogeneous nonlinear Schr\"odinger equation in two dimensions}, J. Hyperbolic Differ. Equ. 18 (1) (2021) 1-28. 

\bibitem{DK1} {V.D. Dinh, S. Keraani:} {\it Long time dynamics of nonradial solutions to inhomogeneous nonlinear Schr\"odinger equations}, SIAM J. Math. Anal. 53 (4) (2021) 4765-4811.

\bibitem{DK2} {V.D. Dinh, S. Keraani:} {\it A compactness result for inhomogeneous nonlinear Schr\"odinger equations}, Nonlinear Anal. 215 (2022), Paper No. 112617, 29 pp. 

\bibitem{EGBB} {B.D. Esry, C.H. Greene, Jr., J. P. Burke, J.L.  Bohn:} {\it Hartree-Fock theory for double condensates}, Phys. Rev. Lett. 78 (1997) 3594-3597.

\bibitem{F} {L.G. Farah:} {\it Global well-posedness and blow-up on the energy space for the inhomogeneous nonlinear Schro\"dinger equation}, J. Evol. Equ. 16 (1) (2016) 193-208.

\bibitem{FG1} {L.G. Farah, C. M. Guzm\'an:} {\it Scattering for the radial 3D cubic focusing inhomogeneous nonlinear Schr\"odinger equation}, J. Differential Equations  262 (8) (2017) 4175-4231.

\bibitem{FG2}{L.G. Farah, C. M. Guzm\'an:} {\it Scattering for the radial focusing inhomogeneous NLS equation in higher dimensions}, Bull. Braz. Math. Soc. (N.S.) 51 (2) (2020) 449-512.

\bibitem{F} {D.J. Frantzeskakis:} {\it Dark solitons in atomic Bose Einstein condensates: From theory to experiments}, J. Phys. A: Math. Theor. 43, 213001 (2010).

\bibitem{FH} {N. Fukaya, M. Hayashi:} {\it Instability of algebraic standing waves for nonlinear Schr\"odinger equations with double power nonlinearities}, Trans. Amer. Math. Soc. 374 (2) (2021) 1421-1447.

\bibitem{FO} {N. Fukaya, M. Ohta:} {\it Strong instability of standing waves with negative energy for double power nonlinear Schr\"odinger equations}, SUT J. Math. 54 (2) (2018) 131-143.

\bibitem{GZ} {C. Gao, Z. Zhao:} {\it On scattering for the defocusing high dimensional inter-critical NLS}, J. Differential Equations 267 (11) (2019) 6198-6215.

\bibitem{G} D. Garrisi, \textit{On the orbital stability of standing-wave solutions to a coupled non-linear Klein-Gordon equation}, Adv. Nonlinear Stud. 12 (3) (2012) 639-658.

\bibitem{G} {F. Genoud:} {\it A uniqueness result for $\Delta u -\lambda u + V (| x|)u^p = 0$ on $\R^2$}, Adv. Nonlinear Stud. 11 (3) (2011) 483-491.

{\bibitem{G1} {F. Genoud:} {\it An inhomogeneous, $L^2$-critical, nonlinear Schr\"odinger equation}, Z. Anal. Anwend. 31 (3) (2012) 283-290.}

\bibitem{GS} {F. Genoud, C.A. Stuart:} {\it Schr\"odinger equations with a spatially decaying nonlinearity: Existence and stability of standing waves}, Discrete Contin. Dyn. Syst. 21 (1) (2008) 137-186.

\bibitem{Gh} {N. Ghoussoub:} {\it Duality and Perturbation Methods in Critical Point Theory,} Cambaridge University Press,
Cambaridge (1993). 

\bibitem{Gou} {T. Gou:} \textit{Existence and orbital stability of standing waves to nonlinear Schr\"odinger system with partial confinement}, J. Math. Phys. 59 (7) (2018) 071508, 12 pp.

\bibitem{GJ2} {T. Gou, L. Jeanjean:} \textit{Existence and orbital stability of standing waves for nonlinear Schr\"odinger systems}, Nonlinear Anal. 144 (2016) 10-22.

\bibitem{GJ1} {T. Gou, L. Jeanjean:} \textit{Multiple positive normalized solutions for nonlinear Schr\"odinger systems}, Nonlinearity 31 (5) (2018) 2319-2345.

\bibitem{GZ} {T. Gou, Z. Zhang:} \textit{Normalized solutions to the Chern-Simons-Schr\"odinger system}, J. Funct. Anal. 280 (5) (2021), Paper No. 108894, 65 pp.

\bibitem{Je} {L. Jeanjean:} \textit{Existence of solutions with prescribed norm for semilinear elliptic equations}, Nonlinear Anal. 28 (1997) 1633-1659.

\bibitem{JL1} {L. Jeanjean, T.T. Le:} {\it Multiple normalized solutions for a Sobolev critical Schr\"odinger-Poisson-Slater equation}, J. Differential Equations 303 (2021)v277-325.

{\bibitem{JL} {L. Jeanjean, T.T. Le:} {\it Multiple normalized solutions for a Sobolev critical Schr\"odinger equation}, Math. Ann. 384 (1-2) (2022) 101-134.}

\bibitem{JS} L. Jeanjean,  S.-S. Lu, \textit{A mass supercritical problem revisited}, Calc. Var. Partial Differential Equations 59 (5) (2020), Paper No. 174, 43 pp.

\bibitem{KM1} {C.E. Kenig, F. Merle:} {\it Global well-posedness, scattering and blow-up for the energy critical, focusing, nonlinear Schr\"odinger equation in the radial case}, Invent. Math. 166 (3) (2006) 645-675.

\bibitem{KM2} {C.E. Kenig, F. Merle:} {\it Global well-posedness, scattering and blow-up for the energy critical, focusing, nonlinear wave equation}, Acta Math. 201 (2) (2008) 147-212.

{\bibitem{KA} {Y.S. Kivshar, G. P. Agrawal:} {\it Optical Solitons: From Fibers to Photonic Crystals}, Academic Press (2003).}

\bibitem{LMR} {S. Le Coz, Y. Martel, P. Rapha\"el:} {\it Minimal mass blow up solutions for a double power nonlinear Schr\"dinger equation}, Rev. Mat. Iberoam. 32 (3) (2016) 795-833.

\bibitem{LL} {E.H. Lieb, M. Loss:} {\it Analysis}, American Mathematical Society, 2001.

\bibitem{Li1} {P-L. Lions:} {\it The concentration-compactness principle in the calculus of variations. The locally com-
pact case, part I}, Ann. Inst. Henri Poincar\'e, Anal. Non Lin\'eaire 1 (1984) 109-145.

\bibitem{Li2} {P-L. Lions:} {\it The concentration-compactness principle in the calculus of variations. The locally com-
pact case, part II}, Ann. Inst. Henri Poincar\'e, Anal. Non Lin\'eaire 1 (1984) 223-283.

\bibitem{LWW} {Y. Liu, X.-P. Wang, K. Wang:} {\it Instability of standing waves of the Schr\"odinger equation with inhomogeneous nonlinearity}, Trans. Amer. Math. Soc. 358 (5) (2006) 2105-2122.

\bibitem{Luo} {Y. Luo:} {\it Sharp scattering threshold for the cubic-quintic nonlinear Schr\"odinger equation in the focusing-focusing regime}, J. Funct. Anal. 283 (1) (2022), Paper No. 109489, 34 pp.

\bibitem{M} {B. Malomed:} {\it Multi-Component Bose-Einstein Condensates: Theory, Emergent Nonlinear Phenomena in Bose-Einstein Condensation}, Springer-Verlag, Berlin, 2008 pp. 287-305.

\bibitem{Me} {F. Merle:} {\it Determination of blow-up solutions with minimal mass for nonlinear Schr\"odinger equations with critical power}, Duke Math. J. 69 (2) (1993) 427-454.

\bibitem{MXZ} {C. Miao, G. Xu, L. Zhao:} {\it The dynamics of the 3D radial NLS with the combined terms}, Comm. Math. Phys. 318 (3) (2013) 767-808.

\bibitem{MZZ} {C. Miao, T. Zhao, J. Zheng:} {\it On the 4D nonlinear Schr\"odinger equation with combined terms under the energy threshold}, Calc. Var. Partial Differential Equations 56 (6) (2017), Paper No.179, 39 pp.

\bibitem{NW1} {N.V. Nguyen, Z.-Q. Wang:} \textit{Orbital stability of solitary waves for a nonlinear Schr\"odinger system},
 Adv. Differential Equations 16 (2011) 977-1000.

\bibitem{NW2} {N.V. Nguyen, Z.-Q. Wang:} \textit{Orbital stability of solitary waves of a 3-coupled nonlinear Schr\"odinger system}, Nonlinear Anal. 90 (2013) 1-26.

\bibitem{NW3} {N.V. Nguyen, Z.-Q. Wang:} \textit{Existence and stability of a two-parameter family of solitary waves for
a 2-couple nonlinear Schr\"odinger system}, Discrete Contin. Dyn. Syst. 36 (2016) 1005-1021.

\bibitem{NTV1} {B. Noris, H. Tavares, G. Verzini:} \textit{Existence and orbital stability of the ground states with prescribed mass for the $L^2$-critical and supercritical NLS on bounded domains}, Anal. PDE 7 (8) (2014) 1807-1838.

\bibitem{NTV2} {B. Noris, H. Tavares, G. Verzini:} \textit{Stable solitary waves with prescribed $L^2$-mass for the cubic Schr\"odinger system with trapping potentials}, Discrete Contin. Dyn. Syst. 35 (12) (2015) 6085-6112.

\bibitem{NTV3} {B. Noris, H. Tavares, G. Verzini:} \textit{Normalized solutions for nonlinear Schr\"odinger systems on bounded domains},  Nonlinearity 32 (3) (2019) 1044-1072.

\bibitem{PPVV} {B. Pellacci, A. Pistoia, G. Vaira, G. Verzini:} \textit{Normalized concentrating solutions to nonlinear elliptic problems}, J. Differential Equations 275 (2021) 882-919.

\bibitem{PG} {D. Pierotti, G. Verzini:} \textit{Normalized bound states for the nonlinear Schr\"odinger equation in bounded domains}, Calc. Var. Partial Differential Equations 56 (5) (2017), Paper No. 133, 27 pp.

\bibitem{RS} {P. Rapha\"el, J. Szeftel:} {\it Existence and uniqueness of minimal blow-up solutions to an inhomogeneous mass critical NLS}, J. Amer. Math. Soc. 24 (2) (2011) 471-546.

\bibitem{S1} {N. Soave:} {\it Normalized ground states for the NLS equation with combined nonlinearities}, J. Differential Equations 269 (9) (2020) 6941-6987.

\bibitem{S2} {N. Soave:} {\it Normalized ground states for the NLS equation with combined nonlinearities: the Sobolev critical case}, J. Funct. Anal. 279 (6) (2020), 108610, 43 pp. 


\bibitem{TVZ} {T. Tao, M. Visan, X. Zhang:} {\it The nonlinear Schr\"odinger equation with combined power-type nonlinearities}, Comm. Partial Differential Equations 32 (7-9) (2007) 1281-1343.

\bibitem{T} {J.F. Toland:} {\it Uniqueness of positive solutions of some semilinear Sturm-Liouville problems on the half line}, Proc. Roy. Soc. Edinburgh Sect. A, 97 (1984) 259-263.

\bibitem{Sch} {J. Van Schaftingen:} {\it Explicit approximation of the symmetric rearrangement by polarizations},
Arch. Math. 93 (2) (2009) 181-190.

\bibitem{X} {J. Xie:} {\it Scattering for focusing combined power-type NLS}, Acta Math. Sin. 30 (5) (2014) 805-826.

\bibitem{Y} {E. Yanagida:} {\it Uniqueness of positive radial solutions of $\Delta u + g(r)u + h(r)u^p = 0$ in $\R^n$}, Arch.
Rational Mech. Anal. 115 (3) (1991) 257-274.



\end{thebibliography}
\end{document}